\documentclass[a4paper, 12pt]{article}

\usepackage{amssymb,amsmath,latexsym,graphicx,stmaryrd}

\usepackage[utf8]{inputenc}  
\usepackage[T1]{fontenc} 

\usepackage{mathrsfs}  

\usepackage{titlesec,titleps,titletoc}
\usepackage{etoolbox}
\usepackage[normalem]{ ulem }
\usepackage{soul}

\makeatletter
\patchcmd{\ttlh@hang}{\parindent\z@}{\parindent\z@\leavevmode}{}{}
\patchcmd{\ttlh@hang}{\noindent}{}{}{}
\makeatother

\usepackage[pdfborder=0,unicode=0]{hyperref}
 
\usepackage[svgnames,dvipsnames]{xcolor} 
\usepackage{lmodern}
\usepackage{lipsum}
\usepackage{ifpdf}

\newtheorem{dref}{Definition}[section] \newtheorem{lemma}[dref]{Lemma}
\newtheorem{theo}[dref]{Theorem} \newtheorem{prop}[dref]{Proposition}

\newenvironment{proof}{\par\noindent{{\bf Proof.}}}{\hfill$\Box$
\medskip}

\title{Resonances over a potential well in an island.}

\author{\\Johannes Sj\"ostrand\\ 
\small IMB {\footnotesize - UMR5584 CNRS},\\
\small  Universit\'e de Bourgogne\\
  \small 9, avenue Alain Savary - BP 47870\\
  \small 21078 Dijon cedex, France\\ \footnotesize
  johannes.sjostrand@u-bourgogne.fr \\
  \and 
 \\ Maher Zerzeri
  \\ \small LAGA {\footnotesize - UMR7539 CNRS},\\
\small  Universit\'e Sorbonne Paris-Nord\\
  \small 99, avenue J.-B. Cl\'ement\\
  \small 93430 Villetaneuse, France\\ \footnotesize
  zerzeri@math.univ-paris13.fr
}

\date{}

\begin{document}

\maketitle

 \par\bigskip

\abstract{In this paper we study the distribution of
    scattering resonances for a multidimensional semi-classical Schr\"odinger operator,
    associated to a potential well in an island at energies close to the maximal one that limits the separation of the
    well and the surrounding sea. 
    
 \par\bigskip\noindent
    
\begin{center}
 \textbf{R\'esum\'e}
\end{center}

Dans cet article, nous \'etudions la distribution des r\'esonances pour 
l'op\'erateur de Schr\"odinger semi-classique multidimensionnel,
associ\'e \`a un puits de potentiel dans l'\^ile aux \'energies
proche de celle qui d\'elimite la s\'eparation du puits et de la mer environnante.
}

\vskip2.5cm\noindent
{\small {\bf 2020 Mathematics Subject Classification.--} 35J10 35B34 35P20 47A55.} 
\par\smallskip\noindent
{\small {\bf Key words and phrases.--} Resonances, semi-classical asymptotic, 
microlocal analysis, Schr\"odinger operator,
potential well, maximal energy.}

\tableofcontents

\vfill\newpage


\section{Introduction and main result}\label{int}
\setcounter{equation}{0}

In this work we consider resonances for a semi-classical Schr\"odinger
operator with potential $V\in C^\infty ({\mathbb{R}}^n;{\mathbb{R}})$ (cf.\
(\ref{int.6})), where we assume that $n\ge 2$ and that
\begin{equation}\label{int.1}
\begin{matrix}\left\{
\begin{aligned}
&V\hbox{ has a holomorphic extension to a  truncated}\\
& \hbox{  sector }\,\, {\Gamma}_{C}:=\Big\{x\in {\mathbb{C}}^n;\, |\Re x|>C,\ 
|\Im x|<\frac{1}{C}|\Re x|\Big\}.
\end{aligned}\right.
\end{matrix}
\end{equation}
Here $C$ is some positive constant. We let $V$ also denote the extension. Assume
\begin{equation}\label{int.2}
V\longrightarrow 0,\,\, \hbox{ when }\, x\longrightarrow \infty\,\, \hbox{ in }\, {\Gamma}_C. 
\end{equation}
Let $E_0>0$. Assume that $V^{-1}(]-\infty ,E_0[)={\mathscr{U}}_{E_0}\sqcup\mathscr{S}_{E_0}$,
where ${\mathscr{U}}_{E_0}$, $\mathscr{S}_{E_0}$ are open connected and mutually disjoint. We let
${\mathscr{U}}_{E_0}$ (the potential well) be the bounded component and $\mathscr{S}_{E_0}$ (the sea)
be the unbounded one. When
\begin{equation}\label{int'.4}
\overline{\mathscr{U}}_{\hskip-1pt E_0}\cap \overline{\mathscr S}_{\hskip-1pt E_0}=\emptyset,
\end{equation}
the situation is quite well understood  (see \cite{HeSj86} and also  
\cite{CoDuKlSe87}): Let
\begin{equation}\label{int.4.5}
p(x,\xi )=\xi^2+V(x)
\end{equation}
and assume that
\begin{equation*}
V\hbox{\ is analytic in a neighborhood of }\overline{\mathscr S}_{\hskip-1pt E_0},
\end{equation*}
\begin{equation*}
  H_p=\partial_\xi p \cdot \partial_x- \partial_x p\cdot \partial _\xi \hbox{ has no
    trapped trajectories in }
  {{p^{-1}(E_0)}_{\vert}}_{\overline{\mathscr S}_{\hskip-1pt E_0}}.
\end{equation*}
Here a trapped trajectory is by definition a maximally extended
integral curve of $H_p$ which is contained in some bounded set. 

\par\smallskip
By suitably modifying the potential near $\overline{\mathscr S}_{\hskip-1pt E_0}$, we get a
new potential $V^{\mathrm{int}}$ which is equal to $V$ in a neighborhood of
$\overline{\mathscr{U}}_{E_0}$ and $\ge E_0+ \frac{1}{{\mathcal{O}}(1)}$\footnote{Here we follow the
  convention that the expression ``${\mathcal{O}}(1)$''  in a denominator
  denotes a bounded positive quantity.} away from that neighborhood,
so that $P^{\mathrm{int}}=-h^2\Delta +V^{\mathrm{int}}$ is self-adjoint
with purely discrete spectrum in $]-\infty ,E_0+\frac{1}{{\mathcal{O}}(1)}[$. The
eigenvalues are distributed according to the semi-classical Weyl law
and it was established in \cite[{Proposition 9.6 and Theorem 9.7}]{HeSj86}, \cite[{ Theorem 4 in Section IV or Theorem 2 in Section V}]{CoDuKlSe87} that the resonances of
\begin{equation}\label{int.6}
P=-h^2\Delta +V
\end{equation}
in $\mathrm{neigh\,}({E_0},{\mathbb{C}})$\footnote{Let $M$ be a topological space. 
Let $N$ be a subset of $M$. The set $\textrm{neigh}(N,M)$  denotes some neighborhood of $N$ in $M$.} 
are related to the eigenvalues of
$P^{\mathrm{int}}$ in $\mathrm{neigh\,}({E_0},{\mathbb{R}})$ via a bijection
$$
b:\, \sigma (P^{\mathrm{int}})\cap \mathrm{neigh\,}({E_0},{\mathbb{R}})\longrightarrow
\mathrm{Res\,}(P)\cap \mathrm{neigh\,}({E_0},{\mathbb{C}}),
$$
such that $b(\mu )-\mu ={\mathcal{O}}(1)e^{-\frac{1}{{\mathcal{O}}(h)}}$. Here $\sigma
(P^{\mathrm{int}})$ denotes the spectrum of $P^{\mathrm{int}}$ and
$\mathrm{Res\,}(P)$ the set of resonances, where both the eigenvalues
and the resonances are counted with their natural multiplicity. 
{See \cite{Ma02}, \cite{FuLaMa11}, \cite{HiMaSj17}, \cite{NaStZw03} for related results for potentials that may be non-analytic on any bounded set.}

When increasing the energy level, 
we may have
$$
V^{-1}(]-\infty ,E[)={\mathscr{U}}_E\sqcup {\mathscr{S}}_E
$$
for $E-E_0>0$ small, where ${\mathscr{U}}_E\supset {\mathscr{U}}_{E_0}$, 
${\mathscr{S}}_E\supset \mathscr{S}_{E_0}$ remain
connected and disjoint until we reach a new energy $E_0=E_0^{\mathrm{new}}$, for which
(\ref{int'.4}) no longer holds. We typically may have
\begin{equation}\label{int'.7}
\overline{\mathscr U}_{E_0}\cap \overline{\mathscr S}_{\hskip-1pt E_0}=\{x_0 \}, 
\end{equation}
for some point $x_0\in {\mathbb{R}}^n$ while the other assumptions remain
valid. In this work we study the distribution of resonances near the
new energy level $E_0>0$.

For simplicity we now replace $V$ by $V-E_0$ (so that the energy level
$E_0$ transforms to the level $0$) and formulate our
assumptions for the new potential $V\in C^\infty ({\mathbb{R}}^n;{\mathbb{R}})$,
($n>1$) assumed to satisfy \eqref{int.1}. Instead of \eqref{int.2} we assume 
\begin{equation}\label{int.2new}
V\longrightarrow -E_0,\,\, \hbox{ when }\, x\longrightarrow \infty\,\, \hbox{ in }\, {\Gamma}_{C}. 
\end{equation}
Assume that
\begin{equation}\label{int.3}
V^{-1}(]-\infty ,0[)={\mathscr{U}}_0\sqcup {\mathscr{S}}_0,
\end{equation}
where ${\mathscr{U}}_0$, ${\mathscr{S}}_0$ are open connected (and mutually disjoint). We let
${\mathscr{U}}_0$ (the potential well) be the bounded component and ${\mathscr{S}}_0$ (the sea)
be the unbounded one. Define $p(x,\xi )$ as in (\ref{int.4.5}).
Assume that
\begin{equation}\label{int.5}
V\hbox{\ is analytic in a neighborhood of }\overline{\mathscr{S}}_0.
\end{equation}
Assume that
\begin{equation}\label{int.7}
\overline{\mathscr{U}}_0\cap \overline{\mathscr{S}}_0=\{x_0 \}, 
\end{equation}
for some point $x_0\in {\mathbb{R}}^n$. After a translation, we may assume that $x_0=0$.\\

\begin{figure}
\hskip2cm
\begin{picture}(0,0)%
\includegraphics[scale=0.3]{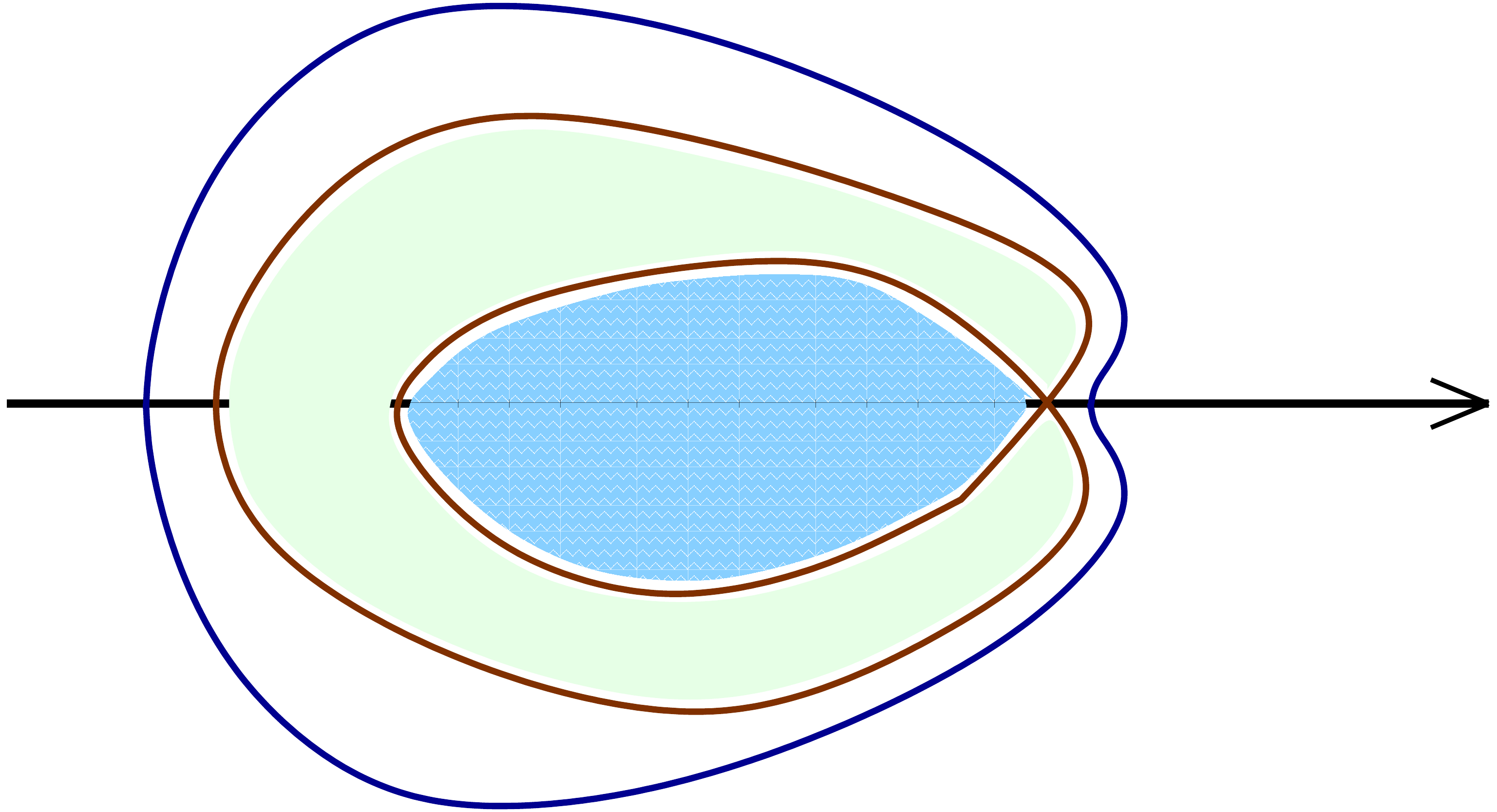}%
\end{picture}%
\setlength{\unitlength}{789sp}%
\begingroup\makeatletter\ifx\SetFigFont\undefined%
\gdef\SetFigFont#1#2#3#4#5{%
  \reset@font\fontsize{#1}{#2pt}%
  \fontfamily{#3}\fontseries{#4}\fontshape{#5}%
  \selectfont}%
\fi\endgroup%
\begin{picture}(12302,8578)(-58,-8209)

\put(10900,-2100){\scalebox{3.5}[1]{\color{Blue}\circle{2000}}}

\put(600,6000){\makebox(0,0)[lb]{\smash{{\SetFigFont{15}{11}{\rmdefault}{\mddefault}{\updefault}$\mathbb{R}^n$}}}}
\put(-1000,-300){\makebox(0,0)[lb]{\smash{{\SetFigFont{19}{11}{\rmdefault}{\mddefault}{\updefault}{\color{Blue}${\mathscr{S}}_0$}}}}}
\put(14350,-1700){\makebox(0,0)[lb]{\smash{{\SetFigFont{15}{9}{\rmdefault}{\mddefault}{\updefault}{\color{Brown}$0$}}}}}

\put(22000,-3000){\makebox(0,0)[lb]{\smash{{\SetFigFont{15}{7.2}{\rmdefault}{\mddefault}{\updefault}{$x_{n}$}}}}}

\put(7900,1200){\makebox(0,0)[lb]{\smash{{\SetFigFont{50}{7.2}{\rmdefault}{\mddefault}{\updefault}{\color{PineGreen}{$\swarrow$}}}}}}
\put(11750,4000){\makebox(0,0)[lb]{\smash{{\SetFigFont{12}{7.2}{\rmdefault}{\mddefault}{\updefault}{\color{PineGreen}{$\{V(x)>0\}$}}}}}}

\put(11500,-1100){\makebox(0,0)[lb]{\smash{{\SetFigFont{55}{7.2}{\rmdefault}{\mddefault}{\updefault}{\color{blue}$\swarrow$}}}}}
\put(15700,1950){\makebox(0,0)[lb]{\smash{{\SetFigFont{19}{7.2}{\rmdefault}{\mddefault}{\updefault}{\color{blue}${\mathscr{U}}_{0}$}}}}}

\put(13750,-4800){\makebox(0,0)[lb]{\smash{{\SetFigFont{40}{7.2}{\rmdefault}{\mddefault}{\updefault}{\color{blue}$\nwarrow$}}}}}
\put(15150,-5850){\makebox(0,0)[lb]{\smash{{\SetFigFont{22}{8}{\rmdefault}{\mddefault}{\updefault}{\color{blue}$\leftarrow$}}}}}
\put(16850,-5550){\makebox(0,0)[lb]{\smash{{\SetFigFont{12}{7.2}{\rmdefault}{\mddefault}{\updefault}{\color{blue}$\{V(x)=E_1<0\}$}}}}}

\put(10500,-8500){\makebox(0,0)[lb]{\smash{{\SetFigFont{31}{7.2}{\rmdefault}{\mddefault}{\updefault}{\color{Brown}$\nwarrow$}}}}}
\put(12900,-9100){\makebox(0,0)[lb]{\smash{{\SetFigFont{12}{7.2}{\rmdefault}{\mddefault}{\updefault}{\color{Brown}$\{V(x)=0\}$}}}}}

\put(-2000,-11000){\makebox(0,0)[lb]{\smash{{\SetFigFont{12}{7.2}{\rmdefault}{\mddefault}{\updefault}{\color{blue}$\{V(x)<0\}=\mathscr{U}_0\cup\mathscr{S}_0$}.}}}}

\end{picture}%
\vskip1cm
\caption{View from above. {\textsl{The topographic situation.}}} \label{Island}
\end{figure}

\noindent
We make the generic assumption that
\begin{equation}\label{int.8}
\begin{matrix}\left\{
\begin{aligned}
& x_0=0  \hbox{\ is a non-degenerate critical }\\
& \hbox{point for } V \hbox{ of signature }(n-1,1).
\end{aligned}\right.
\end{matrix}
\end{equation}

\noindent
The point $(0,0)\in {\mathbb{R}}^{2n}$ is a stationary point and hence a trapped
trajectory for the Hamilton flow of $p$. Assume that
\begin{equation}\label{int.9}
\begin{matrix}\left\{
\begin{aligned}
& \{(0,0) \}  \hbox{ is the only trapped trajectory}\\
 & \hbox{ for the } H_p\hbox{-flow in} \ {{p^{-1}(0)}_\vert}_{\overline{\mathscr{S}}_0}.
\end{aligned}\right.
\end{matrix}
\end{equation}

We have $V(0)=0$. Assume for simplicity that
\begin{equation}\label{int.10}
dV(x)\ne 0,\hbox{ when } x\in \partial {\mathscr{U}}_0\setminus\{0\}.
\end{equation}
As in \cite{HeSj86} one can define a reference operator $P^{\mathrm{int}}$
by increasing $V$ in ${\mathscr{S}}_0$ to get a potential
$V^{\mathrm{int}}$ which is $\ge 0$ away from ${\mathscr{U}}_0$. 
(Take for instance $V^{\mathrm{int}}=1_{{\mathbb{R}}^n {\setminus} 
{\mathscr{S}}_0}V+1_{{\mathscr{S}}_0}$.)
Then we have the standard Weyl asymptotics for the number of
eigenvalues of $P^{\mathrm{int}}$ in $]-\infty ,E]$ when
\begin{equation}\label{int.11}
-\frac{1}{C}\le E\le -\delta 
\end{equation}
for every fixed $0<\delta \ll 1$ and $C\gg 1$, stating that
\begin{equation}\label{int.12}
\hskip-10pt \#\big(\sigma (P^{\mathrm{int}})\cap ]-\infty ,E] \big)=
  \frac{1}{(2\pi h)^{n}}\left( \mathrm{vol}\big({{p^{-1}(]-\infty
        ,E])}_\vert}_{{\mathscr{U}}_0}\big)+o(1)\right)
\end{equation}
as $ h\longrightarrow 0$ and uniformly for $E$ as in (\ref{int.11}).
\footnote{\label{can.projection} Here for $A\subset T^*\mathbb{R}^n$, $B\subset\mathbb{R}^n$ we write 
$A_{|B}=A\cap \pi_x^{-1}(B)$ where $\pi_x : T^*\mathbb{R}^n\longrightarrow \mathbb{R}^n$ 
is the canonical projection, given by $\pi_x(x,\xi)=x$.}

\par\medskip\noindent
For $E\le 0$, put
\begin{equation}\label{int.13}
\omega (E)=\mathrm{vol\,}\left({{p^{-1}(]-\infty
        ,E])}_\vert}_{{\mathscr{U}}_0}\right),
\end{equation}
so that
\begin{equation}\label{int.14}
\omega (E)=C_n\int_{{\mathscr{U}}_0}\Big(E-V(x)\Big)_+^{\frac{n}{2}} \,dx,
\end{equation}
where $\displaystyle C_n=\mathrm{vol}\big(B_{{\mathbb{R}}^n}(0,1)\big)=
\frac{\pi^{\frac{n}{2}}}{\Gamma(\frac{n}{2}+1)}.$
 Since $n\ge 2$, we see that $\omega \in C^1\big([-\frac{1}{C},0]\big)$ and that
\begin{equation}\label{int.15}
\omega '(E)=\frac{\pi^{\frac{n}{2}}}{\Gamma(\frac{n}{2})}
\int_{{\mathscr{U}}_0}\Big(E-V(x)\Big)_+^{\frac{n}{2}-1} \,dx.
\end{equation}

\par\medskip
Let $\omega $ also denote a $C^1$-extension to the interval
$[-\frac{1}{C},\frac{1}{C}]$ so that (\ref{int.13}) holds for $E\le 0$ and so that
$\omega (E)$ is well-defined up to a term $o(E)$ for $0\le E\le \frac{1}{C}$.

\begin{theo}\label{int1}  
Let $V\in C^\infty(\mathbb{R}^n;\mathbb{R})$ and define $P, p$ as in 
\eqref{int.6}, \eqref{int.4.5}. Let $E_0>0$ and assume 
\eqref{int.1}, \eqref{int.2new}, \eqref{int.3},
\eqref{int.5}, \eqref{int.7}, \eqref{int.8}, \eqref{int.9} and \eqref{int.10}.
Let $\omega (E)$ be a $C^1$ function on
  $\big[-\frac{1}{C},\frac{1}{C}\,\big]$ satisfying (\ref{int.13}) for $E\le 0$.

\par
Let $C_0>0$. Then for every $0<\delta \le \frac12$, there exists $0<\varepsilon
(\delta )\ll 1$ such that for every $0< \varepsilon \le \varepsilon
(\delta )$ and $0<h\le h(\delta,\varepsilon)$ small enough:
\begin{itemize}
\item[(A)] The number of resonances (of $P$) in $]-C_0\varepsilon ,\varepsilon
  [+i]-\varepsilon ,-\delta \varepsilon [$ is ${\mathcal{O}}_\delta (h^{-n}\varepsilon ^n)$.
  
\item[(B)] For all $a,b\in ]-C_0\varepsilon ,\varepsilon [ $ with $a<b$, the
    number of resonances in $]a,b[+i]-\delta \varepsilon ,0]$ is equal to
    $(2\pi h)^{-n}\big(\omega (b)-\omega (a)+{\mathcal{O}}(\delta |\ln \delta|\varepsilon) \big)$,
    uniformly with respect to $a,b,h $.
\end{itemize}

\end{theo}

More precise results are known when $n=1$. In this case the function
$\omega $ has a logarithmic singularity at zero. See \cite{FuRa98} and  \cite{BoFuRaZe14}.

At least formally our result is similar to recent ones about
Helmholtz resonators and other capting devices. See \cite{DuGrMa16}, \cite{Sj01}
and also  \cite{DaJi18}. 

\par\medskip\noindent
The remainder of the paper is devoted to the proof of Theorem \ref{int1}. 
{We shall use suitable escape functions and the corresponding spaces of distributions with exponential phase space weights as developed
in \cite{HeSj86}. The main work will take place near the island and we found it convenient to use the global  framework of \cite{HeSj86}. The control near infinity could also be obtained using complex distorsion techniques. (See \cite{DyZw19} for an overview).}

\par\bigskip\noindent 
{{\it Acknowledgments.} We are grateful to the referee for useful remarks that have led to improvements of the exposition.}

\noindent
{The IMB receives support from the EIPHI Graduate School
(contract ANR-17-EURE-0002).}

\section{Outline }\label{out}
\setcounter{equation}{0}

\subsection{Escape functions (Sections \ref{esc}, \ref{bp})}
After a linear change of coordinates, we may assume that near $x=(x',x_n)=0$
$$
p(x,\xi)=\underbrace{\frac{\kappa }{2}\left(\xi _n^2-x_n^2 \right)+q(x',\xi')}_{=:p_0(x,\xi )}
+{\mathcal{O}}(|x|^3),
$$
where $\kappa >0$ and {$q$ is a positive definite quadratic form}. 
Moreover, {$x_n<0$} in the well and {$x_n>0$} in the
sea. Assume for simplicity
that {$\kappa =1$}. 

\par\smallskip 
Let {$G_0(x,\xi )=x_n\xi _n$}. Then {$H_{p_0}G_0=x_n^2+\xi_n^2$}. {We will
define $G=G^\varepsilon $ near $(0,0)$ as 
a truncation of $G_0$. See (\ref{esc.5})}.  
We show that, with $\rho=(x,\xi)$:
\begin{equation*}
\partial _\rho ^\alpha G(\rho )={\mathcal{O}}(1)\Big(\varepsilon +\rho ^2\Big)^{1-\frac{|\alpha|}{2}},
\quad \alpha \in {\mathbb{N}}^{2n}
\end{equation*}
and that there exists $C\ge 1$ such that in {$\mathrm{neigh\,}(0,\mathbb{R}^{2n})$}:
{$$
\hbox{If }p(\rho)<-\frac{\varepsilon}{C}+\frac{\rho ^2}{C}\hbox{ and }\ x_n(\rho )\ge 0,
\hbox{ then } \, H_pG \asymp \varepsilon +\rho ^2.
$$}
Here we write $X\asymp Y$ for $X,Y\in {\mathbb{R}}$ if $X,\, Y$ have the
same sign (or vanish) and $X={\mathcal{O}}(Y)$ and $Y={\mathcal{O}}(X)$.

\par\medskip
We add a bump at the saddle point in order to create a barrier
between the well and the sea: Let {$\chi (x,\xi )=e^{-(\beta x)^2-(\beta \xi )^2}$}
where {$\beta >0$} is small but fixed. Put
{$$
p_\varepsilon =p+\chi _\varepsilon ,\qquad \chi _\varepsilon (x,\xi )=\varepsilon \chi
\Big(\frac{x}{\sqrt{\varepsilon}},\frac{\xi}{\sqrt{\varepsilon}}\Big).
$$}
Then there exist {$b,c>0$} such that:
{$$
\hbox{If }p_\varepsilon(\rho) <b\varepsilon +c\rho ^2\hbox{ and }\ x_n(\rho )\ge 0,\hbox{
  then } H_{p_\varepsilon }G\asymp \varepsilon +\rho ^2.
$$}
\begin{prop}
Possibly after a dilation in {$\varepsilon $} in the definition of {$G$} we
have for {$0<t\ll 1$} and all {$\rho $} in {$\mathrm{neigh\,}(0,\mathbb{R}^{2n})$}
with {$x_n(\rho )\ge 0$}:
{\[
  \begin{aligned}
&\hbox{If }\Re p(\rho +itH_G(\rho ))\le -\frac{\varepsilon
}{\widetilde{C}}+\frac{\rho ^2}{2C},\hbox{ then }\Im p(\rho
+itH_G(\rho ))\asymp -t(\varepsilon +\rho ^2).\\ 
&\hbox{If }\Re p_\varepsilon (\rho +itH_G(\rho ))\le b\varepsilon
+\frac{c}{2}\rho ^2,\hbox{ then }\Im p_\varepsilon (\rho
+itH_G(\rho ))\asymp -t(\varepsilon +\rho ^2).
\end{aligned}
\]}
\end{prop}

\subsection{Resolvents (Sections \ref{ltg}, \ref{prep},
  \ref{pint}, \ref{pext}, \ref{peps}, \ref{or}, \ref{dil})}
  Let {$P_\varepsilon $} be the {$h$}-Weyl quantization of {$p_\varepsilon $}. Let
  {$t>0$} be small enough and fixed.
We can extend {$G$} as a classical escape function (\cite{HeSj86}, \cite{GeSj87}) to
{${\mathbb{R}}^n_x\times {\mathbb{R}}^n_\xi $}, equal to zero over {${\mathscr{U}}_0$} and with
{$H_pG>0$} in {${{p^{-1}(0)}_\vert}_{{\mathscr{S}}_0\setminus
  \mathrm{neigh\,}(0,\mathbb{R}^n)}$}.  { After an 
  $\varepsilon$-dependent dilation 
  (see Definition \ref{APepsilon.2} and the following discussion)}, \cite{HeSj86} applies and
using the spaces $H(\Lambda_{tG};m)$ \footnote{Here $m$ is an order function, see \eqref{dil.7epsilon}.} ,
from that work,
we have a well-defined operator
{$$
P_\varepsilon :\, H(\Lambda _{tG};\widetilde{r}_\varepsilon ^2)
\longrightarrow H(\Lambda_{tG};1),
$$}
where
{$$
 \widetilde{r}_\varepsilon (x,\xi)=(r_\varepsilon ^2(x)+\xi ^2)^{\frac{1}{2}}\quad \hbox{and}\quad 
 r_\varepsilon(x)=\left(\frac{\varepsilon +x^2 }{1+x^2} \right)^{\frac{1}{2}}, 
 \,\, \forall (x,\xi)\in\mathbb{R}^{2n}.
$$}
See Appendix \ref{dil} for more details.  From now on, we denote by 
$H(\Lambda_{tG})$ the space $H(\Lambda_{tG};1)$. 

\par\noindent
Let {$P_\varepsilon ^{\mathrm{int}}$} be a suitable self-adjoint reference operator,
obtained from {$P_\varepsilon $} by ``filling the sea up to the level
{$\frac{\varepsilon}{{\mathcal O}(1)}$}''. We have Weyl asymptotics for its spectrum.

\begin{prop}
  For {$z$} satisfying
  {\begin{equation}\label{re.1}
-{\mathcal{O}}(\varepsilon )<\Re z<\frac{\varepsilon }{{\mathcal{O}}(1)}
\quad\mathrm{and}\quad
-\frac{\varepsilon }{{\mathcal{O}}(1) }<\Im z<{\mathcal{O}}(\varepsilon )
\end{equation}}
with
{\begin{equation}\label{re.2}
\mathrm{dist\,}(z,\sigma (P_\varepsilon^{\mathrm{int}}))\ge
\widetilde{\delta }\ge h^{N_0},
\end{equation}}
we have
$$
(P_\varepsilon -z)^{-1}={\mathcal{O}}(\frac{1}{\widetilde{\delta}}):\, H(\Lambda
_{tG}) \longrightarrow H(\Lambda _{tG},\widetilde{r}_\varepsilon ^2)\cap H(\Lambda _{tG}).
$$
\noindent
Inside the set (\ref{re.1})
there exists a bijection {$b:\sigma (P^{\mathrm{int}}_\varepsilon )\longrightarrow \sigma
(P_\varepsilon )$} with {$b(\mu )-\mu ={\mathcal{O}}(h^\infty )$}. 
\end{prop}

  For {$0<\delta \ll 1$}, {$0<\varepsilon \le \varepsilon (\delta )$},
  {$0<h\le h(\delta,\varepsilon)$}, let {$-{\mathcal{O}}(1)<A<B<\frac{1}{{\mathcal O}(1)}$}. 
  By moving the eigenvalues of {$P_\varepsilon ^{\mathrm{int}}$}
  out of the gaps {$A\varepsilon +]-\delta \varepsilon ,\delta \varepsilon [$}
  and {$B\varepsilon +]-\delta \varepsilon ,\delta \varepsilon [$} we can
  construct an operator {$P_{\varepsilon,\delta }:=P_{\varepsilon ,A,B,\delta }:\, H(\Lambda
  _{tG},\widetilde{r}_\varepsilon ^2)\longrightarrow H(\Lambda _{tG})$} such that
  \begin{itemize}
  \item the eigenvalues of {$P_{\varepsilon,\delta }$} with
    {$-{\mathcal{O}}(\varepsilon )<\Re z<\frac{\varepsilon}{{\mathcal{O}}(1)}$}
    belong to a complex
    {$h^{N_0}$}-neighborhood of
    $$
    R_{\varepsilon ,A,B,\delta }:=\Big]-{\mathcal{O}}(\varepsilon),
    \frac{\varepsilon}{{\mathcal{O}}(1)}\Big[\setminus
   \left(\Big\{A\varepsilon, B\varepsilon\Big\}+\Big]-\frac{\varepsilon \delta}{2},
   \frac{\varepsilon \delta}{2}\Big[\right),
   $$
 \item
   {\begin{multline*}
   \# \Big(\sigma (P_{\varepsilon,\delta })\cap \big(
   ]A\varepsilon ,B\varepsilon [+i]-o(\varepsilon ),o(\varepsilon )[\big)\Big)\\
   =\left(\frac{1}{2\pi h}
   \right)^n\Big(\omega (\varepsilon B)-\omega (\varepsilon A)\Big)
   +{\mathcal{O}}(\delta \varepsilon )h^{-n}.
   \end{multline*}}
  \end{itemize}

\vskip1.4cm


\begin{figure}
\hskip1.5cm
\begin{picture}(0,0)%
\includegraphics[scale=1.5]{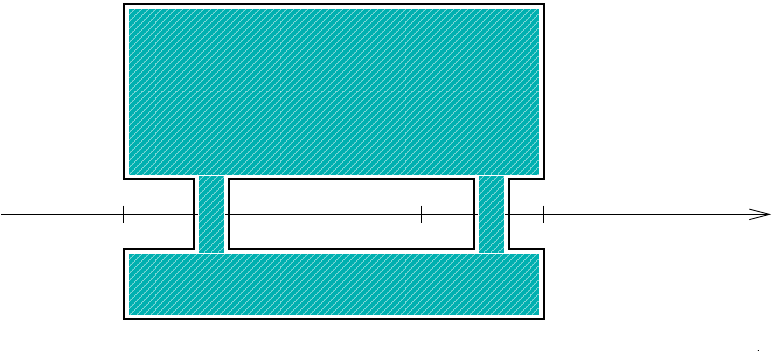}%
\end{picture}%
\setlength{\unitlength}{789sp}%
\begingroup\makeatletter\ifx\SetFigFont\undefined%
\gdef\SetFigFont#1#2#3#4#5{%
  \reset@font\fontsize{#1}{#2pt}%
  \fontfamily{#3}\fontseries{#4}\fontshape{#5}%
  \selectfont}%
\fi\endgroup%
\begin{picture}(12302,8578)(-58,-8209)
\put(15250,-3950){\makebox(0,0)[lb]{\smash{{\SetFigFont{9}{7.2}{\rmdefault}{\mddefault}{\updefault}\color{Green}$0$}}}}
\put(14820,-3500){\makebox(0,0)[lb]{\smash{{\SetFigFont{9}{7.2}{\rmdefault}{\mddefault}{\updefault}\color{Green}$\times$}}}}
\put(3300,-3000){\makebox(0,0)[lb]{\smash{{\SetFigFont{9}{7.2}{\rmdefault}{\mddefault}{\updefault}\textbf{$-C\varepsilon$}}}}}
\put(4100,-3500){\makebox(0,0)[lb]{\smash{{\SetFigFont{9}{7.2}{\rmdefault}{\mddefault}{\updefault}\textbf{$\times$}}}}}
\put(7000,-3000){\makebox(0,0)[lb]{\smash{{\SetFigFont{9}{7.2}{\rmdefault}{\mddefault}{\updefault}\textbf{$A\varepsilon$}}}}}
\put(7250,-3500){\makebox(0,0)[lb]{\smash{{\SetFigFont{9}{7.2}{\rmdefault}{\mddefault}{\updefault}\textbf{$\times$}}}}}
\put(17300,-3000){\makebox(0,0)[lb]{\smash{{\SetFigFont{9}{7.2}{\rmdefault}{\mddefault}{\updefault}\textbf{$B\varepsilon$}}}}}
\put(17300,-3500){\makebox(0,0)[lb]{\smash{{\SetFigFont{9}{7.2}{\rmdefault}{\mddefault}{\updefault}\textbf{$\times$}}}}}
\put(19700,-2700){\makebox(0,0)[lb]{\smash{{\SetFigFont{15}{7.2}{\rmdefault}{\mddefault}{\updefault}\textbf{$\frac{\varepsilon}{C}$}}}}}
\put(19200,-3500){\makebox(0,0)[lb]{\smash{{\SetFigFont{9}{7.2}{\rmdefault}{\mddefault}{\updefault}\textbf{$\times$}}}}}
\put(27800,-4300){\makebox(0,0)[lb]{\smash{{\SetFigFont{15}{7.2}{\rmdefault}{\mddefault}{\updefault}\textbf{$\mathbb{R}$}}}}}
\put(-200,4000){\makebox(0,0)[lb]{\smash{{\SetFigFont{12}{7.2}{\rmdefault}{\mddefault}{\updefault}\color{Red}$\Im(z)=C\varepsilon$}}}}
\put(3000,-2100){\makebox(0,0)[lb]{\smash{{\SetFigFont{12}{7.2}{\rmdefault}{\mddefault}{\updefault}\color{Red}$\delta\varepsilon$}}}}
\put(2300,-5000){\makebox(0,0)[lb]{\smash{{\SetFigFont{12}{7.2}{\rmdefault}{\mddefault}{\updefault}\color{Red}$-\delta\varepsilon$}}}}
\put(2300,-7400){\makebox(0,0)[lb]{\smash{{\SetFigFont{15}{7.2}{\rmdefault}{\mddefault}{\updefault}\color{Red}$-\frac{\varepsilon}{C}$}}}}
\put(21000,1000){\makebox(0,0)[lb]{\smash{{\SetFigFont{12}{7.2}{\rmdefault}{\mddefault}{\updefault}\color{Blue}$R_{\varepsilon,\delta}$}}}}
\put(19300,500){\makebox(0,0)[lb]{\smash{{\SetFigFont{12}{7.2}{\rmdefault}{\mddefault}{\updefault}\color{Blue}{\Large$\longleftarrow$}}}}}
\put(25000,3000){\makebox(0,0)[lb]{\smash{{\SetFigFont{15}{7.2}{\rmdefault}{\mddefault}{\updefault}\color{Red}$\mathbb{C}$}}}}

\end{picture}%
\caption{{The set ${R}_{\varepsilon,\delta }$.}}\label{Repsdelta}
\end{figure}

  
\subsection{Relative determinants (Sections \ref{det}, \ref{ep})}

Recall (see e.g.\ \cite{GoKr69}) that under suitable but very general
assumptions on the linear operators $\mathcal{A}$ and $\mathcal{B}$,
{\begin{multline*}
|\det {\mathcal A}{\mathcal B}^{-1}|=|\det \left( 1+({\mathcal A}-{\mathcal B}){\mathcal B}^{-1}\right)|  
\le \exp \big\|({\mathcal A}-{\mathcal B}){\mathcal B}^{-1}\big\|_{\mathrm{tr}}
\le \exp \left(\big\| {\mathcal A}-{\mathcal B}\big\|_{\mathrm{tr}} \big\| {\mathcal B}^{-1}\big\| \right).
\end{multline*}}
 \par\noindent
Here $\Vert\cdot\Vert$ denotes the operator norm and 
$\Vert\cdot\Vert_{\mathrm tr}$ the trace class norm.
Let introduce the following sets, (see Figure \ref{Repsdelta}):
\begin{itemize}
 
 \item[{}] 
 $$
 R:=\big]-{\mathcal{O}}(\varepsilon ),\frac{\varepsilon}{{\mathcal{O}}(1)}\big[
 +i\big]-\frac{\varepsilon}{{\mathcal O}(1)},{\mathcal{O}}(\varepsilon)\big[,
 $$
 
 \item[{}] 
 
 $$
 R_\delta:=\{ z\in R;\, |\Im z|>\delta\varepsilon\},\quad \textrm{and}
 $$
 
 \item[{}] 
 
\begin{equation*}
\begin{aligned}
  R_{\varepsilon,\delta}:=
  R_\delta 
  & \bigcup \Big( A\varepsilon +\big]-\frac{\delta \varepsilon}{4},
  \frac{\delta \varepsilon}{4}\big[+i\big[-\delta \varepsilon ,\delta \varepsilon\big]\Big)\\
  & \bigcup \Big( B\varepsilon +\big]-\frac{\delta \varepsilon}{4},
  \frac{\delta \varepsilon}{4}\big[+i\big[-\delta \varepsilon ,\delta \varepsilon\big]\Big).           
\end{aligned}
\end{equation*}
\end{itemize}

\par\smallskip\noindent
We can construct an operator {$P_\varepsilon ^{\mathrm{ext}}$} as a ``trace
class filling of {$P_\varepsilon $} over {${\mathscr{U}}_0$}'' 
so that {$(P_\varepsilon^{\mathrm{ext}}-z)^{-1}$} is 
{${\mathcal{O}}(\frac{1}{\varepsilon})$} as a bounded operator from $H(\Lambda _{tG})$
to itself, for all {$z\in R$}. We also have
\begin{equation*}
 \big\| P_\varepsilon -P_\varepsilon^{\mathrm{ext}}\big\|_{\mathrm{tr}}={\mathcal{O}}(h^{-n}),\,\,
\big\| P-P_\varepsilon \big\|_{\mathrm{tr}}={\mathcal O}(\varepsilon^{n+1}h^{-n})\,\, {\textrm{and}}
\end{equation*}
\begin{equation*}
 \big\| P_\varepsilon -P_{\varepsilon,\delta } \big\|_{\mathrm{tr}}={\mathcal O}\Big((\varepsilon \delta)^2 h^{-n}\Big).
\end{equation*}

\par\noindent
Define
{\[\begin{aligned}
{\mathcal D}_P(z)&=\ln \left|\det (P-z)(P_\varepsilon ^{\mathrm{ext}}-z)^{-1}\right|\\
{\mathcal D}_{P_\varepsilon }(z)&=\ln \left|\det (P_\varepsilon -z)(P_\varepsilon ^{\mathrm{ext}}-z)^{-1}\right|\\
{\mathcal D}_{P_{\varepsilon ,\delta }}(z)&=\ln \left|\det (P_{\varepsilon ,\delta }-z)(P_\varepsilon ^{\mathrm{ext}}-z)^{-1}\right| .
\end{aligned}
\]}

\par\noindent
The zeros of {$\det \Big((P-z)(P_\varepsilon^{\mathrm{ext}}-z)^{-1}\Big)$} coincide with the resonances of {$P$}.
We have
{$$
{\mathcal D}_P-{\mathcal D}_{P_\varepsilon }
\left\{\begin{matrix}
 \le {\mathcal{O}}_\delta (1)\varepsilon ^n h^{-n} &\hbox{ in } & {\hskip-3cm} R_\delta,\\
&{} & {} & {}\\
 \ge   -{\mathcal{O}}_\delta(1)\varepsilon ^n h^{-n}
&\hbox{ in }  & R_\delta\cap \big\{z\in\mathbb{C};\, \Re z\le -\frac{\varepsilon}{{\mathcal{O}}(1)}\big\}. 
\end{matrix}\right.
$$}

\par\noindent
Similar estimates hold for {${\mathcal D}_P-{\mathcal D}_{P_{\varepsilon ,\delta}}$} 
with {$\varepsilon ^n h^{-n}$} replaced by {$\varepsilon \delta h^{-n}$} and
{$R_\delta $} by {$R_{\varepsilon,\delta }$}. 

\par\smallskip\noindent
Standard arguments, including Jensen's formula, lead to

\begin{prop} 
\begin{itemize}
 \item[{}] 
 
 \item[(A)] The number of resonances in {$R_\delta$} is {$\le {\mathcal{O}}_\delta (1)\varepsilon^n h^{-n}$} 
with the usual convention that {$0<\varepsilon \le \varepsilon (\delta )$}, {$0<h\le h(\delta,\varepsilon)$}.
 
 \item[(B)] There are plenty of {$z\in R_\delta $} for which
{\begin{equation}
{\mathcal D}_P(z)-{\mathcal D}_{P_\varepsilon }(z)\ge -{\mathcal{O}}_\delta (1)\varepsilon^n h^{-n}.
\end{equation}}

 \end{itemize}
\end{prop}
\noindent
The point {\textit{(A)}} of the above proposition gives part  {\textit{(A)}} in Theorem \ref{int1}.
  
\begin{prop}\label{det'2}

\begin{itemize}
\item[{}]

\item[(A)] The number of resonances of {$P$} in
  {$R_{\varepsilon, \delta }$} is {$\le {\mathcal{O}}(\varepsilon \delta )h^{-n}$}.
  
 \item[(B)] There are plenty of {$z\in R_{\varepsilon, \delta }$} for which
{\begin{equation}
{\mathcal D}_P(z)-{\mathcal D}_{P_{\varepsilon ,\delta }}(z)\ge -{\mathcal{O}}(\varepsilon \delta )h^{-n}.
\end{equation}}
\end{itemize}

\end{prop}

\par\smallskip\noindent
Consider the holomorphic function {$f(z)=\det\Big((P-z)(P_\varepsilon^{\mathrm{ext}}-z)^{-1}\Big)$} on {$R$}. Then
{$$
\Big|f(z)\Big|\le \exp \left(h^{-n}\Big(\phi (z)+{\mathcal{O}}\big(\varepsilon \delta\big)\Big)
\right)\hbox{ in }R_{\varepsilon,\delta },
$$}
where {$\phi (z)=h^n{\mathcal D}_{P_{\varepsilon ,\delta }}$}. By \textit{(B)} in
Proposition \ref{det'2} we have
{$$
\Big|f(z)\Big|\ge \exp \left( h^{-n}\Big(\phi (z)-{\mathcal{O}}\big(\varepsilon \delta\big)\Big)
\right)\hbox{ at plenty of points in }R_{\varepsilon, \delta }.
$$}
We can then apply Theorem 1.1 in \cite{Sj10} (or \cite[Theorem 12.1.1]{Sj19})
  with {$h$} there replaced by {$h^n$}, to finish the proof.

\section{Escape functions}\label{esc}
\setcounter{equation}{0}

From the property (\ref{int.8}), we get after an orthogonal change of
$x$-variables and a subsequent dilation in the variable $x_n$,
$$
p(x,\xi)=\frac{\kappa }{2}(\xi _n^2-x_n^2)+\frac{1}{2}q(x',\xi ')+{\mathcal{O}}(x^3),
$$
where $\kappa$ is a positive constant. Here we write $x=(x',x_n)\in {\mathbb{R}}^n$, 
$x'=(x_1,\ldots,x_{n-1})\in {\mathbb{R}}^{n-1}$
and similarly for the dual variable $\xi $. The quadratic form $q$ is positive
definite. For simplicity, we may assume that $\kappa =1$:
\begin{equation}\label{esc.1}
p(x,\xi)=\frac{1}{2}(\xi _n^2-x_n^2)+\frac{1}{2}q(x',\xi ')+{\mathcal{O}}(|x|^3),
\end{equation}
for $x\in \mathrm{neigh\,}(0,{\mathbb{R}}^n)$, $\xi \in {\mathbb{R}}^n$. Let
\begin{equation}\label{esc.2}
p_0=\frac{1}{2}(\xi _n^2-x_n^2)\quad {\rm and}\quad G_0=x_n\xi _n.
\end{equation}
Then
\begin{equation}\label{esc.3}
H_{p_0}=(\partial _{\xi _n}p_0)\partial _{x_n}-(\partial _{x
  _n}p_0)\partial _{\xi _n}
\end{equation}
and
\begin{equation}\label{esc.4}
H_{p_0}G_0=x_n^2+\xi _n^2.
\end{equation}

\par\smallskip 
Let $\Psi \in C^\infty\big({\mathbb{R}};[0,1]\big)$ have its support 
$=]-\infty ,1]$ and  be equal to one on
$\big]-\infty ,\frac{1}{2}\big]$. 
For $\lambda \ge 1$ large enough (to be fixed below), we put
\begin{equation}\label{esc.5}
G(x,\xi )=G^\varepsilon (x,\xi )=\left(1-\Psi \left(\frac{\lambda x_n}
{\sqrt{\varepsilon +x'^2+\xi ^2}} \right) \right) G_0(x,\xi ).
\end{equation}
Here $\varepsilon >0$ is a small parameter with respect to which our
estimates will be uniform. We will consider $G$ for $(x,\xi )\in
\mathrm{neigh\,}(0,{\mathbb{R}}^{2n})$, $\xi \in {\mathbb{R}}^n$ and notice first
that $G=0$ for $x_n<0$, so we may restrict the attention to the region
$x_n\ge 0$.

\par\smallskip 
The prefactor $1-\Psi (\lambda x_n(\varepsilon +x'^2+\xi
^2)^{-\frac{1}{2}})$ in (\ref{esc.5}) is smooth and positively homogeneous of
degree 0 in the variables $\big(\sqrt{\varepsilon},\rho\big)=\big(\sqrt{\varepsilon},x,\xi\big)$. 
It follows that in $\mathrm{neigh\,}(0,{\mathbb{R}}^{2n})$
\begin{equation}\label{esc.6}
\partial _\rho ^\alpha G={\mathcal{O}}(1)(\varepsilon +\rho ^2)^{1-\frac{|\alpha|}{2}},\quad 
\alpha \in {\mathbb{N}}^{2n}. 
\end{equation}
For $|\alpha |=0,1$, we have 
\begin{equation}\label{esc.7}
G={\mathcal{O}}(1)|\rho |^2,\quad \partial _{\rho }G={\mathcal{O}}(|\rho |).
\end{equation}
Notice that the support of $G$ is contained in the region where
\begin{equation}\label{esc.7.5}
x_n^2\ge \frac{1}{4\lambda ^2}\Big(\varepsilon +(x',\xi ')^2+\xi_n^2\Big).
\end{equation}

\par\noindent 
In the cutoff region, where the prefactor in (\ref{esc.5}) is
$\ne 1$, we have
\begin{equation}\label{esc.8}
x_n^2\le \frac{1}{\lambda ^2}\Big(\varepsilon +(x',\xi ')^2+\xi _n^2\Big).
\end{equation}
Since the quadratic form $q$ is positive definite,
\begin{equation}\label{esc.9}
q(x',\xi ')\ge \frac{1}{C}(x',\xi ')^2,
\end{equation}
we get from (\ref{esc.1}), (\ref{esc.8}), that  in the cutoff region,
\begin{equation}\label{esc.10}
  p\ge \frac{1}{2}\left(\xi _n^2+\frac{1}{C}(x',\xi ')^2 \right)
  -\frac{1}{2\lambda ^2}\left(\varepsilon +(x',\xi ')^2+\xi _n^2\right)
  -{\mathcal{O}}\big(|x|^3\big).
\end{equation}

In view of (\ref{esc.8}) we can add the non-positive term
$$
\frac{1}{2}x_n^2-\frac{1}{2\lambda ^2}\big(\varepsilon +(x',\xi ')^2+\xi _n^2\big)
$$ 
to the right hand side and get in the cutoff region,
$$
p\ge \frac{1}{2}\left(x_n^2+\xi _n^2+\frac{1}{C}(x',\xi ')^2 \right)
-\frac{1}{\lambda ^2}\left( \varepsilon +(x',\xi ')^2+\xi _n^2
\right)-{\mathcal{O}}\big(|x|^3\big).
$$
Assume for simplicity that $C\ge 1$ and fix $\lambda \ge 1$ such that
\begin{equation}\label{esc.11}
\frac{1}{\lambda ^2}<\frac{1}{4C}.
\end{equation}
Absorbing the term $-{\mathcal{O}}(|x|^3)$ by restricting the attention to
a small neighborhood of $x=0$, we get in the cutoff region (that is
the one where the prefactor in (\ref{esc.5}) is $\ne 1$),
\begin{equation}\label{esc.12}
p\ge -\frac{\varepsilon }{4C}+\frac{1}{4}x_n^2+\frac{1}{4}\xi
_n^2+\frac{1}{4C}(x',\xi ')^2,
\end{equation}
hence with a new constant $C\ge 1$:
\begin{equation}\label{esc.13}
p(\rho)\ge -\frac{\varepsilon }{C}+\frac{\rho ^2}{C}.
\end{equation}

\par\smallskip 
Outside the cutoff region, we have $G=G_0$ and hence
$$
H_pG=H_pG_0=H_{p_0}G_0+{\mathcal{O}}(|x|^2)\partial_{\xi_n}G_0,
$$
so
\begin{equation}\label{esc.14}
H_pG=x_n^2+\xi_n^2+{\mathcal{O}}(|x|^2)x_n.
\end{equation}
Here, we also have $\lambda x_n\ge \big(\varepsilon+(x',\xi')^2+\xi_n^2\big)^{\frac12}$,
where $\lambda \ge 1$ is now fixed and hence we have outside the cutoff
region and inside a small neighborhood of $(0,0)$:
\begin{equation}\label{esc.15}
  H_pG\asymp \varepsilon +|\rho |^2.
\end{equation}
Notice that if 
\begin{equation}\label{esc.16}
p(\rho)<-\frac{\varepsilon}{C}+\frac{\rho ^2}{C}, 
\end{equation}
then by (\ref{esc.13}) we are outside the
cutoff region and (\ref{esc.15}) holds. 
\begin{prop}\label{esc1}
Let $\lambda >0$ be sufficiently large in the definition (\ref{esc.5}) of
$G^\varepsilon $ and let $(\rho ,\varepsilon )$ vary in
$\mathrm{neigh\,}(0,{\mathbb{R}}^{2n})\times ]0,\varepsilon _0]$
for $\varepsilon _0>0$ sufficiently small. Then there exists a
constant $C\ge 1$ such that if (\ref{esc.16}) holds, then
(\ref{esc.15}) holds uniformly for $\rho \in \mathrm{neigh\,}(0,{\mathbb{R}}^{2n})$.

\par\smallskip\noindent
If $p(\rho)<-\frac{\varepsilon}{\widetilde{C}}+\frac{\rho^2}{C}$ for some fixed $\widetilde{C}\ge C$, we reach
the same conclusion if we replace $G^\varepsilon $ by
$G^{\widetilde{\varepsilon }}$ where $\widetilde{\varepsilon}=\frac{C}{\widetilde{C}}\varepsilon$. 
\end{prop}

\section{Adding a bump at the saddle point}\label{bp}
\setcounter{equation}{0}

Let $\chi(x,\xi )>0$ be an analytic function on ${\mathbb{R}}^{2n}$ { with a holomorphic extension
to the domain
\begin{equation}\label{bp.1.5}
|\Im (x,\xi )|< \frac{1}{C}\big\langle \Re (x,\xi )\big\rangle\,,
\end{equation}
satisfying
\begin{equation}\label{bp.1}
  \chi(x,\xi )={\mathcal{O}}(1)\exp\left(-\frac{1}{C}\Big(\Re \big(x,\xi\big)\Big)^2\right),
\end{equation}
where} the extension is denoted by the same symbol. We use the standard notation
$\langle \rho \rangle =(1+\rho ^2)^{\frac{1}{2}}$ for real vectors $\rho $.

\par\smallskip 
For $0<\varepsilon \ll 1$, we put
\begin{equation}\label{bp.3}
\chi_\varepsilon (x,\xi )=\varepsilon\, 
\chi\Big(\frac{x}{\sqrt{\varepsilon}},
\frac{\xi}{\sqrt{\varepsilon}}\Big).
\end{equation}
Assume for simplicity that
\begin{equation}\label{bp.4}
\chi\le 1\hbox{ on the real domain}.
\end{equation}
Then
\begin{equation}\label{bp.5}
\chi_\varepsilon \le \varepsilon \hbox{ on the real domain}.
\end{equation}
For every $r_0>0$, there exists $a_0=a_0(r_0,\chi)>0$ such
that
$$
\chi\ge a_0\hbox{ on }B_{{\mathbb{R}}^{2n}}(0,r_0). $$
Since
$\chi$ is positive, we get
\begin{equation}\label{bp.6}
\chi(x,\xi )\ge a\left(1-\frac{x^2+\xi^2}{r_0^2}\right)\hbox{ on }{\mathbb{R}}^{2n},
\hbox{ for }0\le a\le a_0.
\end{equation}
Hence for the same values of $a$
\begin{equation}\label{bp.7}
\chi_\varepsilon (x,\xi )\ge a \left(\varepsilon 
- \frac{x^2+\xi^2}{r_0^2}\right),
\end{equation}
on the real domain.

\par\smallskip
Let
\begin{equation}\label{bp.8}
p_\varepsilon =p+\chi_\varepsilon .
\end{equation}
Then, by (\ref{bp.5}),
\begin{equation}\label{bp.9}
p_\varepsilon \le p+\varepsilon .
\end{equation}

Assume that
\begin{equation}\label{bp.10}
p_\varepsilon(\rho) \le b\varepsilon +c\rho ^2,
\end{equation}
where $b,c>0$ are constants to be chosen below. Using (\ref{bp.6}),
(\ref{bp.7}), we have
$$
p(\rho)=p_\varepsilon(\rho) -\chi_\varepsilon(\rho) \le b\varepsilon
+c\rho^2-a\varepsilon +\frac{a}{r_0^2}\rho ^2, 
$$
i.e.
\begin{equation}\label{bp.11}
p(\rho)\le (b-a)\varepsilon +\left( c+\frac{a}{r_0^2} \right) \rho ^2.
\end{equation}
Let $C\ge 1$ be the constant in Proposition \ref{esc1}. If for some constant
$\widetilde{C}\ge C$,
\begin{equation}\label{bp.12}
b-a\le -\frac{1}{\widetilde{C}},\qquad  c+\frac{a}{r_0^2}\le \frac{1}{C},
\end{equation}
we get
\begin{equation}\label{bp.13}
p(\rho)\le -\frac{\varepsilon }{\widetilde{C}}+\frac{\rho ^2}{C}
\end{equation}
and hence (\ref{esc.15}) holds for $G=G^{\widetilde{\varepsilon }}$
with $\widetilde{\varepsilon } =
\frac{C}{\widetilde{C}}\varepsilon $ in (\ref{esc.5}) (as we saw
in Proposition \ref{esc1})
\begin{equation}\label{bp.14}
H_pG\asymp \varepsilon +\rho ^2.
\end{equation}

For a given $r_0>0$, we know that (\ref{bp.7}) holds for $0<a\le a_0$
for some $a_0>0$. Choose $a$ so that $\frac{a}{r_0^2}\le \frac{1}{2C}$ and put
$c=\frac{1}{2C}$. Then the second inequality in (\ref{bp.12}) is
valid. Choose $b=\frac{a}{2}$ and $\widetilde{C}\ge C$ large enough. Then the
first estimate in (\ref{bp.12}) also holds. With this choice of $a$,
$b$, $c$, 
we know that (\ref{bp.10}) implies (\ref{bp.13}) and
hence also (\ref{esc.15}) ($\Longleftrightarrow$ (\ref{bp.14})), 
for $G=G^{\frac{C\varepsilon }{\widetilde{C}}}$.

After the dilation $\chi(\rho )\mapsto \chi(\alpha
\rho )$, $\alpha \in ]0,1]$, 
(\ref{bp.7}) remains valid. Hence we still have that (\ref{bp.10}) $\Longrightarrow$
(\ref{bp.14}) ($\Longleftrightarrow$ (\ref{esc.15})) (with the same fixed
dilation in $\varepsilon $).

\par\smallskip
We next study
\begin{equation*}
  \begin{aligned}
H_{p_\varepsilon }G
& =H_pG+\left\{ \varepsilon \chi\Big(\frac{x}{\sqrt{\varepsilon}},\frac{\xi}{\sqrt{\varepsilon}}\Big),G \right\}\\ 
& =H_pG+{\mathcal{O}}(1) \Big|(\nabla\chi)
\Big(\frac{x}{\sqrt{\varepsilon}},\frac{\xi}{\sqrt{\varepsilon}}\Big)\Big| \,
\sqrt{\varepsilon} \, |(x,\xi )|,
\end{aligned}
\end{equation*}
where we used (\ref{esc.7}) in the last step.
Here $\{ f,g \}=H_fg$ denotes the Poisson bracket of two $C^1$
functions $f,g$.
Thus
\begin{equation}\label{bp.16}
H_{p_\varepsilon }G=H_pG+{\mathcal{O}}(1)\big\| \nabla\chi\big\|_{L^\infty }(\varepsilon +\rho ^2).
\end{equation}
Replacing $\chi$ with $\chi(\alpha \rho )$,
gives
\begin{equation}\label{bp.17}
H_{p_\varepsilon }G=H_pG+{\mathcal{O}}(1)\alpha (\varepsilon +\rho ^2).
\end{equation}

Choose $a$, $b$, $c$ as in the preceding discussion and $\alpha >0$
small enough in
\begin{equation}\label{bp.18}
p_\varepsilon (\rho )=p(\rho )+\varepsilon \chi(\varepsilon^{-\frac{1}{2}}\alpha \rho ).
\end{equation}
Then in the region (\ref{bp.10}) we have
\begin{equation}\label{bp.19}
H_{p_\varepsilon }G\asymp \varepsilon +\rho ^2,\quad
G=G^{\widetilde{\varepsilon }},\quad \widetilde{\varepsilon
}=\frac{C}{\widetilde{C}}\varepsilon.
\end{equation}
\begin{prop}\label{bp1} Define $p_\varepsilon=p+\chi _\varepsilon  $ as
  in (\ref{bp.1}), (\ref{bp.1.5}), (\ref{bp.3}), (\ref{bp.4}). Assume
  also that $\big\| \nabla \chi \big\|\le \alpha _0$ for some sufficiently
  small $\alpha _0>0$.
Let $G^\varepsilon $ be as in Proposition (\ref{esc1})  and let $(\rho ,\varepsilon )$ vary in
$\mathrm{neigh\,}(0,{\mathbb{R}}^{2n})\times ]0,\varepsilon _0]$
for $\varepsilon _0>0$ sufficiently small. Then there exist
constants $b,c>0$ and $\widetilde{C}\ge C>0$ such that if (\ref{bp.10})
holds, then we have
(\ref{bp.13}), and
(\ref{bp.19}) holds uniformly for $\rho \in \mathrm{neigh\,}(0,{\mathbb
  R}^{2n})$. 
\end{prop}

\section{Deformed phase space}\label{ltg}
\setcounter{equation}{0}

\par\smallskip 
We continue to work with the function $G$ of Section \ref{esc},
where $\lambda \ge 1$ now is fixed and $\varepsilon $ is a small
parameter. $G$ vanishes near $\{ (x,\xi )\in \mathrm{neigh\,}(0,\mathbb{R}^{2n});\,
x_n{\leq} 0 \}$ and we restrict the attention to the set $x_n\ge 0$. We
saw in Proposition \ref{esc1} that there is a constant $C\ge 1$ such that
$$
\hbox{(\ref{esc.16}) }\Longrightarrow \hbox{ (\ref{esc.15})}
$$
uniformly. Also if $\widetilde{C}\ge C$, the estimate
\begin{equation}\label{ltg.7}
p(\rho )\le -\frac{\varepsilon }{\widetilde{C}}+\frac{\rho ^2}{C}
\end{equation}
implies (\ref{esc.15}) uniformly, where $G=G^{\frac{C\varepsilon}{\widetilde{C}}}$.

\par\smallskip 
For $0\le t\ll 1$, we introduce the $\mathbf{IR}$-manifold (see Appendix \ref{dil})
\begin{equation}\label{ltg.1}
\Lambda _{tG}=\big\{ \rho +itH_G(\rho );\, \rho \in
\mathrm{neigh\,}(0,{\mathbb{R}}^{2n}) \big\}.
\end{equation}
By Taylor expansion, we have
\begin{equation}\label{ltg.2}
\Im p\Big(\rho +itH_G(\rho )\Big)=-tH_pG(\rho )+{\mathcal{O}}\left(t^3|\rho |^3 \right),
\end{equation}
\begin{equation}\label{ltg.3}
\Re p\Big(\rho +itH_G(\rho )\Big)=p(\rho )+{\mathcal{O}}\left(t^2\rho ^2 \right).
\end{equation}
Here we also use that $|H_G(\rho )|=|\partial _\rho G|={\mathcal{O}}(|\rho
|)$ by (\ref{esc.7}).
\begin{prop}\label{ltg1}
  Let $\widetilde{C}\ge C>0$ be as above and let
  $G=G^{\frac{C\varepsilon}{\widetilde{C}}}$. Then if $t>0$ is small enough,
  we have that
\begin{equation}\label{ltg.4}
\Re p\big(\rho +itH_G(\rho )\big)\le -\frac{\varepsilon
}{\widetilde{C}}+\frac{\rho ^2}{2C}
\end{equation}
implies that
\begin{equation}\label{ltg.5}
\Im p\big(\rho +itH_G(\rho )\big)\asymp -t(\rho ^2+\varepsilon ).
\end{equation}
We recall that we work in $\mathrm{neigh\,}(0,{\mathbb{R}}^{2n})\cap \{
x_n\ge 0\}$.
\end{prop}

\begin{proof}
  From (\ref{ltg.4}) we get by means of (\ref{ltg.3}),
  \begin{equation}\label{ltg.6}
p(\rho )\le -\frac{\varepsilon }{\widetilde{C}}+\left(\frac{1}{2C}+{\cal
    O}(t^2) \right) \rho ^2
\end{equation}
and hence (\ref{ltg.7}):
$$
p(\rho )\le -\frac{\varepsilon }{\widetilde{C}}+\frac{\rho ^2}{C},
$$
if $t$ is small enough. Then by (\ref{esc.15}), (\ref{ltg.2}) we get
\begin{equation}\label{ltg.8}
\begin{aligned}
    \Im p\big(\rho +itH_G(\rho )\big)&=-tH_pG(\rho )+{\mathcal{O}}(t^3|\rho |^3)\\
    &=-t\left(H_pG+{\mathcal{O}}(t^2)|\rho |^3 \right)\\
    &\asymp -t(\varepsilon +\rho ^2).
\end{aligned}
\end{equation}
\vskip-10pt
\end{proof}

\par\smallskip 
We next turn to $p_\varepsilon $ and recall Proposition \ref{bp1}.
From (\ref{bp.1}) and the Cauchy inequalities, we get after slightly
increasing the constant $C=C_\chi>0$ there:
\begin{equation}\label{ltg.9}
\partial _\rho ^\alpha \chi(\rho )={\mathcal{O}}(1)\exp
\left(-\frac{1}{C}(\Re \rho )^2\right),\qquad |\Im \rho |<\frac{1}{C}\langle \Re \rho \rangle.
\end{equation}
For $\chi_\varepsilon $ (cf.\ (\ref{bp.3})) we get
\begin{equation}\label{ltg.10}
\partial _\rho ^\alpha \chi_\varepsilon (\rho )={\mathcal{O}}(1)
\varepsilon ^{1-\frac{|\alpha |}{2}}\exp \left(-\frac{1}{C\varepsilon}(\Re \rho )^2\right),
\end{equation}
when
\begin{equation}\label{ltg.11}
|\Im \rho |\le \frac{1}{C}\Big(\sqrt{\varepsilon}+|\Re \rho |\Big).
\end{equation}
In particular, we have
\begin{equation}\label{ltg.12}
\partial _\rho ^\alpha \chi_\varepsilon (\rho )={\cal
  O}(1)\exp \left(-\frac{1}{C\varepsilon}(\Re \rho )^2 \right),\quad \hbox{when} \,\,|\alpha| \le 2.
\end{equation}
Since $H_G={\mathcal{O}}(\rho )$, $\Lambda _{tG}$ is included in the region
(\ref{ltg.11}) when $0\leq t\ll 1$ and by Taylor expansion we get
\begin{equation}\label{ltg.13}
  \begin{aligned}
-\Im p_\varepsilon\Big(\rho +itH_G(\rho )\Big)&=tH_{p_\varepsilon }G+{\mathcal{O}}
\left(1+\varepsilon ^{-\frac{1}{2}}e^{-\frac{\rho ^2}{C\varepsilon }}
\right)t^3|\rho |^3\\ &=tH_{p_\varepsilon }G+{\mathcal{O}}(1)t^3|\rho |^2.
  \end{aligned}
\end{equation}
Similarly,
\begin{equation}\label{ltg.14}
\Re p_\varepsilon\Big(\rho +itH_G(\rho )\Big)=p_\varepsilon (\rho )+{\mathcal{O}}(1)t^2\rho ^2.
\end{equation}
This is analogous to (\ref{ltg.2}), (\ref{ltg.3}) and we get
\begin{prop}\label{ltg2}
Let $b,c,C,\widetilde{C}>0$ be as in Proposition \ref{bp1}, choose $G$
as in (\ref{bp.19}). Then for $0\le t\ll 1$, if
\begin{equation}\label{ltg.15}
\Re p_\varepsilon\Big(\rho +itH_G(\rho )\Big)\le b\varepsilon +\frac{c}{2}\rho^2,
\end{equation}
we have (\ref{bp.10}):
$$
p_\varepsilon (\rho )\le b\varepsilon +c\rho ^2,
$$
and we conclude as in Proposition \ref{bp1} that (\ref{bp.13}),
(\ref{bp.19}) hold:
\begin{equation}\label{ltg.16}
p(\rho )\le -\frac{\varepsilon }{\widetilde{C}}+\frac{\rho ^2}{C},
\end{equation}
\begin{equation}\label{ltg.17}
H_{p_\varepsilon }G\asymp \varepsilon +\rho ^2.
\end{equation}
Hence by (\ref{ltg.13}),
\begin{equation}\label{ltg.18}
-\Im p_\varepsilon\Big(\rho +itH_G(\rho )\Big)\asymp t(\varepsilon +\rho ^2),
\end{equation}
when $0<t\ll 1$.
Recall here that $\rho \in \mathrm{neigh\,}(0,{\mathbb{R}}^{2n})$ with
$x_n(\rho )\ge 0$.
\end{prop}

Define $r(x)$, $R(x)$ as in (\ref{dil.8}) below:
$$
r(x)=1,\qquad R(x)=\langle x\rangle, 
$$
and put
$$
\widetilde{r}(x,\xi )=(r(x)^2+\xi ^2)^{\frac{1}{2}}.
$$
Define the class $\dot{S}(\mathbb{R}^{2n};R\,\widetilde{r})$ as in
Definition \ref{AP.2} (b) (see \cite[Chapter 1, D\'efinition 1.4]{HeSj86}) \footnote{
For our special choice of $r,R$, $\dot{S}(\mathbb{R}^{2n};R\,\widetilde{r})={S}(\mathbb{R}^{2n};R\,\widetilde{r})$, but we prefer 
$\dot{S}(\mathbb{R}^{2n};R\,\widetilde{r})$
as in the general theory, allowing for more 
general scales near infinity.}. 
From the appendix in \cite{GeSj87} we see that if 
${\Psi_0}\in C_0^\infty ({\mathbb{R}}^{2n};[0,1])$ is equal to one near $(0,0)$, then
there exists $\widetilde{G}={\Psi_0} G+F$, where $F$ is independent of
$\varepsilon $,
\begin{equation}\label{ltg.19}
F\in \dot{S}(\widetilde{r}R),\ F=0\hbox{ when }|\xi |\gg r(x),
\end{equation}
\begin{equation}\label{ltg.20}
  (0,0)\not\in \mathrm{supp\,}F,
  \end{equation}
  \begin{equation}\label{ltg.21}
H_p\widetilde{G}>0\hbox{ on }
p^{-1}(0)\cap \left(\overline{\mathscr{S}}_0\times\mathbb{R}^n\setminus \{(0,0) \}\right),
\end{equation}
and uniformly $\ge \frac{1}{\mathcal{O}(1)}$ outside any fixed neighborhood of $(0,0)$ in
{$\big(\overline{\mathscr{S}}_0\times\mathbb{R}^n\big)\setminus \{(0,0)\}$}.  
We can also arrange so that $\pi _x(\mathrm{supp\,}F)$ 
 is contained in an arbitrarily small neighborhood of $\overline{\mathscr{S}}_0$
($\pi_x$ is introduced in Footnote \ref{can.projection} {in Section
\ref{int}}).

\par\smallskip 
By Taylor expansion we see that
\begin{itemize}

\item In a small fixed neighborhood of $(0,0)$ we have
  $\widetilde{G}=G$ and the Propositions \ref{ltg1}, \ref{ltg2} hold
  with $G$ replaced by $\widetilde{G}$.
  
\item Away from any fixed neighborhood of $(0,0)$ and for  any fixed
    $t\in ]0, t_0]$ with $t_0>0$ small enough, we have
\begin{equation}\label{ltg.21.5}
\begin{cases}p(\rho )=0,\\ \pi _x(\rho )\in {\mathscr{S}}_0\end{cases} 
\Longrightarrow -\Im p\Big((\rho +itH_{\widetilde{G}}(\rho )\Big)\asymp t,
\end{equation}
uniformly in $\varepsilon $. Also, since $\Re p\Big(\rho+itH_{\widetilde{G}}(\rho )\Big)
=p(\rho )+{\mathcal{O}}(t^2)$, 
we conclude that away from any small fixed neighborhood of $(0,0)$, we have
\begin{equation}\label{ltg.22}
\frac{t}{C}\widetilde{r}^2\le \Big|p\big(\rho
+itH_{\widetilde{G}}(\rho )\big)\Big|\le C \widetilde{r}^2.
\end{equation}
   and ${{p\big(\rho +itH_{\widetilde{G}}(\rho )\big)}_\big\vert}_{\overline{\mathscr{S}}_0\times\mathbb{R}^n}$
   is an elliptic symbol of class $S(\mathbb{R}^{2n};\widetilde{r}^2)$ away from any
   fixed neighborhood of $(0,0)$. 
\end{itemize}

\section{Preparations for the study of $P_{\varepsilon}$}\label{prep}
\setcounter{equation}{0}

\par\smallskip Let
\begin{equation}\label{prep.1}
P=-h^2\Delta +V(x),
\end{equation}
so that $P$ is the $h$-Weyl quantization of the symbol 
$p(x,\xi )=\xi^2+V(x)$. Recall the definition of the symbol $\chi
_\varepsilon (x,\xi )$ in (\ref{bp.3}). By $\chi_\varepsilon $
we will also denote a suitable $h$-quantization (very close to
the Weyl-one). Let
\begin{equation}\label{prep.2}
P_\varepsilon =P+\chi_\varepsilon 
\end{equation}
be the corresponding quantization of
\begin{equation}\label{prep.2.5}
p_\varepsilon =p+\chi_\varepsilon =\xi ^2+V(x)+{\chi}_\varepsilon (x,\xi ).
\end{equation}
Assume for simplicity that
\begin{equation}\label{prep.3}
\partial _\xi \chi(x,0)=0,
\end{equation}
a condition which is fulfilled in the main case that we have in mind:
\begin{equation}\label{prep.4}
\chi(x,\xi )=\exp \left(-\frac{1}{C}(x,\xi )^2 \right).
\end{equation}
Recall that in Section \ref{bp}, we have replaced ${\chi
}(\rho )$ by $\chi(\alpha \rho )$ for some sufficiently
small fixed $\alpha >0$, in order to have Proposition \ref{bp1} available. 
With $\alpha $ small enough, we get from
(\ref{prep.3}) that
\begin{equation}\label{prep.5}
\inf_{\xi\in\mathbb{R}^n }\left( \frac{\xi ^2}{2}+
\chi_\varepsilon (x,\xi) \right)=\chi_\varepsilon (x,0).
\end{equation}
This follows from the fact that
\begin{equation}\label{prep.5.5}
\partial _\rho ^2\chi_\varepsilon ={\mathcal{O}}(\alpha ^2),\quad 
\partial_\rho  \chi_\varepsilon ={\mathcal{O}}(\sqrt{\varepsilon}\,\alpha ),\quad \chi
_\varepsilon ={\mathcal{O}}(\varepsilon ).
\end{equation}
(We could here replace $\frac{\xi ^2}{2}$ by $\theta \xi ^2$ 
for any $0<\theta<1$ if $\alpha =\alpha (\theta )$ is small enough.)

As a natural potential associated to $P_\varepsilon$, we put
\begin{equation}\label{prep.6}
V_\varepsilon (x)=V(x)+\chi_\varepsilon (x,0)=\inf_{\xi\in\mathbb{R}^n}
\left(\frac{\xi ^2}{2}+V(x)+\chi_\varepsilon (x,\xi )\right).
\end{equation}
Using (\ref{prep.5.5}),
we see that $V_\varepsilon $ is a small perturbation of $V$ in $C^2$ and
has a critical point $x_c(\varepsilon )={\mathcal{O}}\big(\sqrt{\varepsilon}\,\big)$ 
which is uniformly non-degenerate of signature $(n-1,1)$. Also,
\begin{equation}\label{prep.7}
V_\varepsilon \big(x_c(\varepsilon)\big)\asymp \varepsilon .
\end{equation}

\par\smallskip 
Assume for simplicity that
\begin{equation}\label{prep.8}
\partial \chi(0)=0.
\end{equation}
Then $x_c(\varepsilon )=0$ and
\begin{equation}\label{prep.9}
V_\varepsilon \big(x_c(\varepsilon )\big)=\varepsilon \chi(0)=:E_\varepsilon .
\end{equation}

\par\smallskip 
In analogy with (\ref{int.3}), we have
\begin{equation}\label{prep.10}
V_\varepsilon ^{-1}\Big(\big]-\infty ,E_\varepsilon \big[\Big)=
{\mathscr{U}}_\varepsilon \sqcup {\mathscr{S}}_\varepsilon ,
\end{equation}
where ${\mathscr{U}}_\varepsilon $, ${\mathscr{S}}_\varepsilon $ are open, connected and mutually
disjoint. Let ${\mathscr{U}}_\varepsilon $ be the bounded component and ${\mathscr{S}}_\varepsilon $
the unbounded one. Again,
\begin{equation}\label{prep.11}
\overline{\mathscr{U}}_\varepsilon \cap \overline{\mathscr{S}}_\varepsilon =\{ 0 \}.
\end{equation}

\par\smallskip 
In an ${\mathcal{O}}\big(\sqrt{\varepsilon}\,\big)$-neighborhood of $0$, we write
$x=\sqrt{\varepsilon} \,\widetilde{x}$ and
\begin{equation}\label{prep.12}\begin{aligned}
V_\varepsilon (x)-E_\varepsilon &=V(\sqrt{\varepsilon}\,\widetilde{x})+\varepsilon
(\chi(\alpha \widetilde{x},0)-\chi(0,0))\\
&=\varepsilon \left(\frac{V(\sqrt{\varepsilon}\,\widetilde{x})}{\varepsilon }
+\chi(\alpha \widetilde{x},0)-\chi(0,0)
\right).
\end{aligned}
\end{equation}
Thus, with $V_0(\widetilde{x})=\big\langle\frac{1}{2}V''(0)\widetilde{x}, \widetilde{x}\big\rangle$,
\begin{equation}\label{prep.13}
V_\varepsilon (x)-E_\varepsilon =\varepsilon \left( V_0(\widetilde{x})+{\cal
    O}(\sqrt{\varepsilon }+\alpha ^2)\widetilde{x}^2 \right)\hbox{ in
}C^\infty .
\end{equation}
Here, we may assume (cf.\ (\ref{esc.1})) that
\begin{equation}\label{prep.14}
V_0(\widetilde{x})=q(\widetilde{x}')-\widetilde{x}_n^2,
\end{equation}
where $q(\widetilde{x}')$ is a positive definite quadratic form and
(\ref{prep.13}) gives
\begin{equation}\label{prep.15}
V_\varepsilon (x)-E_\varepsilon =\varepsilon
\left(q(\widetilde{x}')-\widetilde{x}_n^2+{\mathcal{O}}(\sqrt{\varepsilon}\,+\alpha ^2)\widetilde{x}^2 \right).
\end{equation}

\par\smallskip  
For $-{\mathcal{O}}(\varepsilon )\le E\le E_\varepsilon $ we have
\begin{equation}\label{prep.16}
V_\varepsilon ^{-1}(]-\infty ,E[)={\mathscr{U}}_\varepsilon (E)\cup {\mathscr{S}}_\varepsilon (E),
\end{equation}
where ${\mathscr{U}}_\varepsilon (E)\subset {\mathscr{U}}_\varepsilon ={\mathscr{U}}_\varepsilon (E_\varepsilon )$,
${\mathscr{S}}_\varepsilon (E)\subset {\mathscr{S}}_\varepsilon ={\mathscr{S}}_\varepsilon (E_\varepsilon )$.

\par\smallskip 
Orient the $\widetilde{x}_n$-axis so that $\widetilde{x}_n<0$ in
${\mathscr{U}}_\varepsilon $ and $\widetilde{x}_n>0$ in ${\mathscr{S}}_\varepsilon $. Write
$E=E_\varepsilon -\varepsilon F$, $0\le F\le {\mathcal{O}}(1)$. Then on $\partial
{\mathscr{U}}_\varepsilon (E)\cup \partial {\mathscr{S}}_\varepsilon (E)$, we have by
(\ref{prep.15}),
$$
-F=-\left(1+{\mathcal{O}}(\sqrt{\varepsilon}\,+\alpha ^2)\right)\widetilde{x}_n^2+
  \left(1+{\mathcal{O}}(\sqrt{\varepsilon}\,+\alpha ^2) \right) q(\widetilde{x}'),
    $$
    \begin{equation}\label{prep.17}
      \widetilde{x}_n=
      \pm \left(1+{\mathcal{O}}(\sqrt{\varepsilon}\,+\alpha ^2)\right)
      \left(F+\left(1+{\mathcal{O}}(\sqrt{\varepsilon}\,+\alpha ^2) \right)
          q(\widetilde{x}') \right)^{\frac12}.
    \end{equation}
    Here the plus and minus sign give the local parametrizations of
    $\partial {\mathscr{S}}_\varepsilon (E)$ and $\partial {\mathscr{U}}_\varepsilon (E)$
    respectively. In the original coordinates $x=\sqrt{\varepsilon}\,\widetilde{x}$, this gives
\begin{equation}\label{prep.17.5}
  x_n=
  \pm \left(1+{\mathcal{O}}(\sqrt{\varepsilon}\,+\alpha ^2)\right)
  \left(\varepsilon F+\left(1+{\mathcal{O}}(\sqrt{\varepsilon}\,+\alpha ^2)\right) 
    q(x') \right)^{\frac12}
    \end{equation}
    which is a detailed description of $\partial {\mathscr{U}}_\varepsilon (E)$ and
    $\partial {\mathscr{S}}_\varepsilon (E)$ in any ${\mathcal{O}}(\sqrt{\varepsilon}\,)$-neighborhood
    of $0$. The two sets come closest to each other when $|x'|\ll
    \sqrt{\varepsilon }$ and the distance is
    \begin{equation}\label{prep.19}
2 \left(1+{\mathcal{O}}(\sqrt{\varepsilon}\,+\alpha ^2) \right)(\varepsilon F)^{\frac12}.
    \end{equation}
    Moreover, for $E=E_\varepsilon -\varepsilon F$,
    \begin{equation}\label{prep.20}
     \hskip-7pt \underset{\scriptstyle x\in \mathrm{neigh\,}(0)\cap {\mathscr{S}}_\varepsilon (E)}{\inf\hskip3pt x_n} \,\,
      -\,\,\underset{\scriptstyle x\in \mathrm{neigh\,}(0)\cap {\mathscr{U}}_\varepsilon (E)}{\sup\hskip3pt x_n}
      =2\left(1+{\mathcal{O}}(\sqrt{\varepsilon}\,+\alpha ^2) \right) 
      (\varepsilon F)^{\frac12}. 
    \end{equation}

\par\smallskip 
In the following, we assume that
\begin{equation}\label{prep.20.5}
E=E_\varepsilon -\varepsilon F,\qquad  \frac{1}{{\mathcal{O}}(1)}\le F\le {\mathcal{O}}(1),
\end{equation}
and we shall define two reference operators $P_\varepsilon
^{\mathrm{int}}$, $P_\varepsilon ^{\mathrm{ext}}$ by ``filling the sea''
and ``filling the well'' respectively up to a suitable level. 
Introduce the metric
\begin{equation}\label{prep.21}
\frac{dx^2}{\varepsilon +x^2}
\end{equation}
and let $d_\varepsilon $ be the corresponding
distance. When (\ref{prep.20.5}) holds, we see that
\begin{equation}\label{prep.22}
d_\varepsilon ({\mathscr{U}}_\varepsilon (E),{\mathscr{S}}_\varepsilon (E))\asymp 1.
\end{equation}
If $E<E'=E_\varepsilon -\varepsilon F'$, $\,\frac{1}{\mathcal{O}(1)}\le F'\le
{\mathcal{O}}(1)$, we have
\begin{equation}\label{prep.23}
{\mathscr{U}}_\varepsilon (E')\subset B_{d_\varepsilon }({\mathscr{U}}_\varepsilon (E),{\mathbf{r}}),\quad
{\mathbf{r}}={\mathbf{r}}(E,E',\varepsilon )>0,
\end{equation}
where ${\mathbf{r}}\longrightarrow 0^+$ when $\frac{E'-E}{\varepsilon} \longrightarrow 0$. We have the same
inclusions after replacing ${\mathscr{U}}_\varepsilon $ with ${\mathscr{S}}_\varepsilon $.

\par\smallskip
To $d_\varepsilon $ we can associate the symbol classes $S_\varepsilon (\mathbb{R}^{2n};m)$
given by Definition \ref{APepsilon.2}. Precisely, the function $a=a(x)$ independent of $\xi$ belongs to 
$S_\varepsilon (\mathbb{R}^{2n};m)$ if
\begin{equation}\label{prep.24}
\partial ^\alpha a(x)={\mathcal{O}}(1)m(x)R_\varepsilon(x)^{-|\alpha|},\quad 
\forall \alpha \in {\mathbb{N}}^n.
\end{equation}
Here $R_\varepsilon$ is given by \eqref{prep.37} and  $0<m\in C^\infty ({\mathbb{R}}^n) $ is an order function
independent of $\xi$
(see Definition \ref{APepsilon.1} (b)).

\begin{lemma}\label{prep1}
For every $E'=E_\varepsilon -\varepsilon F'$ with $F-F'\asymp 1$ small, we
can find ${\mathbf{r}}>0$, tending to $0$ when $F-F'\longrightarrow 0$, and $0\le W\in C^\infty
({\mathbb{R}}^n)$ such that
\begin{equation}\label{prep.26}
W\in S_\varepsilon (\mathbb{R}^n;r_\varepsilon^2),
\end{equation}
\begin{equation}\label{prep.27}
\mathrm{supp\,}W\subset B_{d_\varepsilon }({\mathscr{S}}_\varepsilon (E),{\mathbf{r}}),
\end{equation}
\begin{equation}\label{prep.28}
V_\varepsilon +W\ge E'+\frac{r_\varepsilon^2}{C}\,\hbox{ in }\,{\mathbb{R}}^n
\setminus B_{d_\varepsilon}({\mathscr{U}}_\varepsilon (E),{\mathbf{r}}).
\end{equation}
Here the scale function $r_\varepsilon$ is given by \eqref{prep.37}.
\end{lemma}
\begin{proof} This can be done in quite a standard way,
  using partitions of unity, adapted to the metric. See e.g.\
  \cite[Remarque 1.3, p.\ 9]{HeSj86}.
\end{proof}

\par\smallskip  
Put
\begin{equation}\label{prep.29}
P_\varepsilon ^{\mathrm{int}}=P_\varepsilon +W, \quad V_\varepsilon^{\mathrm{int}}=V_\varepsilon+W.
\end{equation}

We next turn to the definition of $P_\varepsilon ^{\mathrm{ext}}$ by
means of ``filling the well''. For technical reasons, we want the
perturbation to be of trace class.

\par\smallskip  
The inequality $e^{-t}+t\ge 1$ for $t\ge 0$ implies that
\begin{equation}\label{prep.30}
\beta e^{-\frac{\xi ^2}{2\beta}}+\frac{\xi^2}{2}\ge \beta 
\end{equation}
for every $\beta >0$.
\begin{lemma}\label{prep2}
For every $E'=E_\varepsilon -\varepsilon F'$ with $F-F'\asymp 1$ small, we
can find ${\mathbf{r}}>0$, tending to $0$ when $F-F'\longrightarrow 0$, and $0\le \beta \in C_0^\infty
({\mathbb{R}}^n)$ such that
\begin{equation}\label{prep.31}
\beta \in S_\varepsilon (\mathbb{R}^{2n}; r_\varepsilon^2),
\end{equation}
\begin{equation}\label{prep.32}
\mathrm{supp\,}\beta \subset B_{d_\varepsilon }({\mathscr{U}}_\varepsilon (E),{\mathbf{r}}),
\end{equation}
\begin{equation}\label{prep.33}
V_\varepsilon +\beta \ge E'\hbox{ in }{\mathbb{R}}^n\setminus B_{d_\varepsilon
}({\mathscr{S}}_\varepsilon (E),{\mathbf{r}}).
\end{equation}
\end{lemma}
\begin{proof}
This can be done in quite a standard way,
  using partitions of unity, adapted to the metric. See e.g.\
  \cite[Remarque 1.3, p.\ 9]{HeSj86}.
\end{proof}

\par\smallskip  
From (\ref{prep.30}), (\ref{prep.33}) we see that
\begin{equation}\label{prep.34}
p_\varepsilon (x,\xi )+\beta (x)e^{-\frac{\xi^2}{2\beta (x)}}\ge \frac{\xi^2}{2}
+V_\varepsilon (x)+\beta (x)\ge E'\end{equation}
in ${\mathbb{R}}^n\setminus B_{d_\varepsilon }({\mathscr{S}}_\varepsilon (E),{\mathbf{r}})$.

\par\smallskip  
We can arrange so that
\begin{equation}\label{prep.35}
\beta \ge \frac{r_\varepsilon^2}{{\mathcal{O}}(1)}\hbox{ in }B_{d_\varepsilon}
\left({\mathscr{U}}_\varepsilon (E),\frac{3{\mathbf{r}}}{4}\right)
\end{equation}
and
\begin{equation}\label{prep.36}
V_\varepsilon \ge E' \hbox{ in }B_{d_\varepsilon }({\mathscr{U}}_\varepsilon(E),{\mathbf{r}})
\setminus B_{d_\varepsilon }({\mathscr{U}}_\varepsilon (E),\frac{{\mathbf{r}}}{2}).
\end{equation}

\par\smallskip  
Let $\chi _{{\mathscr{U}}_\varepsilon }\in C_0^\infty 
\left(B_{d_\varepsilon}({\mathscr{U}}_\varepsilon (E),
\frac{3{\mathbf{r}}}{4});[0,1] \right)$ be of class
${S}_\varepsilon (\mathbb{R}^n;1)$ and equal to one on 
$B_{d_\varepsilon }\left( {\mathscr{U}}_\varepsilon (E),\frac{{\mathbf{r}}}{2} \right)$.

\begin{prop}\label{prep3}
  We have
$$
\beta \exp\Big({-\frac{\xi ^2}{2\beta}}\Big)\in S_\varepsilon\left (B_{d_\varepsilon }
\Big({\mathscr{U}}_\varepsilon(E),\frac{3{\mathbf{r}}}{4}\Big)\times {\mathbb{R}}^n;r_\varepsilon ^2\right).
$$
Here $S_\varepsilon(\bullet;m)$ is given in Definition \ref{APepsilon.2} (c).
\end{prop}

\begin{proof}
We write
\begin{equation}\label{prep.39}
\exp\Big({-\frac{\xi^2}{2\beta (x)}}\Big)=\exp\Big({-\frac{\widetilde{r}_\varepsilon ^2}
{2\beta}}\Big)\exp\Big({\frac{r_\varepsilon ^2}{2\beta}}\Big).
\end{equation}
Here $\frac{r_\varepsilon ^2}{2\beta}\asymp 1$ so $\exp\big({\frac{r_\varepsilon ^2}
{2\beta}}\big)\in {S}_\varepsilon (1)$. This factor does not depend on $\xi $. As for
the first factor in (\ref{prep.39}), we notice that
\begin{equation}\label{prep.40}
0<\frac{\widetilde{r}_\varepsilon ^2}{2\beta }\in {S}_\varepsilon
\left(\left(\frac{\widetilde{r}_\varepsilon }{r_\varepsilon } \right)^2 \right)
\end{equation}
is elliptic. For $\alpha ,\gamma \in {\mathbb{N}}^n$, $\partial _x^\alpha
\partial _\xi ^\gamma\Big(\exp\big({-\frac{\widetilde{r}_\varepsilon ^2}{2\beta}}\big)\Big)$ is a
finite linear combination of terms
\begin{equation}\label{prep.41}
\left(\partial _x^{\alpha _1}\partial _\xi ^{\gamma
  _1}\Big(\frac{\widetilde{r}_\varepsilon ^2}{2\beta}\Big)\ldots
\partial _x^{\alpha _k}\partial _\xi ^{\gamma_k}\Big(\frac{\widetilde{r}_\varepsilon ^2}{2\beta}\Big)\right)
\exp\Big({-\frac{\widetilde{r}_\varepsilon^2}{2\beta}}\Big)
\end{equation}
with $\alpha _1+...+\alpha _k=\alpha $, $\gamma _1+...+\gamma _k=\gamma$, 
$(\alpha _k,\gamma _k)\ne (0,0)$. This term is
$$
={\mathcal{O}}(1)\left(\frac{\widetilde{r}_\varepsilon ^2}{\beta }\right)^k
\exp\Big({-\frac{\widetilde{r}_\varepsilon^2}{2\beta}}\Big)
\widetilde{r}_\varepsilon^{-|\gamma |}R_\varepsilon ^{-|\alpha |}
={\mathcal{O}}(1) \widetilde{r}_\varepsilon^{-|\gamma |}R_\varepsilon^{-|\alpha |},
$$
so $\exp\Big({-\frac{\widetilde{r}_\varepsilon^2}{2\beta}}\Big)\in {S}_\varepsilon (1)$ 
and the proposition follows.
\end{proof}

\par\smallskip  
Using that $\exp\Big({-\frac{\widetilde{r}_\varepsilon^2}{2\beta}}\Big)={\mathcal O}
\left(\Big(\frac{\widetilde{r}_\varepsilon}{R_\varepsilon}\Big)^{-N}\right)$ for every $N\ge 0$, we
can strengthen the conclusion in the proposition to
\begin{equation}\label{prep.42}
\beta \exp\Big({-\frac{\xi ^2}{2\beta}}\Big)\in {S}_\varepsilon \left(B_{d_\varepsilon }
\Big({\mathscr{U}}_\varepsilon (E),\frac{3{\mathbf{r}}}{4}\Big)\times {\mathbb{R}}^n;\;
r_\varepsilon ^2 \Big(\frac{r_\varepsilon}{\widetilde{r}_\varepsilon}\Big)^N\right),
\end{equation}
for every $N\ge 0$. Using also the properties of $\chi_{{\mathscr{U}}_\varepsilon}$, we get 
\begin{equation}\label{prep.43}
  \chi _{{\mathscr{U}}_\varepsilon }^2\beta \exp\Big({-\frac{\xi ^2}{2\beta}}\Big)\in
  {S}_\varepsilon \left({\mathbb{R}}^{2n}; \, r_\varepsilon ^2
  \Big(\frac{r_\varepsilon}{\widetilde{r}_\varepsilon}\Big)^N\right),
  \qquad \forall N\ge 0.
\end{equation}
Here $R_\varepsilon \asymp r_\varepsilon $ over
$\mathrm{supp\,}\chi _{{\mathscr{U}}_\varepsilon }$. We define
\begin{equation}\label{prep.44}
p_\varepsilon ^{\mathrm{ext}}(x,\xi )=p_\varepsilon (x,\xi )+\chi
_{{\mathscr{U}}_\varepsilon }(x)^2\beta (x)\exp\Big({-\frac{\xi ^2}{2\beta (x)}}\Big).
\end{equation}

By direct checking,
\begin{equation}\label{prep.45}
p_\varepsilon \in {S}_\varepsilon (\Lambda_G;\widetilde{r}_\varepsilon ^2),
\end{equation}
and by using also (\ref{prep.43}),
\begin{equation}\label{prep.46}
p_\varepsilon ^{\mathrm{ext}}\in {S}_\varepsilon (\Lambda_G;\widetilde{r}_\varepsilon ^2).
\end{equation}

\section{Study of $P_\varepsilon^{\mathrm{int}}$}\label{pint}
\setcounter{equation}{0}

Recall the definition of $P_\varepsilon ^{\mathrm{int}}$ in
(\ref{prep.29}) and Lemma \ref{prep1}, where $E$ is chosen as in
(\ref{prep.20.5}). We make the assumption (\ref{dil.22}):
$$
\varepsilon \ge h^{\frac{1}{2}-\alpha _0} \hbox{ for some fixed }\alpha _0>0.
$$
With $E'<E$, $\frac{\varepsilon}{{\mathcal{O}}(1)}\le E-E'\ll \varepsilon $ as in Lemma
\ref{prep1}, we know that the self-adjoint operator $P_\varepsilon
^{\mathrm{int}}$ has purely discrete spectrum in $]-\infty ,E'[$.

\par\smallskip 
Recall here that $0$ is a non-degenerate saddle point for
$V_\varepsilon $, with critical value $E_\varepsilon =\varepsilon \chi (0,0)$
(cf.\ (\ref{prep.9})). From  (\ref{prep.6}) we see that $(0,0)$ is a
non-degenerate saddle point of $p_\varepsilon (x,\xi ) =p(x,\xi )+\chi _\varepsilon
(x,\xi )$ with the same critical value $E_\varepsilon $. The discussion
of wells and seas in Section \ref{prep} can be lifted in a straight
forward way from ${\mathbb{R}}^n$ to $T^*{\mathbb{R}}^n$.
For $-{\mathcal{O}}(\varepsilon )\le \widetilde{E}\le E_\varepsilon $ 
we have (cf.\ (\ref{prep.16}) that
\begin{equation}\label{pint.1}
p_\varepsilon ^{-1}(]-\infty ,\widetilde{E}[)={\widehat{\mathscr{U}\,}}{{\hskip-3pt}_{\varepsilon}}(\widetilde{E}) \cup
\widehat{\mathscr{S}\,}{{\hskip-3pt}_{\varepsilon}} (\widetilde{E}),
\end{equation}
where ${\widehat{\mathscr{U}\,}}{{\hskip-3pt}_{\varepsilon}} (\widetilde{E})$, 
$\widehat{\mathscr{S}\,}{{\hskip-3pt}_{\varepsilon}}(\widetilde{E})$ are open,
${\widehat{\mathscr{U}\,}}{{\hskip-3pt}_{\varepsilon}} (\widetilde{E})$ is bounded and (cf.\ (\ref{prep.16}))
\begin{equation}\label{pint.'2}
  \pi _x\left({\widehat{\mathscr{U}\,}}{{\hskip-3pt}_{\varepsilon}} (\widetilde{E}) \right)={\mathscr{U}}_\varepsilon (\widetilde{E}),
  \quad
  \pi _x\left(\widehat{\mathscr{S}\,}{{\hskip-3pt}_{\varepsilon}}(\widetilde{E}) \right)={\mathscr{S}}_\varepsilon (\widetilde{E}).
\end{equation}
Here $\pi _x:T^*{\mathbb{R}}^n\longrightarrow {\mathbb{R}}^n$ is the standard base space
projection. We write
$$
{\widehat{\mathscr{U}\,}}{{\hskip-3pt}_{\varepsilon}} ={\widehat{\mathscr{U}\,}}{{\hskip-3pt}_{\varepsilon}} (E_\varepsilon ),
\quad
\widehat{\mathscr{S}\,}{{\hskip-3pt}_{\varepsilon}}=\widehat{\mathscr{S}\,}{{\hskip-3pt}_{\varepsilon}}(E_\varepsilon ).
$$
When $E'$ comes close to $E$, we know by Lemma \ref{prep1} that
$\mathrm{supp\,}W$ is close to ${\mathscr{S}}_\varepsilon (E)$ and in particular
disjoint from $\pi _x\overline{\widehat{\mathscr{U}\,}}{{\hskip-3pt}_{\varepsilon}}
=\overline{\mathscr{U}}_\varepsilon  $. (Cf.\ (\ref{prep.10}).)

The eigenvalues of $P_\varepsilon ^{\mathrm{int}}$ distribute according
to the semi-classical Weyl law:
\begin{prop}\label{pint1'} For every fixed $0<\varepsilon \ll 1$
  and $-{\mathcal{O}}(\varepsilon )\le a < b\le E$, we have
  \begin{equation}\label{pint.3}
\hskip-5pt\# \left(\sigma (P_\varepsilon ^{\mathrm{int}})\cap [a,b] \right)=
    \frac{1}{(2\pi h)^n}\left(\mathrm{vol\,}\left( p_\varepsilon
        ^{-1}([a,b])\cap {\widehat{\mathscr{U}\,}}{{\hskip-3pt}_{\varepsilon}}  \right) +\varepsilon o(1) \right),
  \end{equation}
  when $h\longrightarrow 0$, uniformly in $a,b$.
\end{prop}

In the remainder of the main text we now abondon \eqref{dil.22} and adopt the assumption of Theorem \ref{int1},
 namely that $0<\delta\ll 1$, $0<\varepsilon\leq \varepsilon(\delta)$ and $0<h\leq h(\delta,\varepsilon)$.

\par\medskip
Since $\varepsilon $ is fixed, this is the standard result. Having no
uniformity in $\varepsilon $ we are free to write the remainder as
$\varepsilon o(1)$ instead of $o(1)$.

We write
\begin{equation}\label{pint.4}
\mathrm{vol\,}\Big(p_\varepsilon ^{-1}([a,b])\Big)=\omega _\varepsilon (b)-\omega
_\varepsilon (a),
\end{equation}
where
\begin{equation}\label{pint.5}
\omega _\varepsilon (a)=\mathrm{vol\,}\left({{p_\varepsilon ^{-1}(]-\infty
    ,a])}_\vert}_{{\widehat{\mathscr{U}\,}}{{\hskip-3pt}_{\varepsilon}} }\right).
\end{equation}
In Appendix \ref{app}, we show that
\begin{equation}\label{pint.6}
\omega (a)-{\mathcal{O}}(\varepsilon ^2)\le \omega _\varepsilon (a)\le \omega
(a). 
\end{equation}
Here $\omega (E)$ is the $C^1$ function already defined by
(\ref{int.13}) or (\ref{int.14}) when $E\le 0$, extended to $0<E\ll 1$
by replacing ${\mathscr{U}}_0$ there by the set $\widetilde{\mathscr{U}}_0$, defined in
Appendix \ref{app}.

\par\smallskip\noindent
We end the section with some remarks about the resolvent $(P_\varepsilon
^{\mathrm{int}}-z)^{-1}$ when
\begin{equation}\label{pint.6.5}
-{\mathcal{O}}(\varepsilon )<\Re z\le E,\quad \Im z={\mathcal{O}}(\varepsilon ),
\end{equation}
and $z\not\in \sigma (P_\varepsilon ^{\mathrm{int}})$. Recalling
(\ref{prep.36}) with $\mathbf{r}$ and $\frac{E'-E}{\varepsilon}$ small, 
we choose {the cutoff function}
$\widetilde{\chi }_{{\mathscr{U}}_\varepsilon }\in C_0^\infty (B_{d_\varepsilon}
({\mathscr{U}}_\varepsilon (E), \frac{3\mathbf{r}}{4});[0,1])$ of class 
${S}_\varepsilon (\mathbb{R}^{n};{r}_\varepsilon^2)$ in the sense of Definition \ref{APepsilon.2}
(cf.\ (\ref{prep.33}) and the slightly different definition of $\chi
_{{\mathscr{U}}_\varepsilon }$ after (\ref{prep.36})) such that
\begin{equation}\label{pint.7}
V_\varepsilon^{\mathrm{int}} +\widetilde{\chi }_{{\mathscr{U}}_\varepsilon }- E'\ge 
\frac{r_\varepsilon^2}{{\mathcal{O}}(1)}\hbox{ on }{\mathbb{R}}^n.
\end{equation}
The scale $r_\varepsilon$ is given by \eqref{prep.37}.
Notice that (\ref{pint.7}) remains valid if we increase $E'$ by
$\frac{\varepsilon}{{\mathcal{O}}(1)}$.
To shorten the notation we write
\begin{equation}\label{pint.8}
Q=P_\varepsilon ^{\mathrm{int}},\qquad 
\widetilde{Q}=P_\varepsilon^{\mathrm{int}}+\widetilde{\chi }_{{\mathscr{U}}_\varepsilon }.
\end{equation}
More explicitly (cf. \eqref{prep.29}),
\begin{equation*}
\widetilde{Q}=-h^2\Delta +(V+W+\widetilde{\chi }_{{\mathscr{U}}_\varepsilon })(x)+
\varepsilon \chi (\varepsilon ^{-\frac12}(x,hD_x)),
\end{equation*}
with symbol
\begin{equation}\label{pint.11}
\widetilde{q}(x,\xi )=\xi ^2 +\Big(V+W+\widetilde{\chi }_{{\mathscr{U}}_\varepsilon }\Big)(x)+
\varepsilon \chi\Big(\frac{x}{\sqrt{\varepsilon}},\frac{\xi}{\sqrt{\varepsilon}}\Big),
\end{equation}
belonging to $S_{\varepsilon}(\mathbb{R}^{2n};\widetilde{r}_\varepsilon^2)$,
 see \eqref{prep.37}, Definition \ref{APepsilon.2}.

\par\smallskip
As in {Appendix} \ref{dil} we put
$$
x=\mu\, \widetilde{x},\qquad \mu =\sqrt{\varepsilon },\qquad hD_x=\mu
\widetilde{h}D_{\widetilde{x}},\qquad \widetilde{h}=\frac{h}{\mu ^2}, 
$$
and get
\begin{equation}\label{pint.9}
  \frac{1}{\varepsilon }\widetilde{Q}=\left(-\widetilde{h}^2\Delta_{\widetilde x} +
  \chi (\widetilde{x},\widetilde{h}D_{\widetilde{x}})\right)
 +\frac{1}{\varepsilon }(V+W+\widetilde{\chi }_{{\mathscr{U}}_\varepsilon
  })(\mu\,\widetilde{x}).
\end{equation}
By the sharp G\aa{}rding inequality and (\ref{prep.5}) for
$\varepsilon =1$, the first of the two terms in the
right hand side is
$\ge \chi (\widetilde{x},0)-{\mathcal{O}}(\widetilde{h})$ in the operator
sense, hence
$$
\frac{1}{\varepsilon }\widetilde{Q}\ge -{\mathcal O}(\widetilde{h})
+\frac{1}{\varepsilon }\inf_x \Big(V_\varepsilon(x)+W(x)
+\widetilde{\chi }_{{\mathscr{U}}_\varepsilon }(x)\Big)\ge 
\frac{1}{\varepsilon }E'-{\mathcal{O}}(\widetilde{h}),
$$
where we used (\ref{pint.7}) in the last step. Thus,
\begin{equation}\label{pint.10}
\widetilde{Q}\ge E'-{\mathcal{O}}(h).
\end{equation}

\par\smallskip 
Now restrict the attention to a domain of the form
(\ref{pint.6.5}). From (\ref{pint.11}) we see that the symbol
$\widetilde{q}(x,\xi )-z$ belongs to $S_{\varepsilon}(\mathbb{R}^{2n};\widetilde{r}_\varepsilon^2)$
is elliptic in that space. See (\ref{prep.37}) and Definition \ref{APepsilon.2}. 
Consequently, with $\widetilde{Q}=\widetilde{q}(\mu
(\widetilde{x},\widetilde{h}D_{\widetilde{x}}))$, the symbol of
$\varepsilon ^{-1}(\widetilde{Q}-z)$ is equal to
$$
\widetilde{\xi }^2+\frac{1}{\varepsilon }(V+W+\widetilde{\chi }_{{\mathscr{U}}_\varepsilon
})(\mu \widetilde{x})+\chi (\widetilde{x},\widetilde{\xi})-\frac{z}{\varepsilon }
$$
and it is an elliptic element of 
$S_{\varepsilon,\mu}(\mathbb{R}^{2n};\widetilde{r}^2_{\varepsilon,\mu})$, 
where the symbol space is defined with
respect to the scales $r_{\varepsilon,\mu} (\widetilde{x})$, $R_{\varepsilon,\mu} (\widetilde{x})$,
$\widetilde r_{\varepsilon,\mu}(x,\xi)$ in (\ref{dil.26}). 
As in \cite[{Chapter 8, Proposition 8.6}]{DiSj99} we know that the inverse
$(\varepsilon ^{-1}(\widetilde{Q}-z))^{-1}$ is an 
$\widetilde{h}$-pseudo-differential operator
with symbol in the space ${S}_1(\mathbb{R}^{2n};\widetilde{r}_{\varepsilon,\mu}^{-2})$, 
where the subscript 1 indicates that we
use the constant scales $r=1$, $R=1$. Back in the original variable,
we get $(\widetilde{Q}-z)^{-1}=\mathrm{Op}(r)$ as an $h$-pseudo-differential operator with
symbol $r\in \widehat{S}_{\varepsilon,\mu} (\mathbb{R}^{2n};\widetilde{r}^{-2}_{\varepsilon})$,
 meaning that
$$
\partial _x^\alpha \partial _\xi ^\beta r={\mathcal{O}}(1)
\widetilde{r}_{\varepsilon}(x,\xi)^{-2}\mu ^{-|\alpha |-|\beta |}.
$$
Also,
$$
r\equiv \frac{1}{\widetilde{q}(x,\xi )-z}\ \mathrm{mod\,}\frac{h}{\mu^2}\widehat{S}_{\varepsilon,\mu} 
(\mathbb{R}^{2n};\widetilde{r}_{\varepsilon}^{-2}).
$$
We get exponentially weighted estimates for the resolvent of
$\widetilde{Q}$ in the following way: The symbol $\widetilde{q}(x,\xi)-z$ 
can be extended holomorphically in $\xi $ to a $\frac{\sqrt{\varepsilon}}{{\mathcal{O}}(1)}$ 
-neighborhood of ${\mathbb{R}}^n_\xi $, the extended symbol
still belongs to $\widehat{S}_{\varepsilon,\mu} (\mathbb{R}^{2n};\widetilde{r}_\varepsilon^2)$ 
in the natural sense and it is still elliptic. By the Kuranishi trick  we then see
that if $f\in C^\infty ({\mathbb{R}}^n)$ is bounded, $|\nabla f|\le
\frac{\sqrt{\varepsilon }}{{\mathcal{O}}(1)}$, $\nabla f\in \sqrt{\varepsilon }
S_{\varepsilon,\mu}(1)$, then
$$
e^{\frac{f}{h}}(\widetilde{Q}-z)e^{-\frac{f}{h}}=e^{\frac{f}{h}}\widetilde{Q}e^{-\frac{f}{h}}-z
$$
is an elliptic $h$-pseudo-differential operator with symbol in the class
$\widehat{S}_{\varepsilon,\mu} (\mathbb{R}^{2n};\widetilde{r}^2_\varepsilon )$. 
The inverse is an $h$-pseudo-differential operator with
symbol in the class $\widehat{S}_{\varepsilon,\mu} (\mathbb{R}^{2n};\widetilde{r}_\varepsilon^{-2})$, of norm $\le
{\mathcal{O}}(\frac{1}{\varepsilon})$. Now this inverse is equal to
$e^{\frac{f}{h}}(\widetilde{Q}-z)^{-1}e^{-\frac{f}{h}}$, so we conclude that
\begin{equation}\label{pint.12}
\begin{aligned}
  e^{\frac{f}{h}}(\widetilde{Q}-z)^{-1}e^{-\frac{f}{h}} & =
  e^{\frac{f}{h}}(P_\varepsilon ^{\mathrm{int}}+\widetilde{\chi }_{{\mathscr{U}}_\varepsilon
  }-z)^{-1}e^{-\frac{f}{h}}\\
  & ={\mathcal{O}}(\frac{1}{\varepsilon}) : L^2\longrightarrow L^2,
\end{aligned}
  \end{equation}
under the above assumptions on $f$ and $z$.

\par\smallskip
Similarly, $\widetilde{\chi }_{{\mathscr{U}}_\varepsilon }$ can be viewed as an
$h$-pseudo-differential operator with symbol in 
$\widehat{S}_{\varepsilon,\mu} (\mathbb{R}^{2n};r_\varepsilon^2)$ and it follows
that $\widetilde{\chi }_{{\mathscr{U}}_\varepsilon }(\widetilde{Q}-z)^{-1}$ and
$(\widetilde{Q}-z)^{-1}\widetilde{\chi }_{{\mathscr{U}}_\varepsilon }$ are
$h$-pseudo-differential operators with symbol in $\widehat{S}_{\varepsilon,\mu}\left(\mathbb{R}^{2n};
\frac{r_\varepsilon^2}{\widetilde{r}_\varepsilon^2} \right)
\subset {\widehat S}_{\varepsilon,\mu}(\mathbb{R}^{2n};1)$. 
We conclude that these operators and their conjugations with $\exp\big({\frac{f}{h}}\big)$
are ${\mathcal{O}}(1):L^2\longrightarrow L^2$.

\par\smallskip 
We next study the resolvent of $Q=P_\varepsilon ^{\mathrm{int}}$ for
$z$ as in (\ref{pint.6.5}). Assume that $z\not\in \sigma (P_\varepsilon
^{\mathrm{int}})$ and let $\delta =\delta (z)$ denote the distance
from $z$ to the spectrum. Recall the telescopic formula
\begin{equation}\label{pint.16}
\begin{aligned}
  (Q-z)^{-1}=&(\widetilde{Q}-z)^{-1}+
  (\widetilde{Q}-z)^{-1} \widetilde{\chi} _{{\mathscr{U}}_\varepsilon } (\widetilde{Q}-z)^{-1}\\ &+
  (\widetilde{Q}-z)^{-1} \widetilde{\chi} _{{\mathscr{U}}_\varepsilon }(Q-z)^{-1}\widetilde{\chi}
  _{{\mathscr{U}}_\varepsilon } (\widetilde{Q}-z)^{-1}.
  \end{aligned}
\end{equation}

\par\smallskip 
For $f$ as above, assume in addition that
\begin{equation}\label{pint.17}
f=\mathrm{Const.}\hbox{ on }\mathrm{supp\,}\widetilde{\chi } _{{\mathscr{U}}_\varepsilon }.
\end{equation}
Now $\widetilde{\chi } _{{\mathscr{U}}_\varepsilon }={\mathcal{O}}(1)$,
$$
e^{\frac{f}{h}}\widetilde{\chi} _{{\mathscr{U}}_\varepsilon }(Q-z)^{-1} 
\widetilde{\chi} _{{\mathscr{U}}_\varepsilon }e^{-\frac{f}{h}}=
\widetilde{\chi} _{{\mathscr{U}}_\varepsilon }(Q-z)^{-1} 
\widetilde{\chi} _{{\mathscr{U}}_\varepsilon }={\mathcal{O}}(\frac{1}{\delta} )
$$
and using also (\ref{pint.12}) and the above remark on the composition
of $\widetilde{\chi }_{{\mathscr{U}}_\varepsilon }$ and the resolvent of
$\widetilde{Q}$, we get by conjugating (\ref{pint.16}):
$$
e^{\frac{f}{h}}(Q-z)^{-1}e^{-\frac{f}{h}}={\mathcal{O}}(\frac{1}{\varepsilon} )+
{\mathcal{O}}(\frac{1}{\delta}  )+
{\mathcal{O}}(\frac{1}{\delta} ),
$$
i.e.
\begin{equation}\label{pint.18}
e^{\frac{f}{h}}(P_\varepsilon ^{\mathrm{int}}-z)^{-1}e^{-\frac{f}{h}}=
{\mathcal{O}}(1)\frac{1}{\delta}:L^2\longrightarrow L^2.
\end{equation}
Summing up, we have:
\begin{prop}\label{pint2}
  Let $f=f_\varepsilon \in C^\infty ({\mathbb{R}}^n;{\mathbb{R}})$ be bounded
  with 
  $$|\nabla f|\le \frac{\sqrt{\varepsilon }}{\mathcal{O}(1)},\qquad \nabla f\in
  \sqrt{\varepsilon }S_{\varepsilon,\mu} (\mathbb{R}^{2n};1).$$
  Then (\ref{pint.18}) holds uniformly for $z$ as in \eqref{pint.6.5}
  with 
  $$\mathrm{dist}(z,\sigma(P_\varepsilon ^\mathrm{int}))\ge\delta >0.$$
\end{prop}

\section{Study of $P_\varepsilon ^{\mathrm{ext}}$}\label{pext}
\setcounter{equation}{0}

We recall the definition of the symbol
\begin{equation}\label{pext.1}
p_\varepsilon ^{\mathrm{ext}}(x,\xi )=p_\varepsilon (x,\xi )+\chi
_{{\mathscr{U}}_\varepsilon }^2(x)\beta (x)e^{-\frac{\xi ^2}{2\beta (x)}}
\end{equation}
in (\ref{prep.44}). With $R_\varepsilon $, $r_\varepsilon $,
$\widetilde{r}_\varepsilon $ defined in (\ref{prep.37}), we see that
\begin{equation}\label{pext.2}
p_\varepsilon ^{\mathrm{ext}}(x,\xi )\in {S}_\varepsilon ({\mathbb{R}}^{2n},\widetilde{r}_\varepsilon ^2).
\end{equation}
(Notice here that in (\ref{prep.43}), we can replace $R_\varepsilon $ by
$r_\varepsilon $ since $\chi _{{\mathscr{U}}_\varepsilon }$ has compact support and
$R_\varepsilon \asymp r_\varepsilon $ on any fixed compact set.)

From Lemma \ref{prep2}, (\ref{prep.34}), (\ref{prep.35}) and
(\ref{pext.1}) we see that if $E'$ is as in the cited lemma, then
\begin{equation}\label{pext.3}
p_\varepsilon ^{\mathrm{ext}}(x,\xi )-E'\ge C^{-1}\widetilde{r}_\varepsilon
^2,\ x\in B_{d_\varepsilon }({\mathbb{R}}^n\setminus {\mathscr{S}}_\varepsilon ,\mathbf{r}),
\end{equation}
where $\mathbf{r}$ is as in the lemma. (Strictly speaking, we apply Lemma
\ref{prep2} with a slightly increased value $E'_{\mathrm{new}}$, where
$E'_{\mathrm{new}}-E'\asymp \frac{\varepsilon}{{\mathcal{O}}(1)}$, or
alternatively we decrease $E'$ in this section with $\frac{\varepsilon}{{\mathcal{O}}(1)}$.)
This means that $p_\varepsilon ^{\mathrm{ext}}-z$ is uniformly elliptic
in $S_{\varepsilon}(\mathbb{R}^{2n};\widetilde{r}_\varepsilon ^2)$ when $x$ varies in 
$B_{d_\varepsilon}({\mathbb{R}}^{2n}\setminus {\mathscr{S}}_\varepsilon ,\mathbf{r})$,
uniformly for
$$
z\in ]-{\mathcal{O}}(\varepsilon ),E'[+i]{-\mathcal O}(\varepsilon ),{\mathcal O}(\varepsilon )[.
$$

Let $G$, $\widetilde{G}$ be the escape functions in Section \ref{ltg}
and recall that $\widetilde{G}$ is an extension of $G$ from a small
neighborhood of $(0,0)$. For simplicity, we drop the tilde in the
following, so that $G$ now denotes the globally defined escape
function. With $\mathbf{r}$ as above, we may arrange so that with
${\mathscr{U}}_\varepsilon $, ${\mathscr{S}}_\varepsilon $ defined after (\ref{prep.16}),
$$
\pi _x(\mathrm{supp\,}G)\subset B_{d_\varepsilon }({\mathbb{R}}^n\setminus
{\mathscr{S}}_\varepsilon ,\mathbf{r}).
$$
From Proposition \ref{ltg2} and (\ref{ltg.2}) we conclude that for
$t>0$ small enough,
$$
p_\varepsilon ^{\mathrm{ext}}(\rho +itH_G(\rho ))-z\in {S}_\varepsilon ({\mathbb{R}}^{2n};\widetilde{r}
_\varepsilon ^2)$$
is a uniformly elliptic symbol on ${\mathbb{R}}^{2n}$ for
\begin{equation}\label{pext.4}
-{\mathcal{O}}(\varepsilon )<\Re z<E',\quad -\frac{t\varepsilon }{{\mathcal{O}}(1)}<\Im
z<{\mathcal{O}}(\varepsilon ).
\end{equation}
Here we replace $E'$ by $\min (E',b\varepsilon )$ where $b$ is given
in Proposition \ref{ltg2}. (In the end we will have $E'\asymp
\frac{\varepsilon}{{\mathcal{O}}(1)}$.)

\par\smallskip 
We now apply { Appendix \ref{dil}}. Let $G_\mu (\widetilde{\alpha })=\mu
^{-2}G(\mu \widetilde{\alpha })$. Recalling that $\Lambda _{tG}$ is
defined by $\Im \alpha =tH_{G}(\Re \alpha )$, we define $\Lambda
_{tG_\mu }$ similarly by $\Im \widetilde{\alpha }=tH_{G_\mu }(\Re
\widetilde{\alpha })$. More explicitly, the latter manifold is given
by
$$
\Im \widetilde{\alpha }_x=t\partial _{\widetilde{\alpha }_\xi }G_\mu
(\Re \widetilde{\alpha }),\ \Im \widetilde{\alpha }_\xi =-t\partial _{\widetilde{\alpha }_x }G_\mu
(\Re \widetilde{\alpha }).
$$
Putting $\alpha =\mu\,\widetilde{\alpha }$, we get
$\alpha \in \Lambda _{tG}$. Thus we have the bijection
$$
\Lambda _{tG_\mu }\ni \widetilde{\alpha }\longmapsto \mu\,\widetilde{\alpha }\in
\Lambda _{tG}.
$$

In { Appendix \ref{dil}}, we defined the FBI-transformations $T$, $T_\mu$ and noted that
$$
T_\mu \widetilde{u}(\widetilde{\alpha};\widetilde{h})=Tu(\alpha ;h)
\quad\hbox{with }\,
u(x)=\mu ^{\frac{n}{2}}\widetilde{u}(\widetilde{x})
$$
for $\alpha\in \Lambda _{tG}$, $\widetilde{\alpha } \in \Lambda_{tG_\mu }$ related by 
$\alpha=\mu \,\widetilde{\alpha }$. See \eqref{dil.29}--\ref{dil.35})

\par\smallskip 
We can define the spaces $H({\Lambda}_{tG_\mu};m_\mu )$ as in
\cite[Chapter 5]{HeSj86} and define the space $H(\Lambda_{tG};m)$ by
requiring that (\ref{dil.36}) holds when 
$m_\mu (\widetilde{\alpha })=m(\mu\,\widetilde{\alpha })$.

\par\smallskip 
We define
\begin{equation}\label{pext.4.5}
  P_\varepsilon ^{\mathrm{ext}}:=P_\varepsilon +\chi _{{\mathscr{U}}_\varepsilon }
  \mathrm{Op}_h\left(\beta e^{-\frac{\xi ^2}{2\beta}} \right)
 \chi _{{\mathscr{U}}_\varepsilon } ,
\end{equation}
where $\mathrm{Op}_h$ denotes the $h$-Weyl quantization. We view
$P_\varepsilon ^{\mathrm{ext}}$ as an $h$-quantization of $p_\varepsilon
^{\mathrm{ext}}$ in (\ref{pext.1}).

\par\smallskip 
The scaling $x=\mu\, \widetilde{x}$ transforms the operator
$\varepsilon ^{-1}(P_\varepsilon ^{\mathrm{ext}}-z)$ into an 
$\widetilde{h}$-pseudo-differential operator of class 
$S({\Lambda}_{tG_\mu};\widetilde{r}_{\varepsilon,\mu}^2)$ which
is uniformly elliptic on $\Lambda _{tG_\mu }$ when $t>0$ is small 
and fixed and $z$ varies in the set (\ref{pext.4}). Consequently, this
operator is bijective with a uniformly bounded inverse 
$H(\Lambda_{tG_\mu};\widetilde{r}_{\varepsilon,\mu} ^2)
\longrightarrow H({\Lambda}_{tG_{\mu}})$. 
This means that
\begin{equation}\label{pext.5}
P_\varepsilon ^{\mathrm{ext}}-z:\, 
H(\Lambda _{tG};\widetilde{r}_\varepsilon^2)\longrightarrow H(\Lambda _{tG})
\end{equation}
is bijective with a uniformly bounded inverse for $z$ in the set
(\ref{pext.4}). Since $\widetilde{r}_\varepsilon ^2\ge \varepsilon $, it
follows that for $z$ in the same set,
\begin{equation}\label{pext.6}
  (P_\varepsilon ^{\mathrm{ext}}-z)^{-1}=
  \begin{cases}
    {\mathcal{O}}(1): H(\Lambda _{tG})\longrightarrow H(\Lambda
    _{tG},\widetilde{r}_\varepsilon ^2),\\
  {\mathcal{O}}\left( \frac{1}{\varepsilon }
\right):\, H(\Lambda _{tG})\longrightarrow H(\Lambda _{tG}). 
\end{cases}
\end{equation}

\section{Study of $P_\varepsilon$}\label{peps}
\setcounter{equation}{0}

We will incorporate ordinary exponentially weighted estimates in the
machinery of \cite{HeSj86} and recall from Chapter 5 in that work that
the spaces $H(\Lambda _{G};m)$ can be defined when $G(\alpha
)-g(\alpha _x)$ is sufficiently small in $S(rR)$ and $g-g_0(\alpha
_x)$ is sufficiently small in 
$\dot{S}^{1,1}({\mathbb{R}}^n)=\dot{S}({\mathbb{R}}^n;R\, \widetilde{r})$ 
and we work with a fixed FBI transform,
depending only on $g_0$. Moreover, when $G=g(\alpha _x)$ is
independent of $\alpha _\xi $, then
\begin{equation}\label{peps.1}
H(\Lambda _g):=H(\Lambda _g;1)=L^2(\Lambda _g;e^{-\frac{2}{h}g(x)}dx).
\end{equation}
(See \cite[Proposition 5.3]{HeSj86}.) When replacing the weight 1 with
suitable weights (like a power of $\widetilde{r}$) we get the
corresponding naturally defined Sobolev spaces.

\par\smallskip 
According to \cite[Proposition 5.7]{HeSj86}, if $\widetilde{G}$ is a
second function with the same structure as $G$ and with the same basic
weight $g_0$, and if $\Lambda _G\le \Lambda _{\widetilde{G}}$ in the
sense that $G\le \widetilde{G}$ and if $\widetilde{m}$ is a second
order function with $m\ge \widetilde{m}$, then
\begin{equation}\label{peps.2}
H(\Lambda _G;m)\subset H(\Lambda _{\widetilde{G}};\widetilde{m})
\end{equation}
and the inclusion map is uniformly bounded.
For more details see Appendix \ref{dil}.

This theory is based on the use of scale functions $R,r,\widetilde{r}$
satisfying (\ref{dil.1}), (\ref{dil.2}), (\ref{dil.3}),
(\ref{dil.5}). However, the dilation in Appendix \ref{dil} allows us to
apply it also in the case of the scales 
$R_\varepsilon, r_\varepsilon, \widetilde{r}_\varepsilon $ that do not satisfy \eqref{dil.5}.

From \cite[{ See the proof of Theorem 8.3.}]{HeSj86}, we can partially extend the estimate
(\ref{pint.18}). Let $f=f(x)$ be small in the space
$\dot{S}(\mathbb{R}^n; R_\varepsilon\,r_\varepsilon)$. Then the theory applies
to
$$
  \widetilde{Q}:=P^{\mathrm{int}}_\varepsilon +\chi _{{\mathscr{U}}_\varepsilon }
  \mathrm{Op}_h\left(\beta e^{-\frac{\xi ^2}{2\beta}} \right)
 \chi _{{\mathscr{U}}_\varepsilon }=: P^{\mathrm{int}}_\varepsilon+\widehat{\chi
 }_{{\mathscr{U}}_\varepsilon },
 $$
 cf.\ (\ref{pext.4.5}), slightly different from ``$\widetilde{Q}$'' in (\ref{pint.8}).
When $z$ belongs to the set
(\ref{pint.6.5}) for $E=\frac{\varepsilon}{C}$ for $C\gg 1$, we see that
$$
(\widetilde{Q}-z)^{-1} ={\mathcal{O}}(1):\ \left\{\begin{aligned}
 H(\Lambda _f;\widetilde{r}_\varepsilon ^{-2}) & \longrightarrow   H(\Lambda _f),\\
 H(\Lambda _f)  & \longrightarrow   H(\Lambda _f;\widetilde{r}_\varepsilon ^2).
\end{aligned}\right.
$$
By the telescopic formula (\ref{pint.16}) for $P_\varepsilon
^{\mathrm{int}}$, with $\widetilde{\chi }_{{\mathscr{U}}_\varepsilon }$ replaced
by $\widehat{\chi }_{{\mathscr{U}}_\varepsilon }$,  we see that if
(\ref{pint.17}) also holds, then
\begin{equation}\label{peps.3}
(P_\varepsilon ^{\mathrm{int}}-z)^{-1}-(\widetilde{Q}-z)^{-1}={\mathcal{O}}(\frac{1}{\delta}):
\left\{\begin{aligned}
H(\Lambda_f) & \longrightarrow  H(\Lambda_f;\widetilde{r}_\varepsilon ^{2}),\\
H(\Lambda_f;\widetilde{r}_\varepsilon ^{-2}) & \longrightarrow  H(\Lambda_f),\\
H(\Lambda_f) & \longrightarrow  H(\Lambda_f).
\end{aligned}\right.
\end{equation}
Since $\widetilde{r}_\varepsilon ^2\ge \varepsilon $, the inclusion maps
$$
H(\Lambda _f,\widetilde{r}_\varepsilon ^2)\longrightarrow H(\Lambda _f) \hbox{ and }
H(\Lambda _f) \longrightarrow H(\Lambda _f,\widetilde{r}_\varepsilon ^{-2})
$$ 
have norms $\le \varepsilon^{-1} $. Estimate (\ref{peps.3}) and the above one for
$(\widetilde{Q}-z)^{-1}$ therefore imply that
\begin{equation}\label{peps.4}
(P_\varepsilon ^{\mathrm{int}}-z)^{-1}={\mathcal{O}}(\frac{1}{\delta}): 
\left\{\begin{aligned}
H(\Lambda_f) & \longrightarrow H(\Lambda_f;\widetilde{r}_\varepsilon ^{2}),\\
H(\Lambda_f; \widetilde{r}_\varepsilon ^{-2}) & \longrightarrow H(\Lambda_f),\\
H(\Lambda_f) & \longrightarrow H(\Lambda_f).
\end{aligned}\right.
\end{equation}

\par\smallskip 
We choose $f$ as above with constant $0$ in (\ref{pint.17}):
\begin{equation}\label{peps.5}
f=0\hbox{ on }\mathrm{supp\,}{\chi}_{{\mathscr{U}}_\varepsilon },
\end{equation}
with $f\le 0$ everywhere and
\begin{equation}\label{peps.6}
f\asymp -R_\varepsilon\,r_\varepsilon  \, \hbox{ on } \, \pi _x\big(\mathrm{supp\,}G\big). 
\end{equation}
This implies that for $t$ small: $\Lambda _f\le \Lambda _{tG}$ and
after a further decrease of $t>0$, that
\begin{equation}\label{peps.7}
W={\mathcal{O}(1)}e^{-\frac{\varepsilon}{\mathcal{O}(h)}} :\ H(\Lambda
_f;\widetilde{r}_\varepsilon ^2)\longrightarrow H(\Lambda _{tG}).
\end{equation}
Combining this with (\ref{peps.4}), (\ref{peps.5}), we get
\begin{equation}\label{peps.8}
W(P_\varepsilon ^{\mathrm{int}}-z)^{-1}\widehat{\chi }_{{\mathscr{U}}_\varepsilon
}={\mathcal{O}}(1)\frac{1}{\delta }e^{-\frac{\varepsilon}{\mathcal{O}(h)}} :\
H(\Lambda _{tG})\longrightarrow H(\Lambda _{tG}).
\end{equation}
Here we also used that
$$
\widehat{\chi }_{{\mathscr{U}}_\varepsilon }={\mathcal{O}}(1):\  H(\Lambda _{f})\longrightarrow H(\Lambda _{tG}), 
$$
in view of (\ref{peps.5}).

\par\smallskip 
In the following, we assume that
\begin{equation}\label{peps.8.5}
\mathrm{dist\,}(z,\sigma (P^{\mathrm{int}}))=\delta \ge h^{N_0}
\end{equation}
for some fixed $N_0>0$. Then the right hand side in (\ref{peps.8}) can
be replaced by ${\mathcal{O}}(1)e^{-\frac{\varepsilon}{{\mathcal{O}}(h)}}$. 

\par\smallskip 
We can now construct a right inverse of $P_\varepsilon -z$.
Let $z$ vary in a set of the form (\ref{pext.4}), now with $0<t\ll 1$
fixed and $E'=\frac{\varepsilon}{C}$ for a fixed sufficiently large $C\gg 1$.
In view of the identity
$$
(P_\varepsilon -z)(P_\varepsilon ^{\mathrm{ext}}-z)^{-1}=1-\widehat{\chi
}_{{\mathscr{U}}_\varepsilon }(P_\varepsilon ^{\mathrm{ext}}-z)^{-1} ,
$$
we try as an approximate right inverse,
\begin{multline}\label{peps.9}
  R_0(z)=(P_\varepsilon ^{\mathrm{ext}}-z)^{-1}+(P_\varepsilon
  ^{\mathrm{int}}-z)^{-1}
  \widehat{\chi }_{{\mathscr{U}}_\varepsilon }(P_\varepsilon ^{\mathrm{ext}}-z)^{-1}\\
  ={\mathcal{O}}\left(\frac{1}{\delta } \right):\ H(\Lambda
  _{tG})\longrightarrow H(\Lambda _{tG},\widetilde{r}_\varepsilon ^2)\cap
  H(\Lambda _{tG}).
\end{multline}
We have
\begin{equation}\label{peps.10}
\begin{aligned}
    (P_\varepsilon -z)R_0(z)&=1-K\ \hbox{ with } \\
    K&=W(P_\varepsilon ^{\mathrm{int}}-z)^{-1}\widehat{\chi}_{{\mathscr{U}}_\varepsilon }(P_\varepsilon ^{\mathrm{ext}}-z)^{-1}.
\end{aligned}
\end{equation}
From (\ref{peps.8}), with the right hand side simplified to ${\mathcal O}(1)e^{-\frac{\varepsilon}{\mathcal{O}(h)}}$, 
we see that
\begin{equation}\label{peps.11}
K={\mathcal{O}}(1)e^{-\frac{\varepsilon}{\mathcal{O}(h)}}:\ H(\Lambda _{tG})\longrightarrow H(\Lambda _{tG}).
\end{equation}
for every small fixed $t>0$. Then for $h>0$ small enough, $1-K$ is
bijective with inverse ${\mathcal{O}}(1)$ and we get the right inverse of
$P_\varepsilon -z$:
\begin{equation}\label{peps.12}
R_0(z)(1-K)^{-1}={\mathcal{O}}\left(\frac{1}{\delta  } \right):\
H(\Lambda _{tG})\rightarrow H(\Lambda _{tG};\widetilde{r}_\varepsilon ^2)\cap
H(\Lambda _{tG}).
\end{equation}
From \cite[{ See the end of proof of Theorem 8.3., page 99}]{HeSj86}
we know that $P_\varepsilon -z: H(\Lambda
_{tG};\widetilde{r}_\varepsilon ^2)\longrightarrow H(\Lambda _{tG})$ is a Fredholm
operator of index $0$ so $R_0(z)(1-K)^{-1}$ is also a left inverse.
\begin{prop}\label{peps1}
Let $0<t\ll 1$ and let $z$ vary in $\{ z\in {\mathbb{C}};\
(\ref{pext.4})\hbox{ and }(\ref{peps.8.5})\hbox{ hold} \}$, where
$E'=\frac{\varepsilon}{C}$, $C\gg 1$. Then for
$h>0$ small enough, $P_\varepsilon -z:\ H(\Lambda_{tG};\widetilde{r}_\varepsilon ^2)
\longrightarrow H(\Lambda _{tG})$ is bijective and
\begin{equation}\label{peps.13}
(P_\varepsilon -z)^{-1}={\mathcal{O}}\left(\frac{1}{\delta }
\right):\ H(\Lambda _{tG})\longrightarrow H(\Lambda
_{tG};\widetilde{r}_\varepsilon ^2)\cap H(\Lambda _{tG}).
\end{equation}
\end{prop}

\par\smallskip 
By a variant of the above arguments, we also get:
\begin{prop}\label{peps2}
For $0<t\ll 1$ we restrict the attention to the region
(\ref{pext.4}). In this region we have a bijection
\begin{equation}\label{peps.14}
b:\sigma (P_\varepsilon ^{\mathrm{int}})\longrightarrow \mathrm{Res\,}(P_\varepsilon )
\end{equation}
such that $b(\mu )-\mu ={\mathcal{O}}(h^{\infty })$\footnote{Strictly
  speaking, to obtain a bijection, we have to modify the bounds in
  (\ref{pext.4}) very slightly, so that no point in
  $\sigma (P_\varepsilon ^{\mathrm{int}})\cup \mathrm{Res\,}(P_\varepsilon
  )$ is too close to the boundary of the region $\Omega $ defined by
  (\ref{pext.4}) and we then get a bijection
  $b:\Omega \cap\sigma (P_\varepsilon ^{\mathrm{int}})\longrightarrow \Omega
  \cap\mathrm{Res\,}(P_\varepsilon )$, when counting the eigenvalues and
  the resonances with their multiplicity.}.
\end{prop}

\begin{proof}
It will be convenient to work with a different approximation of
$(P_\varepsilon -z)^{-1}$. Let $\chi _0\in C_0^\infty ({\mathbb{R}}^n)\cap
{S}_\varepsilon (1)$ have the property that for some small fixed $r$:
$$
\chi _0=\begin{cases}1\hbox{ on }B_{d_\varepsilon }({\mathscr{U}}_\varepsilon ,r),\\
0 \hbox{ on }B_{d_\varepsilon }({\mathscr{S}}_\varepsilon ,r)
\end{cases} .
$$
As a new approximation we take
\begin{equation}\label{peps.15}
  R_0(z)=(P_\varepsilon ^{\mathrm{int}}-z)^{-1}\chi _0+(P_\varepsilon ^{\mathrm{ext}}-z)^{-1}(1-\chi _0),
\end{equation}
which satisfies the estimate
  \begin{equation}\label{peps.15.5}
R_0(z)={\mathcal{O}}(\frac{1}{\delta }):\ H(\Lambda _{tG})\longrightarrow H(\Lambda
_{tG};\widetilde{r}_\varepsilon ^2)\cap H(\Lambda _{tG}).
  \end{equation}

\par\smallskip Then
\begin{equation}\label{peps.16}
\begin{aligned}
(P_\varepsilon -z)R_0(z) &=1-W(P_\varepsilon^{\mathrm{int}}-z)^{-1}\chi _0
-\widehat{\chi }_{{\mathscr{U}}_\varepsilon }
(P_\varepsilon^{\mathrm{ext}}-z)^{-1}(1-\chi _0)\\
&=:1-K,
\end{aligned}
\end{equation}
where (the new) $K$ satisfies (\ref{peps.11}), if we assume that
$\delta \ge h^{N_0}$ for some fixed
$N_0>0$. Then,
\begin{equation}\label{peps.17}
 \begin{aligned}
 (P_\varepsilon -z)^{-1} &=R_0(1-K)^{-1}=R_0+L\\
& ={\mathcal{O}}\left(\frac{1}{\delta} \right):H(\Lambda _{tG})
\longrightarrow H(\Lambda _{tG};\widetilde{r}_\varepsilon
^2)\cap H(\Lambda _{tG})
\end{aligned}
\end{equation}
with
\begin{equation}\label{peps.18}
L={\mathcal{O}}(1)e^{-\frac{\varepsilon}{\mathcal{O}(h)}} :\ H(\Lambda _{tG})
\longrightarrow H(\Lambda _{tG};\widetilde{r}_\varepsilon ^2).
\end{equation}

\par\smallskip 
If $\gamma \subset \Omega $ is a simple closed contour of uniformly
bounded length, along which $\delta \ge h^{N_0}$, we get from
(\ref{peps.15}), (\ref{peps.17}), (\ref{peps.18}):
\begin{equation}\label{peps.19}
\pi _{\varepsilon ,\gamma }=\pi ^{\mathrm{int}}_{\varepsilon ,\gamma }\chi
_0+{\mathcal{O}(1)}e^{-\frac{\varepsilon}{\mathcal{O}(h)}} :\ H(\Lambda _{tG})\longrightarrow
H(\Lambda _{tG}),
\end{equation}
where
\[
  \begin{aligned}
    \pi _{\varepsilon ,\gamma }&=\frac{1}{2\pi i}\int_\gamma (z-P_\varepsilon)^{-1}\,dz,\\
    \pi^{\mathrm{int}} _{\varepsilon ,\gamma }&=\frac{1}{2\pi i}\int_\gamma (z-P^{\mathrm{int}}_\varepsilon )^{-1}\,dz,
  \end{aligned}
\]
are the spectral projections of $P_\varepsilon$, $P_\varepsilon^{\mathrm{int}}$ respectively, 
associated to the part of the spectra inside $\gamma$. By exponentially weighted estimates,
$$
\pi _{\varepsilon ,\gamma }^{\mathrm{int}}\chi _0-\pi _{\varepsilon ,\gamma
}^{\mathrm{int}}
={\mathcal{O}(1)}e^{-\frac{\varepsilon}{\mathcal{O}(h)}}
$$
and we conclude that
$$
\mathrm{rank\,}\pi _{\varepsilon ,\gamma }=\mathrm{rank\,}\pi _{\varepsilon
  ,\gamma }^{\mathrm{int}}.
$$
Hence $P_\varepsilon $ and $P_\varepsilon ^{\mathrm{int}}$ have the same
number of eigenvalues inside $\gamma $. Varying $\gamma $ and $N_0$,
we get the proposition.
\end{proof}

\section{Resolvents of other operators}\label{or}
\setcounter{equation}{0}

We start with the resolvent of $P$ that we realize as an operator from
$H(\Lambda _{tG};\widetilde{r}_\varepsilon ^2)$ to $H(\Lambda _{tG})$ with
the same $G$ as above. Then $P$ has discrete spectrum in the set
(\ref{pext.4}) and the eigenvalues are confined to the lower half
plane. They are the resonances that we want to study. Restricting now
the attention to the set
\begin{equation}\label{or.1}
-{\mathcal{O}}(\varepsilon )<\Re z<-\frac{\varepsilon}{C},\qquad  
-\frac{t}{C}\varepsilon<\Im z<{\mathcal{O}}(\varepsilon ),
\end{equation}
where $C\gg 1$ is large enough, we can adapt the discussion for
$P_\varepsilon $ to $P$. Using Proposition \ref{ltg1} rather than
Proposition \ref{ltg2}, we get
\begin{prop}\label{or1}
Let $0<t\ll 1$ and let $z$ vary in the set (\ref{or.1}). If $|\Im
z|\ge \delta \ge h^{N_0}$ for some fixed $N_0>0$, then $P-z:H(\Lambda
_{tG};\widetilde{r}_\varepsilon ^2)\longrightarrow H(\Lambda _{tG})$ is bijective
and
\begin{equation}\label{or.2}
(P-z)^{-1}={\mathcal{O}}\left( \frac{1}{\delta } \right):\ H(\Lambda
_{tG})\longrightarrow H(\Lambda _{tG};\widetilde{r}_\varepsilon ^2)
\cap H(\Lambda _{tG}).
\end{equation}
\end{prop}

\par\smallskip 
In addition to $P_\varepsilon $ we need a reference operator with
two gaps in the spectrum near ${\mathbb{R}}$. 
Recall that $P_\varepsilon^{\mathrm{int}}$ has discrete spectrum in 
$]-{\mathcal{O}}(\varepsilon),E[$, $E=\frac{\varepsilon}{C}$,
and that we have Weyl asymptotics there by Proposition
\ref{pint1'}. In particular,
\begin{equation}\label{or.3}
\# \left(\sigma (P_\varepsilon ^{\mathrm{int}})\cap \Big(A\varepsilon 
+]-\frac{\delta\varepsilon}{2},\frac{\delta \varepsilon}{2}[\Big)\right)
={\mathcal{O}}(\delta \varepsilon h^{-n}),
\end{equation}
uniformly when $\varepsilon \le\varepsilon (\delta )$, $h\le h(\varepsilon
,\delta )$ ($t$ fixed) and $A$ varies in the interval defined by 
$-{\mathcal O}(\varepsilon )+\frac{\delta \varepsilon}{2}
\le A\varepsilon \le E-\frac{\delta \varepsilon}{2}
$.

\par\smallskip 
Let $\mu _1,...,\mu _N$ (with $N={\mathcal{O}}(\delta \varepsilon h^{-n})$) be
the eigenvalues of $P_\varepsilon ^{\mathrm{int}}$ in
$A\varepsilon +]-\frac{\delta \varepsilon}{2},\frac{\delta \varepsilon}{2}[$ 
and let $e_1,\ldots,e_N\in L^2({\mathbb{R}}^n)$ 
be a corresponding orthonormal family of eigenfunctions, so that
\begin{equation}\label{or.4}
\mathbf{1}_{A\varepsilon +]-\frac{\delta \varepsilon}{2},\frac{\delta \varepsilon}{2}[}
(P_\varepsilon^{\mathrm{int}})u=\sum_j \mu _j(u|e_j)e_j,
\end{equation}
where $(\cdot |\cdot)$ is the usual inner product in $L^2({\mathbb{R}}^n)$. 
We create a gap in the spectrum by moving each $\mu _j$ to
the closest of the two boundary points $A\varepsilon -\frac{\delta\varepsilon}{2}$,
$A\varepsilon +\frac{\delta\varepsilon}{2}$: Put
$$
\widetilde{\mu }_j=
\begin{cases}A\varepsilon -\frac{\delta \varepsilon}{2} 
\hbox{ if } \mu _j\le A\varepsilon ,\\
A\varepsilon +\frac{\delta \varepsilon}{2}
\hbox{ if } \mu _j> A\varepsilon 
\end{cases}
$$
and set
\begin{equation}\label{or.5}
\widetilde{P}_{\varepsilon ,A,\delta }^{\mathrm{int}}u=P_\varepsilon
^{\mathrm{int}}u
+\sum_j (\widetilde{\mu }_j-\mu _j)(u|e_j)e_j,
\end{equation}
so that the eigenvalues $\mu _j$ of $P_\varepsilon ^{\mathrm{int}}$ become
the eigenvalues $\widetilde{\mu }_j$ of $\widetilde{P}_{\varepsilon
  ,A,\delta }^{\mathrm{int}}$ while the eigenvalues of 
  $P_\varepsilon^{\mathrm{int}}$ outside 
  $A\varepsilon +]-\frac{\delta \varepsilon}{2},\frac{\delta\varepsilon}{2}[$ 
remain unchanged.

\par\smallskip 
Now we know that $e_j$ decay exponentially outside ${\mathscr{U}}_\varepsilon $,
as shown in the discussion around (\ref{peps.4}), so if $\chi
_{{\mathscr{U}}_\varepsilon }$ is the cutoff function in 
  (\ref{prep.43}), (\ref{pext.1}), (\ref{pext.4.5}), then
\begin{equation}\label{or.6}
\big\| \widetilde{P}_{\varepsilon ,A,\delta }^{\mathrm{int}}-
P_{\varepsilon ,A,\delta }^{\mathrm{int}}\big\|_{\mathrm{tr}}
={\mathcal{O}}(1)e^{-\frac{\varepsilon}{{\mathcal{O}}(h)}},
\end{equation}
where
\begin{equation}\label{or.7}
P_{\varepsilon ,A,\delta }^{\mathrm{int}}{\bullet}=P_\varepsilon
^{\mathrm{int}}{\bullet}
+\chi _{{\mathscr{U}}_\varepsilon }\sum_j (\widetilde{\mu }_j-\mu _j)(\chi
_{{\mathscr{U}}_\varepsilon }{\bullet}|e_j)e_j.
\end{equation}
From (\ref{or.6}) it follows that
\begin{equation}\label{or.8}
\sigma (P_{\varepsilon,A,\delta  }^{\mathrm{int}})\cap \left( A\varepsilon
  +\Big[-\frac{\varepsilon \delta}{3},\frac{\varepsilon \delta}{3}\Big] \right)= \emptyset . 
\end{equation}
Here ``3'' can be replaced by any number $>2$.

\par\smallskip 
Notice that we could have replaced the definition of $P_\varepsilon^{\mathrm{ext}}$ 
in \eqref{pext.4.5} by
$$
P_\varepsilon ^{\mathrm{ext}}{\bullet}=P_\varepsilon{\bullet} +\chi _{{\mathscr{U}}_\varepsilon }
\sum_{\scriptstyle\mu \in ]-\infty ,E[\cap \sigma (P_\varepsilon ^{\mathrm{int}})}(E-\mu)
(\chi _{{\mathscr{U}}_\varepsilon }\bullet |e_\mu )e_\mu ,
$$
where $e_\mu $ denotes the orthonormal system of eigenfunctions
associated to the $\mu \in ]-\infty ,E[\cap \sigma (P_\varepsilon ^{\mathrm{int}})$.

 \par\smallskip Now, put  
\begin{equation}\label{or.9}
P_{\varepsilon ,A,\delta }{\bullet}=P_\varepsilon{\bullet}
+\chi _{{\mathscr{U}}_\varepsilon }\sum_j (\widetilde{\mu }_j-\mu _j)(\chi
_{{\mathscr{U}}_\varepsilon }{\bullet} |e_j)e_j,
\end{equation}
acting on $H(\Lambda _{tG};\widetilde{r}_\varepsilon ^2)$. As in Section
\ref{peps}, if we restrict the attention to the region
(\ref{or.1}), there is a bijection
\begin{equation}\label{or.10}
b:\, \sigma (P_{\varepsilon ,A,\delta }^{\mathrm{int}})\longrightarrow
\mathrm{Res\,}(P_{\varepsilon ,A,\delta }),
\end{equation}
such that $b(\mu )-\mu ={\mathcal{O}}(h^\infty )$ (with the same proviso
as in the footnote to Proposition \ref{peps2}). In particular, for
every fixed $N_0>0$, $P_{\varepsilon ,A,\delta }$ has no resonances
outside an $h^{N_0}$-neighborhood of
$$
\Big]-{\mathcal{O}}(\varepsilon ),E'\Big[\setminus 
\left(A\varepsilon +\Big]-\frac{\varepsilon \delta}{2},
\frac{\varepsilon \delta}{2}\Big[\right)
$$ 
when $h$ is small. Outside such a
neighborhood in the set (\ref{or.1}), we have
\begin{equation}\label{or.11}
\Big\|(P_{\varepsilon ,A,\delta }-z)^{-1}\Big\|\le 
{\mathcal O}(1)\left(\mathrm{dist\,}\Big(z,{\mathbb{R}}\setminus \big(A\varepsilon 
  +]-\frac{\varepsilon\delta}{2},\frac{\varepsilon \delta}{2}[\,\big)\Big)\right)^{-1}
\end{equation}
as a bounded operator in $H(\Lambda _{tG})$.

\par\smallskip 
In the same way, we can build reference operators with two
gaps. Let $B\in {\mathbb{R}}$ be a second energy level as in (\ref{or.3})
and assume in addition that $B-A\ge \delta $. We first define
$\widetilde{P}_{\varepsilon ,A,B,\delta }^{\mathrm{int}}$ as in
(\ref{or.5}) by replacing each eigenvalue $\mu_j$ of 
$P_\varepsilon^{\mathrm{int}}$ in 
$\{ A\varepsilon,B\varepsilon \}+]-\frac{\delta \varepsilon}{2},\frac{\delta \varepsilon}{2}[$ 
by the closest boundary point $\widetilde{\mu }_j$ of this
set. Then we define $P_{\varepsilon ,A,B,\delta }^{\mathrm{int}}$ as in
(\ref{or.7}). We have the obvious modifications of the bijection in
(\ref{or.10}) and the resolvent estimate (\ref{or.11}). Let us also
notice that
\begin{equation}\label{or.12}
\begin{aligned}
\# \Big(\mathrm{Res\,}(P_{\varepsilon ,A,B,\delta })&\cap 
\big(\, \big]A\varepsilon,B\varepsilon\big[+i\big]-\frac{\varepsilon}{{\mathcal{O}}(1)},{\mathcal{O}}(\varepsilon)\big[\, \big)\Big)\\ 
&=\# \Big(\sigma (P_{\varepsilon ,A,B,\delta }^{\mathrm{int}})\cap \big]A\varepsilon,B\varepsilon\big[ \,\Big) \\
&=\big(\omega (\varepsilon B)-\omega (\varepsilon A)\big)({2\pi h})^{-n}+{\mathcal{O}}(\delta \varepsilon )h^{-n}, 
\end{aligned}
\end{equation}
where the volume function $\omega $ is discussed in Appendix
\ref{app}. This estimate is uniform for $0<\delta \ll 1$, $0<\varepsilon
\le \varepsilon (\delta )$, $0<h\le h(\delta,\varepsilon)$. ($0<t\ll 1$ is fixed.)

\section{Relative determinants}\label{det}
\setcounter{equation}{0}

Fix $t>0$ small so that the earlier estimates are valid
in
\begin{equation}\label{det.1}
R=R_{C_1,C}(\varepsilon )=\big]-{\mathcal{O}}(\varepsilon ),\frac{\varepsilon}{C}\big[
+i\big]-\frac{\varepsilon}{C_1},{\mathcal{O}}(\varepsilon )\big[
\end{equation}
or in certain explicitly given subsets of this region. Here $C$,
${\mathcal{O}}(\varepsilon )$ are as in (\ref{or.1}) and $C_1>0$
is large enough, depending on $t$.\footnote{\label{tparameter} Notice that if we put 
{$\varepsilon_\mathrm{new}=\frac{\varepsilon}{\widetilde{C}}$} for $\widetilde{C}$ fixed
large enough, then the set $R$ will contain a rectangle of the form
$]-{\cal O}(\varepsilon _\mathrm{new}),\varepsilon
_\mathrm{new}[+i]-\varepsilon _\mathrm{new}, \varepsilon
_\mathrm{new}[$ and we recover the scales in Theorem \ref{int1} with
$\varepsilon $ replaced by the rescaled $\varepsilon_\mathrm{new} $.
}

\par\smallskip In (\ref{pext.6}) we have seen that uniformly for $z\in R$,
$$
(P_\varepsilon ^{\mathrm{ext}}-z)^{-1}=\left\{\begin{aligned}
{\mathcal{O}}(1) &:\, H(\Lambda _{tG}) \longrightarrow H(\Lambda_{tG};\widetilde{r}_\varepsilon ^2),\\
{\mathcal{O}}(\frac{1}{\varepsilon}) &:\, H(\Lambda _{tG}) \longrightarrow H(\Lambda_{tG}).
\end{aligned}\right.
$$

\par\smallskip 
From the definition of $P_\varepsilon $ in the beginning of Section
\ref{prep} we see that
\begin{equation}\label{det.2}
\big\|P-P_\varepsilon \big\|_{\mathrm{tr}}={\mathcal{O}}(1)\varepsilon \big(\frac{\varepsilon}{h}\big)^n.
\end{equation}
{ In fact, from \eqref{bp.3} we see that the $h$-quantization of
$\chi_{\varepsilon}$ is unitarily equivalent to the $h=1$-quantization of $a_{\varepsilon,h}(x,\xi)=\varepsilon \chi\Big(\sqrt{\frac{h}{\varepsilon}}(x,\xi)\Big)$. Recalling that $\chi\in\mathcal{S}$
and that $0<\frac{h}{\varepsilon}\leq 1$, we have
$$
\sum_{|\alpha|\leq 2n+1}\Vert \partial_{x,\xi}^{\alpha}a_{\varepsilon,h}\Vert_{L^1(\mathbb{R}^{2n})}
\leq \mathcal{O}\Big(\varepsilon\big(\frac{\varepsilon}{h}\big)^n\Big)
$$
and applying for instance \cite[Theorem 9.4]{DiSj99}, we get $\Vert \chi_{\varepsilon}(x,hD_x)\Vert_{\mathrm{tr}}\leq \mathcal{O}
{\Big(\varepsilon\big(\frac{\varepsilon}{h}\big)^n\Big)}$
.}

Similarly from the definition of $P_\varepsilon ^{\mathrm{ext}}$ in
(\ref{pext.1}), (\ref{pext.4.5}), we have
\begin{equation}\label{det.255}
\big\|P_\varepsilon -P^{\mathrm{ext}}_\varepsilon \big\|_{\mathrm{tr}}={\mathcal{O}}(1)h^{-n}.
\end{equation}
Thus,
\begin{equation}\label{det.3}
\big\|P -P^{\mathrm{ext}}_\varepsilon \big\|_{\mathrm{tr}}={\mathcal{O}}(1)h^{-n}.
\end{equation}

\par\smallskip 
We can define the following relative determinants and their
logarithms for $z\in R$:
\begin{equation}\label{det.5}
{\mathcal D}_{P}(z)=\ln |\det (P-z)(P_\varepsilon ^{\mathrm{ext}}-z)^{-1}|,
\end{equation}
\begin{equation}\label{det.6}
{\mathcal D}_{P_\varepsilon }(z)=\ln |\det (P_\varepsilon -z)(P_\varepsilon ^{\mathrm{ext}}-z)^{-1}|,
\end{equation}
\begin{equation}\label{det.6.5}
{\mathcal D}_{P_{\varepsilon ,\delta }}(z)=\ln |\det (P_{\varepsilon ,\delta }-z)(P_\varepsilon ^{\mathrm{ext}}-z)^{-1}|.
\end{equation}
Here $P_{\varepsilon ,\delta }=P_{\varepsilon ,A,B,\delta }$ is given in Section
\ref{or}. We derive some upper bounds:

\par\smallskip 
Write
$$
(P-z)(P_\varepsilon ^{\mathrm{ext}}-z)^{-1}=1-(P_\varepsilon
^{\mathrm{ext}}-P)(P_\varepsilon ^{\mathrm{ext}}-z)^{-1}.
$$
The last term is of trace class, so ${\mathcal D}_P$ is well defined. More
precisely,
$$
\big\| (P_\varepsilon -P)(P_\varepsilon ^{\mathrm{ext}}-z)^{-1}\big\|_{\mathrm{tr}}
\le \big\| P_\varepsilon ^{\mathrm{ext}}-P\big\|_{\mathrm{tr}}\big\| (P_\varepsilon ^{\mathrm{ext}}-z)^{-1}\big\|
\le {\mathcal{O}}({\varepsilon}^{-1} )h^{-n},
$$
where the norms and trace class norms are the ones for operators in
$H(\Lambda _{tG})$. Since in general (see \cite{GoKr69}),
$$
|\det (1+K)|\le \exp \big\|K\big\|_{\mathrm{tr}},
$$
we conclude that
\begin{equation}\label{det.6.55}
{\mathcal D}_P(z)\le {\mathcal{O}}(1)\varepsilon ^{-1}h^{-n}.
\end{equation}
Similarly, $$\big\| P_\varepsilon ^{\mathrm{ext}}-P_\varepsilon
\big\|_{\mathrm{tr}},\ \big\| P_\varepsilon ^{\mathrm{ext}}-P_{\varepsilon ,\delta }
\big\|_{\mathrm{tr}}={\mathcal{O}}(h^{-n}),$$
so ${\mathcal P}_{P_\varepsilon }(z)$, ${\mathcal D}_{P_{\varepsilon ,\delta }}(z)$
are well defined and satisfy
\begin{equation}\label{det.7}
{\mathcal P}_{P_\varepsilon }(z),\ {\mathcal D}_{P_{\varepsilon ,\delta }}(z)
\le {\mathcal{O}}(1)\varepsilon ^{-1}h^{-n}.
\end{equation}

\par\smallskip Next, look at
\begin{multline}\label{det.8}
{\mathcal D}_P(z)-{\mathcal D}_{P_\varepsilon }(z)=\ln |\det (P-z)(P_\varepsilon
-z)^{-1}|\\ =\ln |\det (1-(P_\varepsilon -P)(P_\varepsilon
-z)^{-1})|
\end{multline}
which is well defined away from $\sigma (P_\varepsilon )$ and bounded
from above by
\begin{equation}\label{det.9}
 \begin{aligned}
 \big\| (P_\varepsilon -P)(P_\varepsilon-z)^{-1})\big\|_{\mathrm{tr}} 
 &\le \big\| P_\varepsilon -P\big\|_{\mathrm{tr}}\big\|(P_\varepsilon -z)^{-1}\big\|\\
&\le {\mathcal{O}}(1)\varepsilon^{n+1}h^{-n}\big\| (P_\varepsilon -z)^{-1}\big\| .
 \end{aligned}
\end{equation}
We know from (\ref{peps.17}), that
\begin{equation}\label{det.10}
\big\| (P_\varepsilon -z)^{-1}\big\|\le \frac{{\mathcal{O}}(1)}{\varepsilon \delta
},\hbox{ for }z\in R \, \hbox{ with } \, |\Im z|>\delta \varepsilon .
\end{equation}
This is uniform for $0<\varepsilon \le \varepsilon (\delta )\ll 1$, $0<h\le
h(\delta,\varepsilon)$. From this and (\ref{det.9}), (\ref{det.8}), we
get the upper bound,
\begin{equation}\label{det.11}
{\mathcal D}_P(z)-{\mathcal D}_{P_\varepsilon }(z)\le {\mathcal{O}}(1)\frac{1}{\delta
}\left(\frac{\varepsilon }{h} \right)^n, \hbox{ for }z\in R \, \hbox{ with } \, |\Im z|>\delta \varepsilon .
\end{equation}

\par\smallskip We also have lower bounds in a smaller part of $R$. From
Proposition \ref{or1} we know that 
\begin{equation}\label{det.12}
\big\| (P-z)^{-1}\big\|\le \frac{{\mathcal{O}}(1)}{\delta \varepsilon },
\end{equation}
for
\begin{equation}\label{det.13}
z\in R \hbox{ with } |\Im z|>\delta \varepsilon \hbox{ and } \Re z<-\frac{\varepsilon}{{\mathcal{O}}(1)},
\end{equation}
where the upper bound on $\Re z$ is the same as in (\ref{or.1}).
Exchanging $P$ and $P_\varepsilon $ in (\ref{det.8}), 
\begin{equation}\label{det.14}
{\mathcal D}_{P_\varepsilon }(z)-{\mathcal D}_P(z)=\ln |\det (1-(P-P_\varepsilon )(P-z)^{-1}|,
\end{equation}
we then get, for $z$ in the subset (\ref{det.13}),
\begin{equation}\label{det.15}
{\mathcal D}_P(z)-{\mathcal D}_{P_\varepsilon }(z)\ge -{\mathcal{O}}(1)\frac{1}{\delta
}\left(\frac{\varepsilon }{h} \right)^n.
\end{equation}

\par\smallskip 
We shall apply Jensen's formula and related estimates,
following \cite[Section 5]{Sj01}. Assume for simplicity that we have 
$-\frac{\varepsilon}{{\mathcal{O}}(1)}=-\frac{\varepsilon}{2}$ in (\ref{det.13}). 
Let
\begin{equation}\label{det.16}
z_0=-\varepsilon -i\frac{\varepsilon }{2C_1},
\end{equation}
where $C_1$ is the constant in (\ref{det.1}), so that (\ref{det.15})
holds for $z=z_0$. (The following can also be carried out in the upper
half-plane with $z_0=-\varepsilon +i\frac{\varepsilon }{2C_1}$.) It will be
convenient to work in the rescaled variable $\widetilde{z}$ with
$z=\varepsilon\,\widetilde{z}$, so that
\begin{equation}\label{det.18}
\widetilde{z}_0=-1-i\frac{1}{2C_1},
\end{equation}
and we put $\widetilde{{\mathcal D}}_P(\widetilde{z})={\mathcal D}_P(z)$ and
similarly for the other ${\mathcal D}_{P_{(...)}}$. Let $r_0=\varepsilon\,
\widetilde{r}_0$ be the largest number such that
\begin{equation}\label{det.19}
D(z_0,r_0)\subset R_\delta :=\Big\{ z\in R;\, \Im z<-\delta \varepsilon \Big\}.
\end{equation}
(More explicitly, $r_0=(\frac{1}{2C_1}-\delta )\varepsilon $.) Consider the
holomorphic function $\widetilde{f}(\widetilde{z})=f(z)$,
\begin{equation}\label{det.20}
f(z)=\det \left((P-z)(P_\varepsilon -z)^{-1} \right),
\end{equation}
for $z\in D(z_0,r_0)$ (corresponding to $\widetilde{z}\in
D(\widetilde{z}_0,\widetilde{r}_0)$). By (\ref{det.11}), we have 
\begin{equation}\label{det.21}
|f(z)|\le \exp \left( {\mathcal{O}}(1)\frac{1}{\delta }
  \left(\frac{\varepsilon }{h} \right)^n \right),\qquad z\in D(z_0,r_0).
\end{equation}
Moreover,
\begin{equation}\label{det.22}
|f(z_0)|\ge \exp \left(-{\mathcal{O}}(1)\frac{1}{\delta }
  \left(\frac{\varepsilon }{h} \right)^n \right),
\end{equation}
since (\ref{det.15}) holds for $z=z_0$. From Jensen's formula it
follows that the number of zeros of $\widetilde{f}$ in
$D(\widetilde{z}_0,(1-\theta )\widetilde{r}_0)$ is $\le {\mathcal{O}}_\delta (1)\varepsilon^n h^{-n}$, 
if $\theta \in ]0,1[$ is any fixed constant. 
Equivalently, $f$ has $\le {\mathcal{O}}_\delta (1) \varepsilon^n
h^{-n}$ zeros in $D(z_0,(1-\theta )r_0)$.

\par\smallskip 
Let $z_j=\varepsilon\, \widetilde{z}_j$, $j=1,2,\ldots,N$ be the zeros of
$f$ in $D(z_0,(1-\theta )r_0)$, repeated according to their
multiplicity, and put
$$
D_w(z;h)=\prod_{j=1}^N (\widetilde{z}-\widetilde{z}_j),\qquad  z=\varepsilon\, \widetilde{z}.
$$
Repeating the (standard) arguments in \cite{Sj01}, we see that
\begin{equation}\label{det.23}
|D_w(z;h)|\le \exp \Big( {\mathcal{O}}_\delta (1) \varepsilon^n h^{-n} \Big) 
\, \hbox{ in } \, D(z_0,(1-\theta )r_0)
\end{equation}
and that for any interval $I\in \big[0,(1-\theta )\widetilde{r}_0\big[$ of
length $|I|>0$ there exists $\widetilde{r}_1\in I$ such that
\begin{equation}\label{det.24}
\hskip-10pt |D_w(z;h)|\ge \exp \left(-{\mathcal{O}}_{\delta,|I|} (1) 
\Big(\frac{\varepsilon}{h}\Big)^{n}
\right)\hbox{ when }|z-z_0|=r_1:=\varepsilon \widetilde{r}_1.
\end{equation}

\par\smallskip 
Next, write
\begin{equation}\label{det.25}
f(z)=e^{G(z)}D_w(z;h),
\end{equation}
with $G$ holomorphic in $D\big(z_0,(1-\theta )r_0\big)$. Using the above
bounds and Harnack's inequality (as in \cite[Section\ 5]{Sj01}) we get
\begin{equation}\label{det.26}
|G(z)|\le {\mathcal{O}}_\delta (1) \varepsilon^n h^{-n},\quad z\in D(z_0,(1-\theta )^2r_0)
\end{equation}
and for any interval $I\in \big[0,(1-\theta )^2\widetilde{r}_0\big[$ of length
$|I|>0$ there exists $\widetilde{r}_1\in I$ such that
\begin{equation}\label{det.27}
|f(z)|\ge \exp \left(-{\mathcal{O}}_{\delta ,|I|}(1)\varepsilon^n h^{-n}
\right)\hbox{ when }|z-z_0|=r_1:=\varepsilon \widetilde{r}_1.
\end{equation}
In other words,
\begin{equation}\label{det.28}
{\mathcal D}_P-{\mathcal D}_{P_\varepsilon }\ge -{\mathcal{O}}_{\delta ,|I|}(1)\varepsilon^n h^{-n},
\end{equation}
for $z$ as in (\ref{det.27}). Here we can take $|I|=(1-\theta
)^2\widetilde{r}_0-(1-\theta )^3\widetilde{r}_0$ and find a
corresponding $r_1$ with $(1-\theta )^3r_0\le r_1<(1-\theta )^2r_0$.

\par\smallskip The argument can now be repeated, by replacing $z_0$ by any new
point on $\partial D(z_0,r_1)$ ... In this way, we continue until we
have covered $R_\delta \setminus (\partial R_\delta +D(0,\delta
\varepsilon ))$ with ${\mathcal{O}}_\delta (1)$ discs, and recalling that the
zeros of $f$ in (\ref{det.20}) are the resonances of $P$ in $R_\delta
$, we get the following result:
\begin{prop}\label{det1}
Define $R_\delta $ as in (\ref{det.19}), (\ref{det.1}). Then 

\medskip\par\noindent (A) The number of resonances in $R_\delta
\setminus (\partial R_\delta +D(0,\varepsilon \delta ))$ is $\le {\mathcal
  O}_\delta (1)\varepsilon^n h^{-n}$ with the usual convention that
$0<\varepsilon \le \varepsilon (\delta )$, $0<h\le h(\delta,\varepsilon)$.

\medskip\par\noindent (B) For all $a,b$ with $0<a<b<1$ independent of
$\varepsilon ,\delta $ and all segments $J\subset I\subset R_\delta
\setminus (\partial R_\delta +D(0,\varepsilon \delta ))$ of lengths
$|J|=a\varepsilon $, $|I|=b\varepsilon $, there exists $z\in J$ such that
\begin{equation}\label{det.29}
{\mathcal D}_P(z)-{\mathcal D}_{P_\varepsilon }(z)\ge -{\mathcal{O}}_{a,b,\delta
}(1) \varepsilon^n h^{-n}.
\end{equation}
\end{prop}
Notice that Proposition \ref{det1} remains valid if we replace
$R_\delta $, defined in (\ref{det.19}), with 
\begin{equation}\label{det.30}
R_\delta ^+=\{ z\in R;\, \Im z >\varepsilon \delta  \}.
\end{equation}
Then \textit{(A)} holds trivially since there are no resonances in the
upper half-plane.

We recall the bounds (\ref{det.11}), (\ref{det.15}). In order to
simplify the notations, we assume that the proposition is valid in all
of $R_\delta $ (and in $R_\delta ^+$), as can be achieved by a slight
dilation of the parameters.

In order to complete the proof of Theorem \ref{int1}, we shall work
with ${\mathcal D}_{P}(z)-{\mathcal D}_{P_{\varepsilon ,\delta }}(z)$, exploiting
the fact that ${\mathcal D}_{P_{\varepsilon ,\delta }}$ is harmonic in
\begin{equation}\label{det.31}
\hskip-15pt R_{\varepsilon,\delta}:=\{ z\in R;\, |\Im z|>\varepsilon \delta \hbox{ or }|\Re z-A\varepsilon
|<\frac{\varepsilon \delta }{4}\hbox{ or }
|\Re z-B\varepsilon
|<\frac{\varepsilon \delta }{4}
 \}.
\end{equation} 
See Figure \ref{Repsdelta}.
Here $P_{\varepsilon ,\delta }=P_{\varepsilon ,A,B,\delta }$ is discussed in
Section \ref{or} and we know from that discussion that 
$$
\big\| P_\varepsilon -P_{\varepsilon,\delta }\big\|\le {\mathcal{O}}(\varepsilon \delta ),
\qquad \big\| P_\varepsilon -P_{\varepsilon,\delta }\big\|_{\mathrm{tr}}\le 
{\mathcal{O}}(\varepsilon \delta )\frac{\varepsilon\delta}{h^n}.
$$ 
For $z$ in the region $R_{\varepsilon,\delta}$ we have
\begin{equation}\label{det.32}
\begin{aligned}
 {\mathcal D}_{P_\varepsilon }(z)-{\mathcal D}_{P_{\varepsilon ,\delta }}(z) & =\ln
\left| \det (P_\varepsilon -z)(P_{\varepsilon ,\delta }-z)^{-1} \right| \\
& \le \big\| (P_\varepsilon -P_{\varepsilon ,\delta })(P_{\varepsilon ,\delta}-z)^{-1}\big\|_{\mathrm{tr}} \\
& \le \big\| P_\varepsilon -P_{\varepsilon ,\delta }\big\|_{\mathrm{tr}}\big\|(P_{\varepsilon ,\delta}-z)^{-1}\big\| \\
& \le {\mathcal{O}}(1)\frac{(\varepsilon \delta )^2}{h^n}\frac{1}{\varepsilon\delta } \\
& ={\mathcal{O}}(1)\frac{\varepsilon \delta }{h^n}. 
\end{aligned}
\end{equation}
Here we also use (\ref{or.11}) or rather its natural analogue for
$P_{\varepsilon ,A,B,\delta }$.

\par\smallskip 
Since $\big\| P-P_\varepsilon \big\|_{\mathrm{tr}}\le {\mathcal{O}}(1)\varepsilon^{n+1} h^{-n}
\le {\mathcal{O}}(1)(\varepsilon \delta )^2h^{-n}$ when
$0<\varepsilon \le \varepsilon (\delta )$, we have
\begin{equation}\label{det.33}
{\mathcal D}_P-{\mathcal D}_{P_{\varepsilon ,\delta }}\le {\mathcal{O}}(\varepsilon
\delta )h^{-n}
\end{equation}
in (\ref{det.31}). Indeed, this follows from (\ref{det.32}) after
replacing $P_\varepsilon $ there with $P$.

\par\smallskip 
On the smaller set
\begin{equation}\label{det.34}
R_\delta:=\{z\in R;\, |\Im z| >\varepsilon \delta  \},
\end{equation}
we have
\begin{equation}\label{det.35}
\big\| (P_\varepsilon -z)^{-1}\big\|\le \frac{{\mathcal{O}}(1)}{\delta \varepsilon }
\end{equation}
and exchanging $P_\varepsilon $ and $P_{\varepsilon ,\delta }$ in
(\ref{det.32}), we get ${\mathcal D}_{P_\varepsilon }-{\mathcal D}_{P_{\varepsilon,\delta }}
\ge -{\mathcal{O}}(\varepsilon \delta )h^{-n}$, hence with
(\ref{det.32}):
\begin{equation}\label{det.36}
\left| {\mathcal D}_{P_\varepsilon }-{\mathcal D}_{P_{\varepsilon
    ,\delta }} \right| \le {\mathcal{O}}(1)\frac{\varepsilon \delta }{h^n},
\end{equation}
for $z$ in the set (\ref{det.34}). This means that the estimates \eqref{det.11}, \eqref{det.29} for
${\mathcal D}_P-{\mathcal D}_{P_\varepsilon }$ carry over to ${\mathcal D}_P-{\mathcal
  D}_{P_{\varepsilon ,\delta } }$, provided that we replace the remainder
estimates ${\mathcal{O}}_{\cdots}(1)\varepsilon^n h^{-n}$ by ${\mathcal{O}}(\varepsilon
\delta )h^{-n}$: For $ z\in R_\delta$, we have statement \textit{(B)} in Proposition \ref{det1} with
\begin{equation}\label{det.37}
{\mathcal D}_P(z)-{\mathcal D}_{P_{\varepsilon ,\delta }}(z)\ge -{\mathcal{O}}(\varepsilon \delta )h^{-n},
\end{equation}
instead of (\ref{det.29}).

\par\smallskip To get lower bounds in $A\varepsilon +]-\frac{\varepsilon \delta}{4},\frac{\varepsilon
\delta}{4}[+i]-\varepsilon \delta ,\varepsilon \delta [$ (and similarly with
$B$ instead of $A$, we can apply the above arguments for ${\mathcal D}_P-{\cal
  D}_{P_\varepsilon }$ in $R_\delta $ to ${\mathcal D}_P-{\cal
  D}_{P_{\varepsilon ,\delta }}$ in
$A\varepsilon +]-\frac{\varepsilon \delta}{4},\frac{\varepsilon
\delta}{4}[+i]-\varepsilon \delta ,\varepsilon \delta [$, now starting at the point
$z_0=A\varepsilon + i 2\varepsilon \delta $ and get:
\begin{prop}\label{det2}
\begin{itemize}
 \item[{}] {}
   
  \item[(A)] The number of resonances of $P$ in
  $A\varepsilon +]-\frac{\varepsilon \delta}{4},\frac{\varepsilon\delta}{4}[
  +i]-\varepsilon \delta ,\varepsilon \delta [$ $($i.e.\ the zeros of 
  $\det \left( (P-z)(P_{\varepsilon ,\delta }-z)^{-1} \right)$ $)$ is 
  $\le {\mathcal {O}}(\varepsilon \delta )h^{-n}$.

\item[(B)] For all $a,b$ with $0<a<b<1$, independent of $\varepsilon ,\delta $
and all segments $J\subset I\subset A\varepsilon +]-\frac{\varepsilon \delta}{4},\frac{\varepsilon
\delta}{4}[+i]-\varepsilon \delta ,\varepsilon \delta [$ of length
$|J|=a\delta \varepsilon $, $|I|=b\delta \varepsilon $, there exist
 $z\in J$, such that
\begin{equation}\label{det.38}
{\mathcal D}_P(z)-{\mathcal D}_{P_{\varepsilon ,\delta }}(z)\ge -{\mathcal{O}}_{a,b}(\varepsilon \delta )h^{-n}.
\end{equation}
\end{itemize}
\end{prop}
The same statements hold with $B$ instead of $A$.

\section{End of the proof}\label{ep}
\setcounter{equation}{0}

We study the number of resonances in  the rectangle
\begin{equation}\label{ep.2}
\Gamma=\big]A\varepsilon, B\varepsilon\big[+i\Big]-\frac{\varepsilon}{2C_{1}},\frac{\varepsilon}{2C_{1}}\Big[,
\end{equation}
where $A,B$ and $C_1$ are positive constants, (see \eqref{det.1} and Section \ref{or}). Set
\begin{equation}\label{ep.1}
\gamma=\partial\Gamma\,.
\end{equation}

Again it is convenient to scale : $z=\varepsilon \widetilde  z$ and use $\widetilde \gamma=\partial\widetilde\Gamma$ where
\begin{equation}\label{ep.3}
\widetilde{\Gamma }=\big]A, B\big[+i\Big]-\frac{1}{2C_{1}},\frac{1}{2C_{1}}\Big[.
\end{equation}

We shall apply Theorem 1.1 of \cite{Sj10}, (see also \cite[Theorem 12.1.1]{Sj19}). In the rescaled variable $\widetilde  z$,
we choose the Lipschitz weight on $\widetilde \gamma$:
\begin{equation}\label{ep.4}
{\widetilde d}(\widetilde  z)=\frac{1}{C}\big(\delta+\frac{1}{2}|\Im(\widetilde  z)|\big),\qquad \textrm{with}\,\, C>1\,\,\, \textrm{large enough},
\end{equation}
satisfying (cf.\ \cite[(1.1)--(1.5)]{Sj10}):
$$
|{\widetilde d}(\widetilde  z)-{\widetilde d}(\widetilde  w)|\leq \frac{1}{2}|\widetilde 
z-\widetilde  w|,\qquad \forall\widetilde  z,\widetilde  w\in\widetilde \gamma.
$$
Extend ${\widetilde d}$ to all $\mathbb C$ by setting
$$
{\widetilde d}(\widetilde  z)=\inf_{\widetilde  w\in\widetilde \gamma}\Big(\widetilde d(\widetilde  w)|+\frac{1}{2}|\widetilde  z-\widetilde  w|\Big),
\quad \forall \widetilde  z\in\mathbb C.
$$
This extended function is also Lipschitz of modulus at most
$\frac{1}{2}$, such that ${\widetilde d}(\widetilde  z)\geq \frac{1}{2}
\textrm{dist}(\widetilde  z,\widetilde  \gamma)$ and
$$
 |\widetilde  z-\widetilde  w|\leq {\widetilde d}(\widetilde  w)\Longrightarrow \frac{{\widetilde d}(\widetilde  w)}{2}
 \leq {\widetilde d}(\widetilde  z) \leq \frac{3{\widetilde d}(\widetilde  w)}{2}\,.
$$
Choose $\widetilde  z_j^0\in \widetilde  \gamma,\,\, j=0,1,\ldots,N-1$ 
distributed along $\widetilde  \gamma=\partial\widetilde\Gamma$
in the positively oriented sense such that
\begin{equation}\label{ep.5}
\frac{{\widetilde d}(\widetilde  z_j^0)}{k}\leq |\widetilde  z_{j+1}^0-\widetilde  z_j^0|\leq \frac{{\widetilde d}(\widetilde  z_j^0)}{2}\,,\quad 0\leq j\leq N-1
\end{equation}
with the convention that $j+1=0$ when $j=N-1$, and for some $k> 2$. Define
\begin{equation}\label{ep.6}
\varphi(z):=h^n\big({\mathcal D}_{P_{\varepsilon,\delta}}(z)+C\varepsilon \delta\big)
\end{equation}
with $C>0$ large enough so that 
\begin{equation}\label{ep.7}
{\mathcal D}_{P}(z) \le h^{-n}\varphi(z) \hbox{ in the
    set  (\ref{det.31}). }
\end{equation}

The $\widetilde  z_j^0$ in $\{A, B\}+i]-\delta ,\delta [$ are choosen
according to \textit{(B)} in Proposition \ref{det2} so that (\ref{det.38})
holds when $z=z_j^0:=\varepsilon \widetilde z_j^0$. Hence 
\begin{equation}\label{ep.8}
{\mathcal D}_{P}(z_j^0) \ge h^{-n}\big(\varphi(z_j^0)-\varepsilon_j\big) ,
\end{equation}
where $\varepsilon_j>0$ is independent of $j$, of the form
\begin{equation}\label{ep.9}
\varepsilon_j=C_0\varepsilon \delta\,
\end{equation}
with $C_0>0$ large enough. The $z_j^0\in \partial \Gamma $ with $|\Im z_j^0|>\varepsilon \delta $
are chosen according to \textit{(B)} in Proposition \ref{det1}, for which we
have (\ref{det.37}) and hence (\ref{ep.8}).

Consider the points $\widetilde  z_1^0,\widetilde  z_2^0,\ldots,\widetilde  z_K^0$ on $B+i]\delta,\frac{1}{2C_{1}}[$ (possibly after relabbeling)
ordered so that $\Im(\widetilde  z_1^0)<\Im(\widetilde  z_2^0)<\ldots<\Im(\widetilde  z_K^0)$. From \eqref{ep.5} we see that $\Im(\widetilde  z_j^0)$
growths geometrically with $j$ and it follows that $K\leq {\mathcal O}(1)|\ln(\delta)|$. 
The same holds for the corresponding  points on  $B+i]-\frac{1}{2C_{1}},-\delta[$, $A+i]\delta,\frac{1}{2C_{1}}[$ 
and $A+i]-\frac{1}{2C_{1}},-\delta[$.
The total number of  points  $\widetilde z_j^0$ on $\widetilde \gamma$
is $N={\mathcal O}(1)|\ln(\delta)|$.

Notice that $\phi(z)$ is harmonic in $\overset{N-1}{\underset{j=0}{\bigcup}}D(z_j^0,r_j)$. Apply Theorem 1.1
in \cite{Sj10} (or \cite[Theorem 12.1.1]{Sj19}) with $h$ there is replaced by $h^n$: We get in view of \eqref{ep.9},

\begin{align*}
\Big|\#\big({\mathcal D}_P^{-1}(0)\cap\Gamma\big)-\frac{1}{2\pi h^{n}}\int_{\Gamma}
\Delta\big(h^n{\mathcal D}_{P_{\varepsilon,\delta}}(z)\big)L(dz)\Big|
&\leq {\mathcal O}(1)h^{-n} N\varepsilon\delta\nonumber\\
{}&={\mathcal O}(1)h^{-n}\varepsilon\delta|\ln\delta|,
\end{align*}

where we first work in the $\widetilde  z$-variable but notice that 
$$
\int_{\frac{1}{\varepsilon}\Gamma} \Delta_{\widetilde  z}\big({\mathcal D}_{P_{\varepsilon,\delta}}
( \widetilde z)\big)\,L(d\widetilde  z)
=\int_{\Gamma}\Delta_z\big({\mathcal D}_{P_{\varepsilon,\delta}}(z)\big)\,L(dz).
$$
Here $\displaystyle\frac{1}{2\pi} \int_{\Gamma}\Delta\big(h^n{\mathcal D}_{P_{\varepsilon,\delta}}(z)\,L(dz)$ 
is equal to the number of zeros in $\Gamma$ of ${\mathcal D}_{P_{\varepsilon,\delta}}$ 
or equivalently the number of resonances in $\Gamma$ of $P_{\varepsilon,\delta}$.
By \eqref{or.12} 
this number is equal to $\displaystyle (2\pi h)^{-n}(\omega(\varepsilon B)-\omega(\varepsilon A))+
{\mathcal O}(\delta|\ln\delta|\varepsilon)h^{-n}$. 
This gives \textit{(B)} in Theorem \ref{int1} (with $a=\varepsilon A$ and $b=\varepsilon B$), 
since we already have the part \textit{(A)} which follows from Proposition \ref{det1}, \textit{(A)}.
\begin{appendix}

\section{Review of \cite{HeSj86}
and adaptation to the dilated situation}\label{dil}
\setcounter{equation}{0}

In this appendix we recall very briefly some basic microlocal tools 
developed in \cite{HeSj86} for the study of semiclassical resonances 
and adapt them to our situation. 
To some extent, we shall follow 
the review in \cite[Section 5]{HiMaSj17}.

\subsection{Order functions and symbols:}
\label{AP1}
\hskip-7pt
Let $R,\,r\in C^\infty({\mathbb{R}}^n;]0,+\infty [\,)$, satisfy for all
$\alpha \in {\mathbb{N}}^n$:
\begin{equation}\label{dil.1}
\partial _x^\alpha R(x)={\mathcal{O}}(1)R(x)^{1-|\alpha |},
\end{equation}
\begin{equation}\label{dil.2}
\partial _x^\alpha r(x)={\mathcal{O}}(1)r(x)R(x)^{-|\alpha |}.
\end{equation}

Define $\widetilde{r}(x,\xi )\in C^\infty ({\mathbb{R}}^{2n};]0,+\infty [)$ by
\begin{equation}\label{dil.3}
\widetilde{r}(x,\xi )=\left(r(x)^2+\xi ^2 \right)^{\frac12}.
\end{equation}
Then
\begin{equation}\label{dil.4}
\partial _x^\alpha \partial _\xi ^\beta \widetilde{r}(x,\xi )={\mathcal
  O}(1)\widetilde{r}(x,\xi )^{1-|\beta |}R(x)^{-|\alpha |}.
\end{equation}
We make the important assumption that
\begin{equation}\label{dil.5}
r(x)\ge 1,\quad  r(x)R(x)\ge 1.
\end{equation}
The quantities $R$, $r$, $\widetilde{r}$ are our basic scale functions. 
The functions $R$ and $\widetilde{r}$ give
the scale in $x$ and $\xi$, respectively.

\begin{dref}[Order functions]\label{AP.1}
 \begin{itemize}
  \item[{}]
  \item[\textbf{(a)}] We say that $m\in C^\infty ({\mathbb{R}}^{2n};]0,+\infty [)$ is an order
function if
\begin{equation}\label{dil.6}
\partial _x^\alpha \partial _\xi ^\beta m(x,\xi )={\mathcal{O}}(1)m(x,\xi
)R(x)^{-|\alpha |}\widetilde{r}(x,\xi )^{-|\beta |},
\end{equation}
for all $\alpha ,\beta \in {\mathbb{N}}^n$.
  \item[\textbf{(b)}] A function $m=m(x)\in C^\infty
({\mathbb{R}}^n;]0,+\infty [)$, independent of $\xi $, is an order function
if
$$\partial _x^\alpha m(x)={\mathcal{O}}(1)m(x)R(x)^{-|\alpha |},$$
for all $\alpha \in {\mathbb{N}}^n$.
 \end{itemize} 
\end{dref}

Any finite product $\Big(m_1(x,\xi )\times\ldots\times m_N(x,\xi)\Big)$ of order functions is an order function. 
We notice that $R$, $r$, $\widetilde{r}$ are order functions. 

\begin{dref}[Symbol classes]\label{AP.2} Let $m$ be an order function.
 \begin{itemize}
 \item[\textbf{(a)}] We say that $a\in C^\infty ({\mathbb{R}}^{2n})$ 
is a symbol of order $m$ and write $a\in S(\mathbb{R}^{2n};m)$, 
if for all $\alpha ,\beta \in {\mathbb{N}}^n$,
\begin{equation}\label{dil.7}
\partial _x^\alpha \partial _\xi ^\beta a(x,\xi )={\mathcal{O}}(1)m(x,\xi
)R(x)^{-|\alpha |}\widetilde{r}(x,\xi )^{-|\beta |}\,\, \hbox{ on } \ \mathbb{R}^{2n}.
\end{equation}
 \item[\textbf{(b)}] We write $a\in \dot{S}(\mathbb{R}^{2n},m)$
 when (\ref{dil.7}) holds for all $\alpha$,
$\beta \in {\mathbb{N}}^n$ with $|\alpha |+|\beta |>0$.
\item[\textbf{(c)}] If  $\mathcal{U}$ is open subset of ${\mathbb{R}}^{2n}$ 
we define ${S}(\mathcal{U};m)$ similarly, 
replacing ${\mathbb{R}}^{2n}$ by $\mathcal{U}$.
\end{itemize}
\end{dref}

Sometimes $m,a,R,r$ depend on parameters. We then require
(\ref{dil.1}), (\ref{dil.2}), (\ref{dil.6}), (\ref{dil.7}), to hold
uniformly with respect to the parameters, if nothing else is specified.   

For more details see \cite[Chapter 1, pp.7--15]{HeSj86}.

\subsection{IR-Lagrangian manifolds:}
\label{AP2}
Let $G\in \dot{S}(\mathbb{R}^{2n};R\,\widetilde{r})$ be real-valued. Then the manifold
\begin{equation}\label{dil.8}
\Lambda_G=\Big\{(x,\xi )\in {\mathbb{C}}^{2n};\,\, \Im(x,\xi)=H_G\big(\Re(x,\xi)\big)\Big\}
\end{equation}
is $\mathbf{I}$-Lagrangian, i.e. Lagrangian in ${\mathbb{C}}^{2n}$ for the real
symplectic form $-\Im \sigma $, where $\sigma =\underset{1\leq j\leq n}\sum d\xi _j\wedge
dx_j$ is the complex symplectic form.

\noindent
The one-form $-{\Im(\xi \cdot dx)}_{\big\vert\Lambda _G}$ is closed on $\mathbb{C}^{2n}$ 
and hence exact for topological reasons. 
The primitive $H$ is unique up to a constant and we can choose
\begin{equation}\label{dil.9}
H=-\Re \xi \cdot \Im x +G\big(\Re (x,\xi ))=G(\Re (x,\xi )\big)-
\Re \xi \cdot G'_\xi\big(\Re (x,\xi )\big).
\end{equation}

\par
If we also assume that $G$ is small in $\dot{S}(R\,\widetilde{r})$, 
then $\Lambda _G$ is $\mathbf{R}$-symplectic, 
i.e.\ a symplectic sub-manifold of ${\mathbb{C}}^{2n}$,
equipped with the symplectic form $\Re \sigma $. In other words,
${{\sigma }_\vert}_{\Lambda _G}$ is a (real) symplectic form on
$\Lambda _G$ and we have the volume element
$$
d\alpha =\frac{1}{n!}\big({\sigma^{\wedge n}}\big)_{\big\vert\Lambda _G}.
$$
For more details see \cite[Chapter 2, p.16]{HeSj86}.

\subsection{FBI-transforms and weighted Hilbert spaces:}
\label{AP3}
\hskip-7pt 
Clearly (\ref{dil.8}) gives a parametrization 
$$
{\mathbb{R}}^{2n}\ni \rho \longmapsto \rho +iH_G(\rho )
$$ 
of $\Lambda _G$ and we can then define
symbol spaces $S(\Lambda_G;m)$ of functions on $\Lambda _G$ 
by pulling back functions and weights to ${\mathbb{R}}^{2n}$. 
In particular, we define the scales $R$ and $\widetilde{r}$ by this pull back.

\par\smallskip
Let $\lambda =\lambda (\alpha )\in S(\Lambda _G;R^{-1}\,\widetilde{r})$
be positive, elliptic in the sense that $\lambda$ is
 non-vanishing and $\lambda^ {-1}\in S(\Lambda _G;\,R\,\widetilde{r}^{-1})$ and put
\begin{equation}\label{dil.10}
\phi (\alpha ,y)=(\alpha _x-y)\alpha _\xi +i\frac{\lambda(\alpha )}{2}(\alpha_x-y)^2
\end{equation}
with $\ \alpha =(\alpha _x,\alpha _\xi )\in \Lambda _G$ and $y\in {\mathbb{C}}^n.$

\par\smallskip 
The amplitude will be a ${\mathbb{C}}^{n+1}$-valued smooth function
${\mathbf t}(\alpha ,y;h)$ on $\Lambda_G \times {\mathbb{C}}^n_y$
which is affine linear in $y$. When discussing symbol properties of such
functions we restrict the attention to a region
\begin{equation}\label{dil.11}
|y-\alpha _x|<{\mathcal{O}}(1)R(\alpha _x),
\end{equation}
and with this convention, we require that ${\mathbf t}\in
h^{-\frac{3n}{4}}S(\Lambda_G;  R^{-\frac{n}{4}}\,\widetilde{r}^{\frac{n}{4}})$ and that
${\mathbf t}, \partial _{y_1}{\mathbf t},\ldots,\partial
_{y_n}{\mathbf t}$ are maximally linearly independent in the
sense that with ${\mathbf t}$ treated as a column vector,
\begin{equation}\label{dil.12}
\left|\det \begin{pmatrix}{\mathbf t} &\partial
  _{y_1}{\mathbf t}
  &\ldots &\partial _{y_n}{\mathbf t}\end{pmatrix}\right|\asymp 
  R^{-n}\left( h^{-\frac{3n}{4}}  R^{-\frac{n}{4}}\, \widetilde{r}^{\frac{n}{4}} \right)^{n+1}.
\end{equation}
{ (Vector valued symbols appear naturally after substitution of variables in a Gaussian resolution of the identity, see \cite[Section 4]{HeSj86}.)
}

Notice that the determinant is independent of $y$. If $\mathcal{B}_0$ is the canonical basis in 
$\mathbb{C}^{n+1}$, we can choose, for all $\alpha\in\Lambda_G$ and $y\in\mathbb{C}^n$ satisfying \eqref{dil.11}, 
\begin{equation}\label{dil.12.5}
 {\mathbf t}(\alpha,y;h)=t_0(\alpha;h)\begin{pmatrix}
                           1\\
                           \frac{\alpha_{x_1}-y_1}{R(\alpha_x)}\\
                           \vdots\\
                           \frac{\alpha_{x_n}-y_n}{R(\alpha_x)}\\
                          \end{pmatrix}_{\mathcal{B}_0},
\end{equation}
where $t_0(\bullet;h)\in h^{-\frac{3n}{4}}S(\Lambda_G;  R^{-\frac{n}{4}}\,\widetilde{r}^{\frac{n}{4}})$.
\par\smallskip 
Let $\chi \in C_0^\infty \big(B(0,\frac{1}{C})\big)$ be equal to one in
$B(0,\frac{1}{2C})$, where $C>0$ is large enough. We define the
FBI-transform 
$$
T:{\mathcal D}'({\mathbb{R}}^n)\longrightarrow C^\infty (\Lambda _G;{\mathbb C}^{n+1})
$$
by
\begin{equation}\label{dil.13}
Tu(\alpha ;h)=\int e^{\frac{i}{h}\phi (\alpha
  ,y)}{\mathbf t}(\alpha ,y;h) \chi _\alpha (y)u(y)dy,
\end{equation}
where $\chi _\alpha (y)=\chi \left(\frac{y_1-{\Re \alpha} _{x_1}}{R(\Re \alpha_x)},\ldots,
\frac{y_n-{\Re \alpha} _{x_n}}{R(\Re \alpha_x)}\right)$. 
Here the domain of integration is equal to ${\mathbb{R}}^n$ 
and the integral is defined as the bilinear scalar product of 
$u\in {\mathcal D}'({\mathbb{R}}^n)$ and a test function in
$C_0^\infty ({\mathbb{R}}^n)$.
\par\smallskip 
We assume from now on that $G$ belongs to $S(\mathbb{R}^{2n}; R\,\widetilde{r})$. 
We also assume:
\begin{equation}\label{dil.14}
\left\{\begin{aligned}
&\hbox {There exist } g_0=g_0(x)\in S(\mathbb{R}^{2n};R\, r), \hbox{ such that }\\
&G (x,\xi )-g_0(x) \hbox{ has its support in a region where }\\
&|\xi |\le {\mathcal{O}}(r(x)) \hbox{ and } G(x,\xi )-g_0(x) \hbox{ is sufficiently }\\
&  \hbox{ small in }S(\mathbb{R}^{2n};R\, r).
\end{aligned}\right.
\end{equation}
Notice that the order function $\widetilde{r}$ is controlled by $r$ in the 
region $|\xi |\le {\mathcal{O}}(r(x))$.

\par\smallskip\noindent  
Let $H$ be given in (\ref{dil.9}). Then $H\in S(\Lambda_{G}; R\, \widetilde{r})$. 
Using $T$ we shall define the function spaces $H(\Lambda _G;m)$,
essentially by requiring that
$$
Tu\in L^2(\Lambda _G;m^2e^{-\frac{2}{h}H}d\alpha ).
$$
Here, $m$ is an order function.  

\par\smallskip 
Let $G$ satisfy \eqref{dil.14} and be
sufficiently small in $S(\mathbb{R}^{2n}; R\,\widetilde{r})$, 
or more generally, assume \eqref{dil.14}. 
Define $H$ as in (\ref{dil.9}), let $m$ be an order function on $\Lambda _G$ 
and let $T$ be an associated FBI-transform as in (\ref{dil.13}). 
In \cite[Proposition 4.4]{HeSj86} it is shown that $T$ is injective on
  $C_0^\infty ({\mathbb{R}}^n)$ and also on more general Sobolev spaces with
  exponential weights,  by the construction of
an approximate left inverse of $T$ which works with exponentially
small errors. See \cite[Chapter 4, pp. 20--42]{HeSj86} for accurate results.

\begin{dref}[Sobolev spaces associated to the IR-manifolds]\label{AP.3}
The set $H(\Lambda _G;m)$ is the completion of $C_0^\infty ({\mathbb{R}}^n)$ 
for the norm
\begin{equation}\label{dil.15}
\big\| u\big\|_{H(\Lambda _G;m)}=
\big\| Tu\big\|_{L^2(\Lambda _G;m^2e^{-\frac{2}{h}H}d\alpha )}.
\end{equation}
\end{dref}

\noindent
The following facts were established in \cite[Chapter 5, pp. 43--54]{HeSj86}:
\begin{itemize}
\item $H(\Lambda _G;m)$ is a Hilbert space.
\item If we modify the choice of $\lambda $ and ${\mathbf t}$
  in the definition of $T$, we get the same space $H(\Lambda_G;m)$
  and the new norm is uniformly equivalent to the earlier one, when
  $h$ tends to $0$.
\item If $G_1\le G_2$ and $m_1\ge m_2$, then $H(\Lambda
    _{G_1};m_1)\subset H(\Lambda _{G_2};m_2)$,
and the inclusion map is uniformly bounded.

\item When $G=g(x)$ is independent of $\xi $ and $m=m_0(x)$, we get
$$
H(\Lambda _G;m)=L^2({\mathbb{R}}^n;m_0^2e^{-\frac{2}{h}g(x)}\,dx)
$$
with uniform equivalence of norms. 
More generally, when $G=g(x)$ and 
$m(x,\xi)=m_0(x)\left(\frac{\widetilde{r}(x,\xi )}{r(x)}\right)^{N_0}$
with $N_0\in {\mathbb{R}}$, then $H(\Lambda _G;m)$ is the naturally 
defined exponentially weighted Sobolev space.
\end{itemize}


\subsection{Schr\"odinger operator on $H(\Lambda_G;m)$:}
\label{AP4}
\hskip-10pt
We now consider a Schr\"odinger operator
\begin{equation}\label{dil.155}
P=-h^2\Delta +V(x),\ x\in {\mathbb{R}}^n,
\end{equation}
where $V$ is real-valued and 
\begin{equation}\label{dil.16}
V\in S(\mathbb{R}^n;r^2),
\end{equation}
so that the symbol $p(x,\xi )=\xi ^2+V(x)$ belongs to $S(\mathbb{R}^{2n};\widetilde{r}^2)$.

A basic element in the theory is that if $V$ extends holomorphically to a truncated sector ${\Gamma}_C$
as in \eqref{int.1},
and the extension satisfies $|V(x)|\le {\mathcal{O}}(1)r(\Re x)^2$, 
and if $G$ satisfyies \eqref{dil.14} with $G-g_0=0$ near the analytic singular support of $V$,
then 
$$
P:\, H(\Lambda _G;\widetilde{r}^2)\longrightarrow H(\Lambda _G)
$$ 
can be viewed as an $h$-pseudo-differential operator with leading symbol
${{p}_\vert}_{\Lambda _G}$. For precise results, see
\cite[Chapter 6, Th\'eor\`emes 6.8, 6.8(corrig\'e), and the paragraph in the pages 77--78]{HeSj86}.

\subsection{Dilations:}\label{AP5}\hskip-7pt
Let $\varepsilon_0$ be a positive small constant. For $\varepsilon\in]0,\varepsilon_0]$, 
we introduce the basic scale functions:
\begin{equation}\label{prep.37}
\begin{aligned}
& R_\varepsilon (x)= (\varepsilon +x^2)^{\frac12},  \quad 
 r_\varepsilon (x)=\frac{(\varepsilon+x^2)^{\frac12}}{(1+x^2)^{\frac12}}\\
&  \hbox{ and } \, \widetilde{r}_\varepsilon (x,\xi )=(r_\varepsilon(x)^2+\xi ^2)^{\frac12}.
\end{aligned}
\end{equation}
As above, with these scales, we define the notion of order functions and symbols.
To emphasize the dependence on the parameter $\varepsilon$, we add it in the notations as follows:

\begin{dref}[$\varepsilon$-Order functions]\label{APepsilon.1}
 \begin{itemize}
  \item[{}]
  \item[\textbf{(a)}] We say that $m\in C^\infty ({\mathbb{R}}^{2n};]0,+\infty [)$ is an $\varepsilon$-order
function if
\begin{equation}\label{dil.6.epsilon}
\partial _x^\alpha \partial _\xi ^\beta m(x,\xi )={\mathcal{O}}(1)m(x,\xi
)R_\varepsilon(x)^{-|\alpha |}\widetilde{r}_\varepsilon(x,\xi )^{-|\beta |},
\end{equation}
for all $\alpha ,\beta \in {\mathbb{N}}^n$.
  \item[\textbf{(b)}] A function $m=m(x)\in C^\infty
({\mathbb{R}}^n;]0,+\infty [)$, independent of $\xi $, is an $\varepsilon$-order function
if
$$\partial _x^\alpha m(x)={\mathcal{O}}(1)m(x)R_\varepsilon(x)^{-|\alpha |},$$
for all $\alpha \in {\mathbb{N}}^n$.
 \end{itemize} 
\end{dref}

We require the estimates to be uniform in $\varepsilon$. 
In the special case when $\varepsilon$ is fixed $=1$, we get an order function in
the sense of \cite{HeSj86}. 

\begin{dref}[$\varepsilon$-Symbol classes]\label{APepsilon.2} Let $m$ be an $\varepsilon$-order function.
 \begin{itemize}
 \item[\textbf{(a)}] We say that $a\in C^\infty ({\mathbb{R}}^{2n})$ 
is an $\varepsilon$-symbol of order $m$ and write $a\in S_{\varepsilon}(\mathbb{R}^{2n};m)$, 
if for all $\alpha ,\beta \in {\mathbb{N}}^n$,
\begin{equation}\label{dil.7epsilon}
\hskip-5pt \partial _x^\alpha \partial _\xi ^\beta a(x,\xi )={\mathcal{O}}(1)m(x,\xi
)R_{\varepsilon}(x)^{-|\alpha |}\widetilde{r}_{\varepsilon}(x,\xi )^{-|\beta |}\, \hbox{ on } \mathbb{R}^{2n}.
\end{equation}
 \item[\textbf{(b)}] We write $a\in \dot{S}_{\varepsilon}(\mathbb{R}^{2n},m)$
 when (\ref{dil.7epsilon}) holds for all $\alpha$,
$\beta \in {\mathbb{N}}^n$ with $|\alpha |+|\beta |>0$.
\item[\textbf{(c)}] If  $\mathcal{U}$ is open subset of ${\mathbb{R}}^{2n}$ 
we define ${S}_{\varepsilon}(\mathcal{U};m)$ similarly, 
replacing ${\mathbb{R}}^{2n}$ by $\mathcal{U}$.
\end{itemize}
\end{dref}

Now return to our Schr\"odinger operator $P$ as in Theorem \ref{int1}. 
Let $G=G^\varepsilon$ or $G=G^{\frac{C\varepsilon}{\widetilde{C}}}$
be the local escape functions in Propositions \ref{esc1},
\ref{bp1}, \ref{ltg1}, \ref{ltg2}. At the end  of Section \ref{ltg} we
recalled how to extend $G$ to a global escape function $\widetilde{G}$
satisfying (\ref{ltg.21.5}), (\ref{ltg.22}) away from a fixed neighborhood of $(0,0)$. 
We now drop the tilde and denote by $G$ (or $G^\varepsilon $,
$G^{\frac{C\varepsilon}{\widetilde{C}}}$) this global escape function for
which the above cited propositions hold near $(0,0)$, in addition to
(\ref{ltg.21.5}), (\ref{ltg.22}). We also have $G\in
{S}_\varepsilon (\mathbb{R}^{2n};R_\varepsilon\,\widetilde{r}_\varepsilon )$ 
and we can arrange so that 
$|\xi |\le {\mathcal{O}}(1)r_\varepsilon (x)$ for $(x,\xi)\in \mathrm{supp\,}G$. 

\par\smallskip 
We wish to apply \cite{HeSj86} to the operators $P$, $P_\varepsilon $ 
and $P_\varepsilon^{\mathrm{ext}}$ given by \eqref{prep.1}, \eqref{prep.2} 
and \eqref{pext.4.5} respectively with the scale functions 
$R=R_\varepsilon $, $r=r_\varepsilon $, $\widetilde{r}_\varepsilon $ given by (\ref{prep.37}). 
Notice that these functions satisfy (\ref{dil.1}), (\ref{dil.2}), (\ref{dil.4}). However
(\ref{dil.5}) does not hold near $x=0$ and we only have
\begin{equation}\label{dil.17}
r_\varepsilon (x)\ge \sqrt{\varepsilon },\quad  R_\varepsilon (x) r_\varepsilon (x)\ge\varepsilon.
\end{equation}

\par\smallskip 
In order to remedy for the failure of (\ref{dil.5}), we make the change of variables
$$
x=\mu\,\widetilde{x}\,\,\,\hbox{ with }\,\,\mu =\sqrt{\varepsilon }
$$
and introduce a new semi-classical parameter, by requiring that
\begin{equation}\label{dil.21}
hD_x=\mu \widetilde{h}D_{\widetilde{x}},\quad \hbox{ i.e., }\,\widetilde{h}=\frac{h}{\mu ^2}.
\end{equation}
Assuming from now on that
\begin{equation}\label{dil.22}
\varepsilon \ge h^{\frac12-\alpha _0},\hbox{ for some fixed }\alpha _0>0,
\end{equation}
we see that $\widetilde{h}$ tends to $0$ when $h$ goes to $0$.

\par\smallskip 
The corresponding dilation of a semi-classical operator
$Q=q(x,hD_x)$ is $\widetilde{Q}=q(\mu(\widetilde{x},\widetilde{h}D_{\widetilde{x}}))$, 
whose semi-classical symbol (with respect to $\widetilde{h}$) is given by
\begin{equation}\label{dil.23}
\widetilde{q}(\widetilde{x},\widetilde{\xi })=q(\mu \widetilde{x},\mu\widetilde{\xi }).
\end{equation}

\par\smallskip 
The balls $B(x;R_\varepsilon (x))$ and $B(\xi;\widetilde{r}_\varepsilon (x,\xi ))$ 
become $B\left(\widetilde{x},\frac{1}{\mu} R_\varepsilon(\mu \widetilde{x})\right)$ 
and $B\left(\widetilde{\xi},\frac{1}{\mu}\widetilde{r}_\varepsilon (\mu \widetilde{x},\mu \widetilde{\xi })\right)$ 
respectively in the $\widetilde{x}$, $\widetilde{\xi }$ coordinates.
It is then natural to put:
\begin{equation}\label{dil.24}
R_{\varepsilon,\mu}(\widetilde{x})=\frac{1}{\mu }R_\varepsilon (\mu \widetilde{x}),\quad
r_{\varepsilon,\mu}(\widetilde{x})=\frac{1}{\mu }r_\varepsilon (\mu \widetilde{x}),
\end{equation}
\begin{equation}\label{dil.25}
\widetilde{r}_{\varepsilon,\mu} (\widetilde{x},\widetilde{\xi })
=\frac{1}{\mu}\widetilde{r}_\varepsilon (\mu \widetilde{x},\mu \widetilde{\xi})
=\Big(r_{\varepsilon,\mu} (\widetilde{x})^2+\widetilde{\xi }^2\Big)^{\frac{1}{2}}.
\end{equation}
Recalling that $\mu =\sqrt{\varepsilon }$, we get more explicitly,
\begin{equation}\label{dil.26}
R_{\varepsilon,\mu} (\widetilde{x})=(1+\widetilde{x}^2)^{\frac12},\quad 
r_{\varepsilon,\mu}(\widetilde{x})=
\frac{(1+\widetilde{x}^2)^{\frac12}}{(1+\varepsilon \widetilde{x}^2)^{\frac12}}.
\end{equation}
The rescaled functions satisfy (\ref{dil.1}), (\ref{dil.2}), (\eqref{dil.4}) and (\ref{dil.5}):
\begin{equation}\label{dil.27}
\partial _{\widetilde{x}}^\alpha R_{\varepsilon,\mu} 
={\mathcal{O}}(1) R_{\varepsilon,\mu}^{1-|\alpha |}\quad
\partial _{\widetilde{x}}^\alpha r_{\varepsilon,\mu}  
={\mathcal{O}}(1) r_{\varepsilon,\mu}  R_{\varepsilon,\mu}^{-|\alpha |}.
\end{equation}
\begin{equation}\label{dil.28}
r_{\varepsilon,\mu}(\widetilde x)\ge 1\quad \hbox{ and }\quad r_{\varepsilon,\mu}(\widetilde x)
R_{\varepsilon,\mu}(\widetilde x)\ge 1.
\end{equation}

\par\smallskip
Let $S_{\varepsilon,\mu} (\bullet;m_{\varepsilon,\mu} )$ denote the symbol space $S(\bullet;m)$,
  defined as above, but now with respect to the scales $r_{\varepsilon,\mu}$,
  $R_{\varepsilon,\mu}$, $\widetilde{r}_{\varepsilon,\mu} $ 
  (from now on $\mu $-scales for short), with $m_{\varepsilon,\mu} $ 
  being an order function: $m_{\varepsilon,\mu}\in S_{\varepsilon,\mu} (\bullet;m_{\varepsilon,\mu} )$. 
  
  \par\noindent
  If $m$ is an order function for the scales $r_\varepsilon $, $R_\varepsilon $,
  $\widetilde{r}_\varepsilon $ (from now on $\varepsilon $-scales) and
  $a\in S_\varepsilon (\bullet;m_\varepsilon)$ and if we define
  \begin{equation}\label{dil.28.5}
  a_\mu (\widetilde{x},\widetilde{\xi })=a(\mu (\widetilde{x},
  \widetilde{\xi })),
  \quad m_{\varepsilon,\mu} (\widetilde{x},\widetilde{\xi })=m_\varepsilon(\mu
  (\widetilde{x},\widetilde{\xi} )),
 \end{equation}
then $m_\mu $ is an order function for the $\mu$-scales and $a_{\varepsilon,\mu} \in S_{\varepsilon,\mu} (m_\mu )$.

\par\smallskip
 Consider an FBI-transform $T$ as in (\ref{dil.10})--(\ref{dil.13})
  and let $G$, $g_0$ be as in the corresponding discussion so that
  (\ref{dil.14}) holds with respect to the $\varepsilon$-scales. In
  (\ref{dil.13}) we make the change of variables,
  $$
  \alpha =\mu\, \widetilde{\alpha },\quad 
  y=\mu\, \widetilde{y}
  $$ 
  and define $\widetilde{u}(\widetilde{y})$ by 
  \begin{equation}\label{dil.29}
    u(y)=\mu^{\frac{n}{2}}\widetilde{u}(\widetilde{y}), 
 \end{equation} 
 so that the map 
 \begin{equation*}
   \begin{matrix}
  \left\{
  \begin{aligned}
   L^2(\mathbb{C}^{n};dy) & \longrightarrow L^2(\mathbb{C}^{n};d\widetilde{y})\\
     u(\bullet)              & \longmapsto \widetilde{u}(\widetilde\bullet)
  \end{aligned}\right.
 \end{matrix}
 \qquad
 \hbox{is unitary}.
 \end{equation*}
We have
  $\displaystyle\frac{1}{h}\,\phi (\alpha ,y)=\displaystyle\frac{1}{\widetilde{h}}\phi_\mu(\widetilde{\alpha},\widetilde{y})$, 
  where
  $$ 
  \phi _\mu (\widetilde{\alpha },\widetilde{y})=\frac{1}{\mu ^2}\,
  \phi (\mu\widetilde{\alpha },\mu \widetilde{y}) 
  =(\widetilde{\alpha}_x-\widetilde{y})\cdot \widetilde{\alpha }_\xi 
  +i\frac{\widetilde{\lambda}(\widetilde{\alpha })}{2}\,(\widetilde{\alpha }_x-\widetilde{y})^2, 
  $$
  and
  $\widetilde{\lambda }(\widetilde{\alpha })=\lambda(\frac{\widetilde{\alpha }}{\mu} )\in 
  S_{\varepsilon,\mu}({\Lambda}_{G_\mu}; R_{\varepsilon,\mu}\,\widetilde{r}_{\varepsilon,\mu})$ is
  elliptic and positive. Here ${\Lambda}_{G_\mu}$ is the $\mathbf{I}$-Lagrangian manifold associated to
  the escape function $G_\mu$ (in the $\mu$-scales) given by \eqref{dil.35.5} below. We have 
\begin{equation}\label{dil.30}
    Tu(\alpha ;h)=T_\mu \widetilde{u}(\widetilde{\alpha};\widetilde{h}), 
\end{equation}
where 
\begin{equation}\label{dil.31} 
T_\mu\widetilde{u}(\widetilde{\alpha } ; \widetilde{ h})
=\int e^{\frac{i}{\widetilde{ h}}\widetilde{\phi} (\alpha ,y)}{\bf
      t}_\mu (\widetilde{\alpha } , \widetilde{y}; \widetilde{h}) \chi
    _{\widetilde{\alpha}} (\widetilde{y})
    \widetilde{u}(\widetilde{y})d
    \widetilde{y}, 
    \end{equation} 
    \begin{equation}\label{dil.32}
    \widetilde{\chi }_{\widetilde{\alpha }}(\widetilde{y})=\chi_\alpha (y), 
    \end{equation} and 
    \begin{equation}\label{dil.33}
    {\mathbf t}_\mu (\widetilde{\alpha },\widetilde{y};\widetilde{h})=\mu
    ^{\frac{3n}{2}}{\mathbf t}(\alpha ,y;h)=\mu^{\frac{3n}{2}}{\mathbf t}(\mu
    \widetilde{\alpha } ,\mu
    \widetilde{y};\mu^2\widetilde{h}).  
    \end{equation}
  The function $\widetilde{\chi }$ has the same cut-off properties in the
  $\mu $-scales as $\chi _\alpha $ in the $\varepsilon
  $-scales. Moreover, ${\mathbf t}_\mu $ is affine linear in
  $\widetilde{y}$ and 
  \begin{equation}\label{dil.34} 
  {\mathbf t}_\mu \in\widetilde{h}^{-\frac{3n}{4}}S_{\varepsilon,\mu} 
  ({\Lambda}_{G_\mu};R_{\varepsilon,\mu}^{-\frac{n}{4}}\,\widetilde{r}_{\varepsilon,\mu}^{\frac{n}{4}}).  
  \end{equation} 
  From (\ref{dil.12}) we get by straight forward
  calculation, 
  \begin{equation}\label{dil.35}
    \left|\det \begin{pmatrix}
    {\mathbf t}_\mu &\partial_{\widetilde{y}_1}{\mathbf t}_\mu &\ldots &\partial
        _{\widetilde{y}_n}{\mathbf t}_\mu 
        \end{pmatrix}\right|
        \asymp R_{\varepsilon,\mu}^{-n}\left(\widetilde{h}^{-\frac{3n}{4}}
        R_{\varepsilon,\mu}^{-\frac{n}{4}}\widetilde{r}_{\varepsilon,\mu}^{\frac{n}{4}}
    \right)^{n+1}, 
    \end{equation} 
    which is analogous to (\ref{dil.12}).

\par\medskip 
  Let $G\in S(\mathbb{R}^{2n}; R_\varepsilon\,\widetilde{r}_\varepsilon)$ be
  real-valued. Define $G_{\mu}$ by
  \begin{equation}\label{dil.35.5}
 \frac{1}{h} G(\alpha ) = \frac{1}{\widetilde{h}} G_{\mu}(\widetilde{\alpha }),
  \hbox{ i.e. } G_{\mu} (\widetilde{\alpha })=\frac{1}{\mu^2}G(\frac{\widetilde{\alpha}}{\mu}).    
  \end{equation}
  Then
  $$ 
  G_\mu \in \mu ^{-2}{S}_{\varepsilon,\mu} \left(\mathbb{R}^{2n};
  (\mu R_{\varepsilon,\mu})\,(\mu \widetilde{r}_{\varepsilon,\mu}) \right)
    ={S}_{\varepsilon,\mu} (\mathbb{R}^{2n}; R_{\varepsilon,\mu}\,\widetilde{r}_{\varepsilon,\mu} ).  
  $$ 
  If (\ref{dil.14}) holds for the $\varepsilon $-scales,
  then it also holds for $G_\mu $, $g_{0,\mu }$ for the $\mu $-scales
  with
  $g_{0,\mu }(\widetilde{\alpha }_x)=\mu^{-2}g_0(\mu \widetilde{\alpha }_x)$.  
  
  \par\smallskip 
  With $H$ as in (\ref{dil.9}) we get by straight forward calculation,
  \begin{equation}\label{dil.35.55}
  H(\alpha )=\mu^2H_\mu (\widetilde{\alpha }), 
  \end{equation}
  where
  $$ 
  H_\mu (\widetilde{\alpha })=G_\mu (\Re \widetilde{\alpha })-\Re
  \widetilde{\alpha }_\xi \cdot \partial _{\widetilde{\xi }}G_\mu (\Re
  \widetilde{\alpha }).  
  $$
Since the weights $R_{\varepsilon,\mu}$, $r_{\varepsilon,\mu}$, $\widetilde{r}_{\varepsilon,\mu}$ 
   satisfy (\ref{dil.1}), (\ref{dil.2}), (\ref{dil.4}) and (\ref{dil.5}) we can define the Sobolev spaces 
   ${\Lambda}_{G_\mu}$,$H(\Lambda_{G_\mu};m_{\varepsilon,\mu})$,
   associated to the IR-manifolds as in Definition \ref{AP.3}. 
   In view of (\ref{dil.30}) this allows us to define the spaces
  $H(\Lambda _G;m)$ for the scales $r_\varepsilon $, $R_\varepsilon $,
  $\widetilde{r}_\varepsilon $: 
  
  \par 
  We say that $u\in H(\Lambda _{G};m_\varepsilon)$ if $\widetilde{u}\in H(\Lambda _{G_\mu };m_{\varepsilon,\mu} )$ with
  $m\longleftrightarrow m_{\varepsilon,\mu}$, $u\longleftrightarrow \widetilde{u} $
  and $G\longleftrightarrow G_\mu $ related as above, see \eqref{dil.28.5}, \eqref{dil.29} and \eqref{dil.35.5}.
 
\par\noindent
For $u\in H(\Lambda _G;m_\varepsilon)$ we choose the norm
$$
\big\| u\big\|_{H(\Lambda _G;m_\varepsilon)}=
\big\| Tu\big\|_{L^2(\Lambda _G;m_\varepsilon^2 e^{-\frac{2}{h}H}\,d\alpha )}.
$$
Defining similarly
$$
\big\| \widetilde{u}\big\|_{H(\Lambda _{G_\mu };m_{\varepsilon,\mu})}=
\big\| T_\mu \widetilde{u}\big\|_{L^2(\Lambda _{G_\mu };m_{\varepsilon,\mu}^2 
e^{-\frac{2}{\tilde h}H_\mu}\, d\widetilde{\alpha })},
$$
we find,
\begin{equation}\label{dil.36}
\big\| u\big\|_{H(\Lambda _G;m_\varepsilon)}=
\mu^n\big\|\widetilde{u}\big\|_{H(\Lambda_{G_\mu};m_{\varepsilon,\mu})},
\end{equation}
since $d\alpha =\mu ^{2n}d\widetilde{\alpha }$ 
and the relationships between the different involved quantities (See 
\eqref{dil.28.5}, \eqref{dil.29}, \eqref{dil.35.5} and \eqref{dil.35.55}). 

\par\smallskip 
In conclusion, the changes of variables above allow us to replace the
scales $r_\varepsilon $, $R_\varepsilon  $, $\widetilde{r}_\varepsilon$ 
that do not satisfy (\ref{dil.5}) by the
scales $r_{\varepsilon,\mu}$, 
$R_{\varepsilon,\mu}$, $\widetilde{r}_{\varepsilon,\mu}$ that do so, 
and we can then apply the theory of \cite{HeSj86}.

\section{Volume functions}\label{app}
\setcounter{equation}{0}

Recall that $\omega (E)$ is defined by \eqref{int.13} or
\eqref{int.14} when $E\leq 0$.  The saddle point being $x=0$, we
choose coordinates as in the beginning of Section \ref{esc}.  We
extend $\omega (E)$ to small positive values of $E$ by replacing
${\mathscr{U}}_0$ in \eqref{int.13} or \eqref{int.14} with
$\widetilde {\mathscr{U}}_0:=\{x\in\mathrm{neigh}(\overline{\mathscr{U}}_{0},{\mathbb{R}}^n);\, x_n\le 0\ 
\hbox{when}\ x\in \mathrm{neigh\,}(0,{\mathbb{R}}^n)\}$.  
Then \eqref{int.15} remains valid again with ${\mathscr{U}}_0$ 
replaced by $\widetilde {\mathscr{U}}_0$ and this
gives a nice $C^1$-extension of $\omega (E)$ to
$E\in \mathrm{neigh}(0,{\mathbb{R}})$. 
The volume $\omega_{\varepsilon}(E)$
in (\ref{pint.5}), can be written
\begin{equation}\label{app.1}
 \omega_{\varepsilon}(E)=\int_{\widetilde {\mathscr{U}}_0}
 \left(\int_{p_{\varepsilon}(x,\xi)\leq E}\,d\xi\right)\, dx.
\end{equation}
Recall that $p_\varepsilon$ is given by \eqref{bp.8}, \eqref{bp.3}. 
Assume for simplicity that 
$\chi (x,0)\geq \chi (x,\xi)$ for all $\xi\in{\mathbb{R}}^n$. Then, since 
$$
p_\varepsilon(x,\xi)=\xi^2+V(x)+\chi_\varepsilon (x,\xi),
$$
where $\chi_\varepsilon$ is given by \eqref{bp.3}, we have
\begin{equation}\label{app.2}
 \xi^2+V(x)\leq p_\varepsilon(x,\xi)\leq \xi^2+V(x)+\chi_\varepsilon (x,0)
\end{equation}
and it follows that 
\begin{equation}\label{app.3}
C_n\int_{\widetilde {\mathscr{U}}_0}\big(E-(V(x)+\chi_\varepsilon (x,0))\big)_{+}^{\frac{n}{2}}\, dx
\leq \omega_\varepsilon(E) 
\leq  C_n\int_{\widetilde {\mathscr{U}}_0}\big(E-V(x)\big)_{+}^{\frac{n}{2}}\, dx\,,
\end{equation}
where $\displaystyle C_n=\mathrm{vol}\big(B_{{\mathbb{R}}^n}(0,1)\big)
=\frac{\pi^{\frac{n}{2}}}{\Gamma(\frac{n}{2}+1)}$, and the last member is $\omega(E)$
by \eqref{int.14}.
Using that $[0,1]\ni t\longmapsto \big(E-(V(x)+t\chi_\varepsilon (x,0))\big)_{+}^{\frac{n}{2}}$ 
is a convex function, we get
$$
\big(E-V(x)\big)_{+}^{\frac{n}{2}}-\big(E-(V(x)+\chi_\varepsilon
(x,0))\big)_{+}^{\frac{n}{2}}\leq  
\frac{n}{2}\big(E-V(x)\big)_{+}^{\frac{n}{2}-1}\chi_\varepsilon (x,0)
$$
so by \eqref{app.3}:
\begin{equation}\label{app.4}
\begin{aligned}
 0\leq \omega(E)-\omega_\varepsilon(E) & \leq 
 \frac{\pi^{\frac{n}{2}}}{\Gamma(\frac{n}{2})}  
 \int_{\widetilde {\mathscr{U}}_0} \big(E-V(x)\big)_{+}^{\frac{n}{2}-1}\chi_\varepsilon (x,0)\, dx\\
{} & \leq  \mathcal{O}(1)\int_{\widetilde {\mathscr{U}}_0} \chi_\varepsilon (x,0)\, dx\\                                                   
 {} & \leq  \mathcal{O}(\varepsilon) \varepsilon^{\frac{n}{2}}\\
{} & \leq  \mathcal{O}(1)\varepsilon^{2}\,. 
\end{aligned}
\end{equation}

\end{appendix}

%

\end{document}